\documentclass[11pt]{amsart}
\usepackage{amsmath,amsfonts,amsthm,amssymb,amscd,color,xcolor,mathrsfs}
\usepackage{amsfonts}
\usepackage{amsmath,amsthm}
\usepackage{hyperref}
\usepackage{latexsym}
\usepackage{array}
\usepackage{amssymb}
\usepackage{enumerate}

\usepackage[francais,english]{babel}

\usepackage{color}
\usepackage[latin1]{inputenc}

\usepackage[margin=1.2in]{geometry}

\numberwithin{equation}{section}

\usepackage{amsfonts}
\usepackage{latexsym}
\usepackage{amsmath,amsthm,amssymb, color, amscd,marginnote}
\usepackage[all]{xy}
 \usepackage{mathrsfs}
\usepackage{MnSymbol}
\usepackage{marginnote}
\usepackage{hyperref}
\hypersetup{colorlinks, linkcolor=blue,citecolor=blue}

\usepackage{graphicx,psfrag,epsfig}

\theoremstyle{plain}
\newtheorem{theorem}{Theorem}

\newtheorem{lemma}[theorem]{Lemma}
\newtheorem{corollary}[theorem]{Corollary}
\newtheorem{amplification}[theorem]{Amplification}
\newtheorem{proposition}[theorem]{Proposition}
\theoremstyle{definition}
\newtheorem{definition}[theorem]{Definition}
\newtheorem{remark}[theorem]{Remark}
\newtheorem{assumption}[theorem]{Assumption}
\newtheorem{example}[theorem]{Example}

\def\eps{\varepsilon}
\def\ep{\varepsilon}

\def\lan{\langle}
\def\ran{\rangle}

\def\bdef{\begin{definition}}
\def\endef{\end{definition}}
\def\bthm{\begin{theorem}}
\def\ethm{\end{theorem}}
\def\blm{\begin{lemma}}
\def\elm{\end{lemma}}
\def\brm{\begin{remark}}
\def\erm{\end{remark}}
\def\bprop{\begin{proposition}}
\def\eprop{\end{proposition}}
\def\bcor{\begin{corollary}}
\def\ecor{\end{corollary}}
\def\be{\begin{eqnarray}}
\def\ee{\end{eqnarray}}
\def\beal{\begin{aligned}}
\def\enal{\end{aligned}}

\def\om{\omega}

\def\eps{\varepsilon}
\def\grr{\tau}

\def\phi{\varphi}
\def\f{\varphi}

\def\R{\mathbb R}
\def\C{\mathbb C}
\def\Nn{\mathbb N}
\def\T{\mathbb T}

\def\Z{\mathbb Z}

\def\M{\mathcal M}

\def\cC{\mathcal C}
\def\cP{\mathcal P}

\def\cF{\mathcal F}
\def\~{\tilde}

\def\cE{\mathcal E}

\def\cH{\mathcal H}
\def\cA{\mathcal A}
\def\cB{\mathcal B}
\def\cC{\mathcal C}

\def\CC{{\mathcal C}^{m_0}(P)}

\def\cD{\mathcal D}

\newcommand{\Tr}{\operatorname{Trace}}

\def\PP{\mathbf{P}}
\def\EE{\mathbf{E}}

\def\p{\partial}
 \newcommand{\strela}{\rightharpoonup}
 \def\Ye{ Y}
  \def\BR{\bar B_R}
  \def\BM{\bar B_M}

\def\llan{\langle\!\langle}
\def\rran{\rangle\!\rangle}

\def\be{\begin{equation}}
\def\ee{\end{equation}}
\def\bdef{\begin{definition}}
\def\endef{\end{definition}}
\def\blm{\begin{lemma}}
\def\elm{\end{lemma}}
\def\beal{\begin{aligned}}
\def\enal{\end{aligned}}

\newtheorem*{Pf}{Proof}
\renewenvironment{proof}{\begin{Pf} \begin{upshape}} {\end{upshape} \qed\end{Pf}}

\numberwithin{equation}{section}
\numberwithin{theorem}{section}

\title[ Averaging and mixing for stochastic perturbations]
{Averaging and mixing  for stochastic  perturbations of linear conservative systems}

\author{Guan Huang}
\address{School of Mathematics and Statistics, Beijing Institute of Technology, Beijing, China,\& Peoples' Friendship University of Russia (RUDN University), Moscow, Russia}
\email{huangguan@tsinghua.edu.cn}	

\author{Sergei Kuksin}
\address{Universit\'e Paris-Diderot (Paris 7), UFR de Math\'ematiques - Batiment Sophie Germain, 5 rue Thomas Mann, 75205 Paris,
\& Peoples' Friendship University of Russia (RUDN University), Moscow, Russia}
\email{ sergei.kuksin@imj-prg.fr}

\begin{document}

\maketitle

\hfill{\it Dedicated to the memory of M.I.Vishik } 

\hfill{\it  on the occasion of his 100-th birthday.} \medskip

%Dedicated to 

\begin{abstract}
We study stochastic perturbations of linear systems of the form 
$$
dv(t)+Av(t)\,dt={\ep} P(v(t))dt+\sqrt{{\ep}}\, \cB(v(t)) d W (t), \quad v\in\R^{D},  \eqno{ (*)}
$$
where $A$ is a linear operator with non-zero imaginary spectrum. It is  assumed that the vector field $P(v)$ 
and the matrix-function $\cB(v)$ are locally Lipschitz with at most a polynomial growth at infinity, that the equation
is well posed and first few moments of norms of solutions $v(t)$ are bounded uniformly in $\eps$. We use the  Khasminski 
approach to stochastic averaging to show that as $\eps\to0$, a solution $v(t)$, written in  the interaction representation in
 terms of operator~$A$, for $0\le t \le\,$Const$\,\eps^{-1}$ converges in distribution to a solution of an effective equation. 
 The latter  is obtained from $(*)$ by means of certain averaging. Assuming that eq.~$(*)$ and/or the effective
 equation are mixing, we examine this convergence further. % Our work may serve as an introduction to the stochastic averaging.
 \end{abstract}

\date{}

	\maketitle
	%\centerline { \today}
%	
	\tableofcontents

%\newpage	

\section{Introduction}

\subsection{The setting and problems.}

The goal of this paper is to present an averaging theory for   perturbations of  conservative  linear differential  equations
 by locally Lipschitz nonlinearities and stochastic terms. % based on the  Khasminski approach to stochastic averaging. 
Namely, we will examine  stochastic  equations
\begin{equation}\label{1}
%\frac{\partial v}{\partial t}
dv(t)+Av(t)\,dt={\ep} P(v(t))dt+\sqrt{{\ep}}\, \cB(v(t)) d W (t), \qquad v\in\R^{D}, 
\end{equation}
where $0<\eps \le1$, 
 $A$ is a  linear operator with non-zero pure imaginary eigenvalues $\{ i\lambda_j\}$ 
 (so the dimension $D$ is even),  $P$ is a locally  Lipschitz  vector field on $\R^D$, 
  $W(t)$ is the standard  Wiener process in $\mathbb{R}^{N}$ and $\cB(v)$ is an $D\times N$-matrix.  We wish 
  to study for small $\eps$ the behaviour of solutions for eq.~\eqref{1} on time-intervals of order $\eps^{-1}$, and under some 
  additional restriction on the equation examine the limiting as $\eps\to0$ behaviour of 
  solutions uniformly in time.

\subsection{Our results and their deterministic analogies.} 
 We tried to make  our work ``reader-friendly" and    accessible  to people  with just a limited knowledge of 
stochastic calculus.  To achieve that in the main part of the paper we restrict ourselves to 
 the case of equations with  additive noise $\sqrt\eps\, \cB\, dW(t)$ and exploit there  a technical convenience: 
 we introduce in $\R^D$ a complex structure, re-writing the phase-space  $\R^D$  as $\C^{D/2}$
(recall that $D$ is even), in such a way that in the corresponding complex coordinates the operator $A$ is diagonal,
$
A = $diag$\{i\lambda_j\}.
$
General equations \eqref{1} are discussed in Section~\ref{s-ef-eq}, where they  are treated  in parallel with earlier considered   
equations with  additive noise. 

As it is custom in the classical  deterministic   Krylov--Bogolyubov (K-B) averaging  (e.g. see \cite{BM}, 
 \cite{AKN} and  \cite{JKW}), 
 to study solutions $v(t)\in \C^{D/2}$ we write them in  the 
  interaction representation which preserves the norms of the complex components $v_j(\tau)$, but 
amends their angles. See below  substitution \eqref{interac}. 
The first  principal result of the work is  given by Theorem~\ref{average-thm}, where we assume 
uniform in $\eps$ and in $t\le C \eps^{-1}$ bounds on a first few moments of 
 norms of solutions. The  theorem states that  as $\eps\to0$,  for  $t\le C\eps^{-1}$
  solutions $v(t)$, written in terms of  the interaction representation,  weakly converge in distribution to solutions 
  of an additional {\it effective equation}. The latter is obtained from eq.~\eqref{1} by means of 
certain averaging of vector field $P$  in terms of the spectrum $\{i\lambda_j\}$  and in many cases may be written down explicitly. 
Proof of Theorem~\ref{average-thm}, given in Section~\ref{s_4}, is obtained by a synthesis of the K-B method (as it is 
presented e.g. in \cite{JKW}) and the Khasminski approach to the stochastic averaging \cite{Khas68}; it  may serve as an 
introduction to the latter. The number of works 
 on the  stochastic averaging is immense, see Section~\ref{s_1.3} for some references. 
We were not able to find there the result of Theorem~\ref{average-thm}, but do not insist on its novelty (and certainly
related statements may be found in the  literature). 

 In Section~\ref{s_5} we suppose that the bounds on the moments of   norms of solutions, mentioned above, are uniform in time, and that 
 eq.~\eqref{1} is mixing. So its solutions, as time goes to infinity, 
 converge in distribution to a unique stationary measure (which is a Borel measure in $\R^D = \C^{D/2}$). 
  In Theorem~\ref{thm-limit-mixing}, postulating  that the effective equation 
 as well is mixing, we prove that  when $\eps\to0$ the stationary measure for eq.~\eqref{1} converges to that for the effective equation.
 Note that this convergence holds  without  passing to  the interaction representation.

 In short Section \ref{ss_51} we discuss non-resonant systems \eqref{1} (where the frequencies $\{\lambda_j\}$ are 
 rationally independent). In particular, we show that then the actions $I_j(v(t))$ of solutions $v(t)$ (see below
 \eqref{Iphi}), as $\eps\to0$,  converge in distribution to solutions of a system of stochastic equations, depending only on
 actions. The convergence hold on time-intervals $0\le t\le C\eps^{-1}$. 
 
  In Section~\ref{s_8} we keep the assumption on the norms of solutions from  Section~\ref{s_5}. Supposing  that the 
  effective equation is mixing (and without assuming that for the original equation   \eqref{1}) 
  we prove there Theorem~\ref{thm-uniform}. It states 
   that the convergence as in the principal Theorem~\ref{average-thm} is uniform for $t\ge0$   (and not only for 
  $t\le C\eps^{-1}$).

  In Proposition~\ref{p_6.4} we present a simple sufficient condition on eq.~\eqref{1}, based on results from \cite{khasminskii}, 
   which insures that  Theorems~\ref{average-thm}, \ref{thm-limit-mixing} and \ref{thm-uniform}  apply to it.
 
 In Section \ref{s-ef-eq} we pass to general equations \eqref{1}, where the dispersion matrix $\cB$  depends on $v$. 
 Assuming the same estimates on solutions as in Section~\ref{s_4} we show that   Theorem~\ref{average-thm}  remains valid
 if either  the matrix $\cB(v)$ is non-degenerate,  or it is  a $C^2$-smooth function of $v$. 
  Theorems~\ref{thm-limit-mixing} and \ref{thm-uniform}  also remains true for general  systems \eqref{1}, but we do not discuss this,
 hoping that the corresponding modifications of the proofs should be clear after reading Section~\ref{s-ef-eq}.
 \smallskip

 A deterministic analogy of our results, which deals  with eq.~\eqref{1} with $W=0$ and describes  the behaviour of its solutions 
 on time-intervals of order  $\eps^{-1}$ in  the interaction representation in comparison 
with solutions for a corresponding effective equation,   is given by the K-B
 averaging (see \cite{BM, AKN, JKW}).\footnote{Theorem~\ref{average-thm} 
also  applies  to such equations, but  then  its assertion becomes unnatural.}
  Theorem~\ref{thm-limit-mixing} has no analogy for deterministic systems, but Theorem~\ref{thm-uniform} has it. Namely, for the K-B averaging it is known that if the effective equation has 
a globally asymptotically stable equilibrium, then the convergence of  solutions for eq.~\eqref{1}$_{W=0}$, 
 written in  the interaction representation, to solutions of the effective equation, is uniform in time. 
 This result is known in  folklore as the second K-B theorem and may be found in \cite{Burd}.

The  K-B method and Khasminski approach to averaging which we exploit, are flexible tools. They are 
applicable    to various stochastic systems 
  in finite and infinite dimensions, including stochastic PDEs, and the specific realization of the two methods which we use now  is inspired
  by our previous work on averaging for stochastic PDEs.  See \cite{HKM, KM} for an analogy for SPDEs of  Theorem~\ref{average-thm}, 
   \cite{HKM} -- for an analogy of Theorem~\ref{thm-limit-mixing} and \cite{HGK22} -- for an analogy of Theorem~\ref{thm-uniform} (also see \cite{DW} for more results and references on averaging for SPDEs).

  %as is shown in our papers, mentioned above). 

  \subsection{Relation with  classical  stochastic averaging.} \label{s_1.3}
   The averaging in stochastic systems is a well developed topic, usually dealing 
  with fast-slow stochastic   systems,  e.g. see  publications     \cite{Khas68}, 
   \cite[Section~7]{FW}, \cite[Section~II.3]{Skor}, \cite{Krs},    \cite{Ver}, \cite{Kif} and references  there.
   To explain relation of that  theory with our work let us write eq.~\eqref{1} in the complex form $v(t) \in \C^n$, $n=D/2$ (when the operator $A$ is diagonal)
   and then pass  to  the slow time $\tau = \eps t$ and    action-angle coordinates
$(I, \varphi) = (I_1, \dots, I_n; \f_1, \dots, \f_n)
\in \R^n_+\times \T^n$, where $\R_+= \{x\in \R: x\ge0\}$, $\T^n = \R^n/ (2\pi \Z^n)$ and 
\be\label{Iphi}
 \begin{split}
I_k(v)=\tfrac{1}{2}|v_k|^2=\tfrac{1}{2}v_k\bar v_k, \quad
 \varphi_k(v)=\text{Arg} \;v_k\in \mathbb{S}^1= \R/ 2\pi\Z,
 \quad k=1, \dots, n
\end{split}
 \ee
(if $v_k=0$, we set $\varphi_k(v)=0\in \mathbb{S}^1$).  In these  coordinates  equation \eqref{1} takes the form
\be\label{ac_an}
\begin{cases}
dI (\tau)=P ^I(I,\varphi)d\tau+ \Psi^I(I,\varphi) d\beta(\tau),\\
d\varphi (\tau) +{\ep}^{-1}\Lambda d\tau=P ^\varphi(I,\varphi)d\tau+
 \Psi^\varphi(I,\varphi) d\beta (\tau).\end{cases} %k=1,\dots,n.
\ee
Here $\beta=(\beta_1, \dots, \beta_N)$, where   $\{\beta_l\}$ are independent standard real Wiener processes, and 
 the coefficients of the system  are given by  It\^o's formula. 
This is a fast-slow system with the slow variable $I$ and  fast variable $\varphi$. 
 The stochastic averaging treats systems like \eqref{ac_an}, usually adding to the fast part of the 
  $\phi$-equation a non-degenerate stochastic  term of order
 $\eps^{-1/2}$. The (first) goal of a system's analysis usually 
  is to prove that on time-intervals $0\le\tau\le T$ distributions  of the $I$-components of 
 solutions converge as $\eps\to0$ to distributions of solutions for a suitably averaged $I$-equation. After that other goals may be 
 pursued.\footnote{E.g. one may 
 study deviation of the $I$-components of solutions from the averaged dynamics (see \cite{FW, Kif}) or, under stronger restrictions on the system, 
 may examine behaviour of solutions on longer time intervals (see paper \cite{Khas66} and works, descending from it). }
 
 Unfortunately  the stochastic averaging does not apply directly to systems \eqref{ac_an}, coming from equations 
 \eqref{1}, since then the coefficients of the 
 $\phi$-equation have singularities when some  $I_k$ vanishes, and since  the fast 
  $\phi$-equation is  rather degenerate if the vector $\Lambda$ is 
 resonant. Instead we borrow from the theory  the Khasminski method   \cite{Khas68} of stochastic averaging  and apply it to 
eq.~\eqref{1}, written in  the interaction representation,   thus arriving at the assertion of  Theorem~\ref{average-thm}.
 Averaging theorem for stationary solutions of eq.~\eqref{ac_an} and for the corresponding stationary measures are known in the stochastic averaging, but 
 (naturally) they control the limiting behaviour  only of the $I$-components of the stationary solutions and measures, while our 
 Theorem~\ref{thm-limit-mixing} describes a limit of the whole stationary measure. It seems that no analogy of 
 Theorem~\ref{thm-uniform} is   known in the stochastic averaging. 
 
 At the origin of this paper are lecture notes for an online course which SK was teaching in the Shandong University (PRC) in the 
 autumn term of the year 2020.

\bigskip
\noindent{\it Notation.}
 For a Banach space $E$ and $R>0$, $B_R(E)$ stands for the open $R$-ball $\{ e\in E: | e|_E < R\}$, and $\bar B_R(E)$ -- 
 for its  closure $ \{| e|_E \le R\}$;
% $\mathcal{P}(E)$ stands for the space of  probability Borel measures on $E$,
  $C_b(E)$ stands for  the space of bounded continuous function on  $E$, and $C([0,T], E)$ -- for
   the space of continuous curves
$[0,T] \to E$, given the sup-norm. For  any $0<\alpha\leqslant1$ and  $u\in C([0,T], E)$, 
% $\|u\|_{\alpha}$ is the $\alpha$-H\"older-norm of  $u:$
\be\label{Hold}
\|u\|_\alpha=\sup_{0\leqslant  \tau <\tau'\leqslant T}\frac{|u(\tau')-u(\tau)|_E}{|\tau'-\tau|^\alpha}+\sup_{\tau\in[0,T]}|u(\tau)|_E \le \infty. 
\ee
This is a norm in the H\"older space $C^\alpha([0,T], E)$.  We denote the standard $C^m$-norm for $C^m$-smooth functions on $E$ as $|\cdot|_{C^m(E)}$. 
  By  $\mathcal{D}(\xi)$ we denote the  law of a random variable $\xi$,  by $\strela$ -- the weak convergence of measures, and by $\cP(M)$ -- the
  space of Borel measures on a metric space $M$.  For a measurable mapping $F: M_1\to M_2$ and $\mu\in \cP(M_1)$ 
  we denote by $F\circ\mu\in \cP(M_2)$ the image of $\mu$ under $F$; i.e. $F\circ\mu(Q) = \mu(F^{-1}(Q))$. 
  
   If $m\ge0$ and 
    $L$ is $\R^n$ or $\C^n$,  then $\text{Lip}_m(L, E)$ is the collection of maps $F:L \to E$ such
 that for any $R\geqslant1$ we have
\be\label{nota}
 (1+ |R|)^{-m}\big(\text{Lip}(F|_{\BR(L)})+\sup_{v\in \BR(L)}|F(v)|_E\big)  =: \cC^m(F) <\infty,
\ee
where $\text{Lip}(f)$ is the Lipschitz constant of a mapping $f$  (note that, in particular,
  $
  |F(v)|_E \le \cC^m(F) (1+ |v|_L)^m\,
  $
  for any $v\in L$).
For a complex matrix $A=(A_{ij})$, $A^*= (A^*_{ji})$ stands for its Hermitian conjugate:
$
A^*_{ij} = \bar A_{ji}
$
(so for a real matrix $B$, $B^*$ is  the transposed matrix).  For a set $Q$  we denote 
 by $\mathbf{1}_Q$  its  indicator  function, and by $Q^c$ -- its complement.  
 Finally, $\R_+$ ($\Z_+$)  stands for the set of nonnegative real numbers (nonnegative integers),
  and for real numbers $a$ and $b$, $a\vee b$ and $a\wedge b$ indicate their  maximum and minimum.

\section{Linear systems and their perturbations}\label{s_1.1}
In this section we give the  setting of the problem and  specify  assumptions on the operator $A$, vector field $P$ and noise 
$\sqrt\eps\,\cB(v)d{W}$ in eq.~\eqref{1}.
To 
simplify presentation and explain better the  ideas, in the main part of the text we assume that the noise 
 is additive, i.e. $\cB$ is a constant matrix (possibly degenerate).  We will  discuss 
general equations \eqref{1} in  Section~\ref{s-ef-eq}.

\subsection{Assumptions on $A$ and  ${W}(t)$}
We assume that the unperturbed  linear system
\begin{equation}\label{linear-s1}
%\frac{d{v}}{d t}
(d/dt)v +Av=0, \qquad v\in \R^D,
\end{equation}
is such that all its trajectories
are bounded as $t\to\pm\infty$. Then 
 the eigenvalues of $A$ are pure imaginary,  go in pairs $\pm i\lambda_j$ and $A$ has no Jordan cells. We also assume that $A$ is invertible. 
 So 
 \begin{enumerate}
\item  eigenvalues  of $A$ are of the form $\pm i\lambda_j$, where $\R\ni \lambda_j\ne0$; 
%non-vanishing:  $\lambda_j\not\equal0\ \forall j$;

\item in the Jordan normal form of $A$ there are  no Jordan cells.
\end{enumerate}

By these assumptions  ${D}=2n$, and in $\mathbb{R}^{2n}$ there exists a base 
$\{\mathbf{e}_1^+,\mathbf{e}_1^-,\dots, \mathbf{e}_n^+, \mathbf{e}_n^-\}$  in which  
 the linear operator $A$ takes the block-diagonal  form:
\[A=\left(\begin{array}{c c c}
\begin{array}{cc}0&-\lambda_1\\
\lambda_1&0\end{array}&&0\\
&\ddots&\\
 0&&\begin{array}{cc}0&-\lambda_n\\\lambda_n&0\end{array}\end{array}\right)%_{2n\times 2n}
 \]
 We denote by $(x_1,y_1,\dots, x_n,y_n)$ the coordinates, corresponding to this base, and for 
 $j=1,\dots,n$\  set  $z_j=x_j+iy_j$. Then $\mathbb{R}^{2n}$ becomes the space of complex vectors  $(z_1,\dots,z_n)$, i.e.
$ \mathbb{R}^{2n} \simeq  \C^n$. In the complex coordinates the standard inner product in $\R^{2n}$ reads
\be\label{scalar}
\langle z,z'\rangle=\text{Re}\sum_{j=1}^nz_j\bar{z}_j',\quad z,\;z'\in\mathbb{C}^n.
\ee
Let us denote  by
$$\Lambda=(\lambda_1,\dots,\lambda_n)\in (\R\setminus \{0\})^n $$
the {\it frequency vector} of the linear system \eqref{linear-s1}. Then in the complex coordinates $z$ the operator $A$ reads
\[Az=\text{diag}\{i\Lambda\} z,
\]
where $\text{diag}\{i\Lambda\} $ is the  diagonal operator, sending  a vector 
$
%\text{diag}\{i\Lambda\}:
 (z_1,\dots,z_n) 
 $
 to 
 $
 (i\lambda_1z_1,\dots,i\lambda_nz_n).
 $
Therefore in $\R^{2n}$,  written as
the complex space $\mathbb{C}^n$,   linear equation \eqref{linear-s1} takes the diagonal form,
\[
%\frac{d}{d t}
(d/dt)
v_k+i\lambda_k v_k=0, \quad 
\;\; 1\leqslant k\leqslant n.\]
Below we examine the perturbed   eq.  \eqref{1}, using these complex coordinates.

We next discuss the random process $W(t)$, written in the complex coordinates.
The standard {\it complex Wiener process } has the form 
\be\label{complex}
\beta^c(t)=\beta^+(t)+i\beta^-(t) \in \C,
\ee
where $\beta^+(t)$ and $\beta^-(t)$ are independent  standard (real) Wiener processes, defined on some 
 probability space  $(\Omega, \mathcal{F}, \mathbf{P})$. 
 Then $\bar\beta^c(t)=\beta^+(t)-i\beta^-(t)$,
 and any Wiener process  $W(t)\in \mathbb{C}^n$ may be conveniently 
  written in the complex from as %\,\footnote{we identify Wiener processes with the same distribution.}
\be\label{general}
W_k = \sum_{l=0}^{n_1}\Psi_{kl}^1\beta^c_l+\sum_{l=1}^{n_1}\Psi_{kl}^2 \bar\beta^c_l,\quad  k=1,\dots,n,
\ee
where $\Psi^1=(\Psi_{kl}^1)$ and $\Psi^2=(\Psi_{kl}^2)$ are complex $n\times n_1$ matrices and $\{\beta^c_l\}$ are independent
standard  complex Wiener processes. Again, in order to   simplify  presentation below we suppose 
   that the noise in  eq.~\eqref{1} is of  the form
\[
W (t)=\sum_{l=1}^{n_1}\Psi_{kl} \beta^c_l(t), \quad k=1,\dots,n.
\]
We  {\it do not} assume that the matrix $\Psi$ is non-degenerate (in particular, it may be zero). 
%where the dispersion
%$\Psi=(\Psi_{kl})$ is a complex $n\times n$ matrix and $\{\beta_l(t),l=1,\dots n\}$ are independent standard complex Winer processes.
 Then   the perturbed   equation \eqref{1} in the complex coordinates  reads as
\begin{equation}\label{v-equation1}
d v_k+i\lambda_kv_kdt={\ep} P_k(v)dt+\sqrt{{\ep}}\sum \Psi_{kl}d\beta^c_l(t), \quad k=1,\dots,n\,,
\end{equation}
where $ v=(v_1,\dots,v_n)\in\mathbb{C}^n$ and $0<\eps\le1$. 

The results, obtained below for eq. \eqref{v-equation1}, remain true for general equations \eqref{1} on the price 
of heavier calculation. The corresponding argument is sketched in Section~\ref{s-ef-eq}.

\subsection{Assumptions on $P$ and on  the perturbed equation.}
Our first goal is to study  equation~\eqref{v-equation1}  with  $0<{\ep}\ll1$ 
on a time interval $0\le t \le   \eps^{-1}T$,  where $T>0$ is a fixed constant.  Introducing the slow time
 $$\tau={\ep} t$$
 we write the  equation   as
\begin{equation}\label{v-equation-slowtime}
dv_k(\tau)+ i{\ep}^{-1}\lambda_kv_kd\tau=P_k(v)d\tau+\sum_{l=1}^{n_1}\Psi_{kl}d\tilde{\beta^c_l}(\tau),\qquad k=1,\dots,n, \; \;0\le \tau\le T.
\end{equation}
 Here 
$\{\tilde\beta^c_l(\tau),l=1,\dots,n_1\}$ is another set of independent  standard complex Wiener processes, which we now
 re-denote back to $\{\beta^c_l(\tau),l=1,\dots,n_1\}$.
 We stress that the equation above is nothing but the original  eq.~\eqref{1}, where its linear part
 \eqref{linear-s1} is conservative and non-degenerate in the sense of conditions (1) and (2),  written in the complex coordinates  
 and slow time. So all results below concerning eq.~\eqref{v-equation-slowtime} (and eq.~\eqref{a-equation2}) may be 
 reformulated for eq.~\eqref{1} on the price of heavier notation.

 Everywhere below in this paper we assume  the following assumption concerning  the well-posedness of eq.~\eqref{v-equation-slowtime}:

\begin{assumption}\label{assume-wp1}
 \begin{enumerate} \item The  drift $P(v)=(P_1(v),\dots,P_n(v))$ is a locally Lipschitz vector field,
belonging to $ \text{Lip}_{m_0}(\mathbb{C}^n,\mathbb{C}^n)$ for some 
 $m_0\ge0$ (see \eqref{nota}).
\item For any $v_0\in\mathbb{C}^n$,  equation \eqref{v-equation-slowtime} has a unique strong
solution $v^{\ep}(\tau;v_0)$, $\tau\in[0,T]$, equal $v_0$ at $\tau=0$.
Moreover,  there exists  
$m_0'>(m_0\vee1)$ such that 
\begin{equation}\label{moment1}
{\EE} \sup_{0\leqslant\tau\leqslant T}|v^{{\ep}}(\tau;v_0)|^{2 m'_0}\leqslant C_{m'_0}(|v_0|,T)<\infty
\qquad \forall \, 0<\eps\le1, 
\end{equation}  
where $C_{m'_0}(\cdot)$ is a non-negative  continuous function on $\R_+^2$,  non-decreasing in both arguments.
\end{enumerate}
\end{assumption}

Our proofs easily generalize to the case when the vector field $P$ is locally Lipschitz, satisfying 
$
|P(v)| \le C (1+|v|)^{m_0}
$
for all $v$ and some $C>0, \, m_0\ge0$.\footnote{ 
See \cite{JKW} for averaging in deterministic perturbation of  eq.~\eqref{linear-s1} by locally Lipschitz vector fields.}
In this case the argument remains essentially the same (but become a bit longer), and the constants in estimates depend
not only on $m_0$, but also on the locally Lipschitz constant of $P$, which is the function
$
R \mapsto $Lip$\,P\!\mid\!_{\bar B_R(\C^n)}. 
$

Below  $T>0$ is fixed and the dependence of constants on $T$  usually is not indicated.  
Solutions of eq.  \eqref{v-equation-slowtime}  are assumed to be strong, unless otherwise stated. As usual,  strong solutions are
understood in the sense of an integral equation. I.e., $v^\eps(\tau;v_0) = v(\tau)$, $0\le\tau\le T$, is a strong solution, equal $v_0$ at $\tau=0$,
if a.s.
$$
v_k(\tau)+ \int_0^\tau\big (
i{\ep}^{-1} \lambda_kv_k(s) - P_k(v(s))\big)ds
= v_{0k}+\sum_{l=1}^{n_1}\Psi_{kl}  {\beta^c_l}(\tau),
\;\;\; k=1,\dots,n,
$$
for $0\le\tau\le T$. 

\subsection{Interaction representation.}

Now  in eq.  \eqref{v-equation-slowtime} we  pass to the   {\it   interaction representation}, which means that we  substitute
\be\label{interac}
v_k(\tau)=e^{-i \tau{\ep}^{-1}\lambda_k}a_k(\tau), \quad k=1, \dots, n. 
\ee
Then $v_k(0)=a_k(0)$  and we obtain the following equations for variables $a_k(\tau)$:
\begin{equation}\label{a-equation1}
da_k(\tau)=e^{i\tau{\ep}^{-1}\lambda_k}P_k (v)+e^{i\tau{\ep}^{-1}\lambda_k}\sum_{l=1}^{n_1}\Psi_{kl}d\beta^c_l(\tau),\quad  k=1,\dots,n.
\end{equation}
Actions $I_k = |a_k|^2/2$ for solutions of \eqref{a-equation1} are the same as the actions for solutions of \eqref{v-equation-slowtime}.
Compare to \eqref{v-equation-slowtime}, in  eq.~\eqref{a-equation1}  we have removed from the drift
the large term ${\ep}^{-1}\text{diag}(i\Lambda)v$ on the price that now  coefficients of the system
 are fast oscillating functions of $\tau$.

To rewrite conveniently the equations above we introduce  the rotation operators $\Phi_w$: for
 any real vector $w=(w_1,\dots,w_n)\in\mathbb{R}^n$ we denote
\be\label{Phi}
\Phi_w: \mathbb{C}^n\to\mathbb{C}^n, \quad  \Phi_w=\text{diag}\{e^{iw_1},\dots,e^{iw_n}\}.
\ee
Then
$
(\Phi_w)^{-1}=\Phi_{-w}$, $\Phi_{w_1}\circ\Phi_{w_2}=\Phi_{w_1+w_2}$, $ \Phi_0=\mathrm{id},
$
 each $\Phi_w$ is an unitary transformation, so $\Phi_w^* = \Phi_w^{-1}$. Moreover, 
$$
|(\Phi_wz)_j|=|z_j|,\qquad \forall z, w, j.
%\in\mathbb{C}^n, w\in\mathbb{R}^n.
$$
In terms of  operators $\Phi$ we write  $v(\tau)$ as  $\Phi_{\tau{\ep}^{-1}\Lambda} a(\tau)$, and write 
 system \eqref{a-equation1}  as
\begin{equation}\label{a-equation2}
da(\tau)=\Phi_{\tau{\ep}^{-1}\Lambda}P(\Phi_{-\tau{\ep}^{-1}\Lambda} a(\tau)  )d\tau+\Phi_{\tau{\ep}^{-1}\Lambda}\Psi d\beta^c(\tau), \quad
a(\tau)\in \C^n, 
\end{equation}
where $\beta^c(\tau)=(\beta^c_1(\tau),\dots,\beta^c_{n_1}(\tau))$.
This is the equation which we are going to study for   small ${\ep}$ for $0\leqslant \tau\leqslant T$
 under the initial condition
 \be\label{same}
 a(0)=v(0)=v_0.
 \ee
The solution  $a^\eps(\tau; v_0) = \Phi_{-\tau{\ep}^{-1}\Lambda} v^\eps(\tau; v_0)$
 of \eqref{a-equation2}, \eqref{same}  also satisfies  estimate \eqref{moment1},
for each $\eps \in (0,1]$.

 We recall that a $C^1$-diffeomorphism $G$ of  $\mathbb{C}^{n}$ transforms a vector filed $V$ to the  field $G_*V$, where
$(G_*V)(v)=dG(u)\big(V(u)\big)$ for $  u=G^{-1}v.$
In particular, 
$$
\big((\Phi_{ \tau{\ep}^{-1}\Lambda })_*P\big)(v)=\Phi_{ \tau{\ep}^{-1}\lambda_k}\circ P(\Phi_{-{\ep}\tau\Lambda}v).
$$
So  equation \eqref{a-equation2} may be written as
\[
da(\tau)=\big((\Phi_{\tau{\ep}^{-1}\Lambda})_*P\big)(a(\tau))d\tau+\Phi_{\tau{\ep}^{-1}\Lambda} \Psi d\beta^c(\tau).
\]

\subsection{The compactness}\label{s_2.3}
For $0<{\ep}\leqslant1$ we denote by $a^{{\ep}}(\tau;v_0)$ a  solution of  equation \eqref{a-equation2}, equals $v_0$ at $\tau=0$.
 Then
 $$
 a^{{\ep}}(\tau;v_0)=\Phi_{\tau{\ep}^{-1}\lambda_k}v^{{\ep}}(\tau;v_0).
 $$
 A unique solution
  $v^{{\ep}}(\tau; v_0)$ of \eqref{v-equation-slowtime} exists by Assumption \ref{assume-wp1}, so a solution $a^{{\ep}}(\tau;v_0)$  as well exists
  and is unique. Our goal is to examine its law 
$$Q_\eps:= \mathcal{D}(a^{{\ep}}(\cdot;v_0))\in\mathcal{P}(C([0,T],\mathbb{C}^n))$$ 
as ${\ep}\to0$. While $v_0$ is fixed we will usually write  $ a^{{\ep}}(\tau;v_0)$  as  $ a^{{\ep}}(\tau)$.

\begin{lemma}\label{pre-compact1}
Under Assumption \ref{assume-wp1}, the set of probability measures 
$Q_\eps, \,{\ep}\in(0,1]\}$ is  pre-compact 
 in the space $\mathcal{P}(C([0,T],\mathbb{C}^{n}))$ with respect to the weak topology.
 
% that is, for any sequence ${\ep}_k\to0$, there exists a subsequence ${\ep}_{k}'\to0$ and a measure $Q^0\in\mathcal{P}(C([0,1],\mathbb{C}^n))$ such that  $$\mathcal{D}(a^{{\ep}_k'}(\cdot))\rightharpoonup Q^0,,\;\;\text{ as }{\ep}_k'\to0.$$
\end{lemma}
\begin{proof}
Let us denote  the random force in equation \eqref{a-equation2} by 
$d\zeta^{{\ep}}(\tau):=\Phi_{\tau{\ep}^{-1}\Lambda}\Psi d\beta^c(\tau)$, where $\zeta^{\ep}(\tau)=(\zeta_l^{\ep}(\tau),\, l=1,\dots,n_1)$.
For any $k$ we have 
$$
\zeta^{{\ep}}_k(\tau)=\int_0^\tau d\zeta_k^{{\ep}}=\int_0^\tau e^{i s{\ep}^{-1}\lambda_k}\sum_{l=1}^{n_1}\Psi_{kl}d\beta^c_l(s).
$$
So $\zeta^{{\ep}}(\tau)$ is a stochastic integral of a non-random vector function. Hence, it is a Gaussian random process with zero
mean value,  and  its increments over non-intersecting time-intervals are independent.
% for any  $0\leqslant \tau_1\leqslant \tau_2\leqslant\tau_3\leqslant\tau_4\leqslant1$  the increments $\zeta^{\ep}(\tau_2)-\zeta^{\ep}(\tau_1)$ and $\zeta^{\ep}(\tau_4)-\zeta^{\ep}(\tau_3)$
 For each $k$, %\in\{1,\dots,n\}$, we have
 $$
 {\EE}|\zeta_k^{\ep}(\tau)|^2=\int_0^\tau2 \sum_{l=1}^{n_1}|\Psi_{kl}|^2 ds=:2C_k^\zeta \tau,
 \qquad C_k^\zeta  =  \sum_{l=1}^{n_1}|\Psi_{kl}|^2 \ge0,
 $$
and 
${\EE}\zeta_k^{\ep}(\tau)\zeta_j^{\ep}(\tau)={\EE}\bar\zeta_k^{\ep}(\tau)\bar\zeta_j^{\ep}(\tau)=0.$
Therefore   $\zeta_k^{\ep}(\tau)=C_k^\zeta \beta^c_k(\tau)$,  where by
   L\'evy's theorem (see \cite[p.~157]{brownianbook})
  $\beta^c_k(\tau)$ is a standard complex Wiener process.  However, the processes $\zeta^{\ep}_j$ and $\zeta^{\ep}_k$ with  $j\neq k$ are not necessarily  independent.

By the basic properties of Wiener process,  a.s. the curve $[0,T]\ni\tau\mapsto\zeta^{\ep}(\omega,\tau)\in\mathbb{C}^n$ is H\"older-continuous with exponent $\tfrac{1}{3}$, and  since $C_k^\zeta $ does not depend on ${\ep}$, then
abbreviating $ (C^{1/3}([0,T], \C^n)$ to $ C^{1/3}$ we have 
%$${B}_R(C^{1/3})=\{v\in C([0,1],\mathbb{C}^n): \|v\|_{1/3}<R\}.$$ Then we have
$$
{\PP}\big(\zeta^{\ep}(\cdot)\in \BR(C^{1/3})\big)\to1\;\; \text{as}\;\; R\to\infty, %\quad  C^{1/3}:= (C^{1/3}([0,T], \C^n), 
$$
uniformly in ${\ep}$. 
Let us write  equation \eqref{a-equation2} as
$$da^{{\ep}}(\tau)=V^{{\ep}}(\tau)d\tau+d\zeta^{{\ep}}(\tau).$$
By Assumption \ref{assume-wp1} and  since $|a^{\ep}(\tau)|\equiv |v^{\ep}(\tau)|$ we have that
\[
{\EE}\sup_{\tau\in[0,T]}|V^{\ep}(\tau)|\leqslant \cC^{m_0}\!(P){\EE} 
(1+\sup_{\tau\in[0,T]}|v^{\ep}(\tau)|)^{m_0}\leqslant
  C(|v_0|)<\infty.\]
Therefore, by   Chebyshev's  inequality, \[
\mathbf{P}(\sup_{0\leqslant \tau\leqslant T}|V^{{\ep}}(\tau)| > R)\leqslant C(|v_0|) R^{-1},
\]
 uniformly in ${\ep}\in(0,1]$. 
Since $$a^{{\ep}}(\tau)=v_0+\int_0^\tau V^{{\ep}}(s)ds+\zeta^{{\ep}}(\tau),$$
 then we get from the above that
\begin{equation}\label{compact-up1}\mathbf{P}\big(\|a^{{\ep}}(\cdot)\|_{1/3} > R\big)\to0\quad \text{as}\quad R\to\infty, \end{equation}
uniformly in ${\ep}\in(0,1]$.  By the Ascoli-Arzela theorem  the sets  $\BR(C^{1/3})$ are  compact in $C([0,T];\mathbb{ C}^n)$,
and in view of  \eqref{compact-up1} for any $\delta>0$ there exists $R_\delta$ such that
$$
%\mathcal{D}^{\ep}
Q_\eps  \big(\bar B_{R_{\delta}}(C^{1/3})\big)\geqslant1-\delta,\quad \forall {\ep}>0.
$$
So by   Prokhorov's  theorem  the set of measures
 $\{Q_\eps,0<{\ep}\leqslant1\}$ is pre-compact  in $\mathcal{P}(C[0,T],\mathbb{C}^n)$.
\end{proof}

By this lemma, for any sequence ${\ep}_l\to0$ there exists a subsequence ${\ep}_{l}'\to0$ and a measure $Q_0\in\mathcal{P}(C([0,T],\mathbb{C}^n))$ such that
\be\label{precomp}
Q_{{\ep}_l'}  \rightharpoonup Q_0\;\;\text{ as }\; {\ep}_l'\to0.
\ee

\section{Averaging of vector fields  with respect to the frequency vector}\label{s_3}
For a vector field  ${\tilde P}\in \text{Lip}_{m_0}(\mathbb{C}^n,\mathbb{C}^n)$  we
denote
$$
Y_{\tilde P}(a;t)= \big( (\Phi_{t\Lambda })_*\tilde P\big) (a)=
\Phi_{t\Lambda }\circ {\tilde P}(\Phi_{-t\Lambda }a),\quad a\in\mathbb{C}^n,\; t\in\mathbb{R},
$$
and for $T'>0$  define   {\it partial averaging}  $\llangle {\tilde P}\rrangle^{T'}$ 
of the vector field ${\tilde P}$ with respect to the frequency vector $\Lambda$ as follows:
\be\label{part_aver}
\llangle {\tilde P}\rrangle^{T'}(a)= \frac{1}{T'}\int_0^{T'}Y_{\tilde P}(a;t)dt =  \frac{1}{T'}\int_0^{T'}  \Phi_{t\Lambda }\circ {\tilde P}(\Phi_{-t\Lambda }a) dt.
\ee

%Recall that by Assumption \ref{assume-wp1}, $P\in \text{Lip}_{m_0}(\mathbb{C}^n,\mathbb{C}^n)$.
\begin{lemma}For any ${T'}>0$, $\llangle {\tilde P}\rrangle^{T'}(a)\in \text{Lip}_{m_0}(\mathbb{C}^n,\mathbb{C}^n)$ and
$
\cC^{m_0} (\llangle {\tilde P}\rrangle^{T'}) \le \cC^{m_0} ({\tilde P})
$
(see \eqref{nota}).
% with the same constant~$C_P$.
\end{lemma}
\begin{proof}
%Let $  B_R(\C^n)=\{z\in\mathbb{C}^n:|z|\leqslant R\}$.
 If $a\in   \BR(\C^n)$, then $\Phi_{-t\Lambda }a\in  \BR(\C^n)$ for each $t$.
 So
$% \begin{equation}\label{y-t-1}
 |Y_{\tilde P}(a;t)|=|(\Phi_{t\Lambda})_*  {\tilde P}(a)|=|{\tilde P}(\Phi_{-t\Lambda }a))|,
$
and from here
$$
|\llangle {\tilde P}\rrangle^{T'}(a)|\leqslant \sup_{0\le t \le {T'}}|Y_{\tilde P}(a;t)|\leqslant \cC^{m_0} ({\tilde P}) (1+ R)^{m_0}.
$$
Similarly, for any $a_1,a_2\in  \BR(\C^n)$,
$$%\begin{equation}\label{y-t-2}
|Y_{\tilde P}(a_1;t)-Y_{\tilde P}(a_2;t)|=|{\tilde P}(\Phi_{-t\Lambda }a_1)-{\tilde P}(\Phi_{-t\Lambda }a_2)| 
\leqslant \cC^{m_0}({\tilde P})(1+R)^{m_0}|a_2-a_1|\quad  \forall t\geqslant0,
$$%\end{equation}
 so that
%Form \eqref{y-t-1}, we have,
%and from \eqref{y-t-2}, we have
\[|\llangle {\tilde P}\rrangle^{T'}(a_1)-\llangle {\tilde P}\rrangle^{T'}(a_1)|\leqslant \cC^{m_0}({\tilde P}) (1+R)^{m_0}|a_1-a_1|.\]
This proves the assertion.
\end{proof}

Now we define 
  {\it  averaging}  of the  vector field ${\tilde P}$ with respect to the frequency vector $\Lambda$ as

\begin{equation}\label{P-averaging}
\llangle {\tilde P}\rrangle(a)=\lim_{{T'}\to\infty}\llangle {\tilde P}\rrangle^{T'}(a)=
\lim_{{T'}\to\infty}\frac{1}{{T'}}\int_0^{T'} (\Phi_{t\Lambda })_* {\tilde P}(a)dt.
\end{equation}

\begin{lemma}\label{average-lm}
\begin{enumerate}
\item The limit \eqref{P-averaging} exists for any $a$. Moreover, 
 $\llangle {\tilde P}\rrangle$ belongs to $\text{Lip}_{m_0}(\mathbb{C}^n,\mathbb{C}^n)$  and $\cC^{m_0}( \llangle {\tilde P}\rrangle ) \le
\cC^{m_0}({\tilde P})$.
 %with the same constant  $C_P$;
\item If $a\in   \BR(\C^n)$, then the rate of convergence in \eqref{P-averaging} depends not on $a$, but only on $R$.
\end{enumerate}
\end{lemma}

This is the main lemma of  deterministic averaging for vector fields.  For its proof  see
 \cite[Lemma~3.1]{JKW}.\footnote{ In fact, 
 if $\tilde P$ is any  locally Lipschitz vector field, then $\llangle {\tilde P}\rrangle$ is well defined and locally Lipschitz with the same 
locally  Lipschitz constant (see discussion after Assumption~\ref{assume-wp1})  as $\tilde P$. See in \cite{JKW}.}
\smallskip

The averaged vector field~$\llangle {\tilde P}\rrangle$ is invariant with respect to transformations $\Phi_{\theta\Lambda}$:
\begin{lemma}\label{average-invariant}
%The mapping $a\mapsto\llangle P\rrangle (a)$ commutes with all the operators $\Phi_{-\theta\Lambda }$, $\theta\in\mathbb{R}$. That is
For all $a\in \C^n$ and $\theta \in \R$,
$$
\big(\Phi_{\theta\Lambda}\big)_*  \llangle {\tilde P} \rrangle (a) \equiv
 \Phi_{\theta\Lambda } \circ \llangle {\tilde P}\rrangle\circ\Phi_{-\theta\Lambda }(a)=  \llangle {\tilde P} \rrangle (a) .
$$
\end{lemma}
\begin{proof}
For definiteness let $\theta>0$. For any ${T'}>0$ we have:
\[\llangle {\tilde P}\rrangle^{T'}(\Phi_{-\theta\Lambda}(a))=\frac{1}{{T'}}\int_0^{T'}\Phi_{t\Lambda }\circ {\tilde P}(\Phi_{-t\Lambda }\circ \Phi_{-\theta\Lambda}(a))dt=\frac{1}{{T'}}\int_0^{{T'}}\Phi_{t\Lambda }\circ {\tilde P}(\Phi_{-(t+\theta)\Lambda}a)dt.\]
Since  $\Phi_{t\Lambda }=\Phi_{-\theta\Lambda}\circ \Phi_{(t+\theta)\Lambda}$, then this equals
\[
\frac{1}{{T'}}\Phi_{-\theta\Lambda}\Big(\int_0^{T'}\Phi_{ (t+\theta)\Lambda}\circ {\tilde P}(\Phi_{-(t+\theta)\Lambda}a)dt\Big)
=\Phi_{-\theta\Lambda }\circ \llangle {\tilde P}\rrangle^{{T'}}(a)+O\Big(\frac{1}{{T'}}\Big).
\]
Passing to the limit as ${T'}\to\infty$ we obtain the assertion.
\end{proof}

The  statement below asserts that the averaged vector filed $\llangle {P}\rrangle$ is as least as smooth as ${P}$. 
\begin{proposition} \label{C-m-smooth-average}If ${P}\in C^m(\mathbb{C}^n)$ for some $m\in\mathbb{N}$, then $\llangle{P}\rrangle\in C^m(\mathbb{C}^n)$ and $|\llangle {P}\rrangle|_{C^m(\bar B_R)}\leqslant | {P}|_{C^m(\bar B_R)}$, $\forall R>0$.
\end{proposition}
\begin{proof}
 We first fix  any  $R>0$. Then there exists a sequence of polynomial vector fields  $\{P_{R,j},j\in\mathbb{N}\}$ (cf.
 below item 3) in Section~\ref{ss_calculating})
  such that 
$|P_{R,j}  -{P} |_{C^m(\bar B_R)}\to0$ as $j\to\infty$.  An easy calculation shows that 
\be\label{easyc}
|\llangle P_{R,j}\rrangle^T - \llangle  P_{R,j}\rrangle|_{C^m(\bar B_R)} \to 0  \quad \text{as} \;\; T\to\infty,
\ee
for each $j$. Since the transformations $\Phi_{t\Lambda}$ are unitary, then differentiating the integral in  \eqref{part_aver} in $a$ we get that 
\be\label{33}
|\llangle{\tilde P}\rrangle^T|_{C^m(\bar B_R)}\leqslant|{ \tilde P}|_{C^m(\bar B_R)}\quad \forall\, T>0, 
\ee
for any $C^m$-vector field $\tilde P$. Therefore
\be\label{22}
|\llangle P_{R,j}\rrangle^T - \llangle  P \rrangle^T|_{C^m(\bar B_R)} \le | P_{R,j} -   P |_{C^m(\bar B_R)} =: \kappa_j 
\to 0  \quad \text{as} \;\; j\to\infty,
\; \;\; \forall\, T>0. 
\ee
So 
$$
|\llangle P_{R,j}\rrangle^T - \llangle  P_{R,k} \rrangle^T|_{C^m(\bar B_R)} \le  2\kappa_{j\wedge k} 
\; \;\; \forall\, T>0. 
$$
From this estimate and \eqref{easyc} we find that 
$
| \llangle P_{R,j}\rrangle - \llangle  P_{R,k} \rrangle|_{C^m(\bar B_R)} \le  2\kappa_{j\wedge k} .
$
Thus $\{  \llangle P_{R,j}\rrangle \}$ is a Cauchy sequence in $C^m(\bar B_R)$. So it $C^m$-converges to a limiting field 
$\llangle P_{R,\infty}\rrangle$. As $P_{R,j}$ converges to $P$ in $C^m(\bar B_R)$, then using again \eqref{33} we find that
$
|\llangle P_{R,\infty}\rrangle  |_{C^m(\bar B_R)} \le | P |_{C^m(\bar B_R)} .
$ 
But by Lemma~\ref{average-lm} $\llangle P_{R,\infty}\rrangle$ must equal $\llangle P \rrangle$. 
Since $R>0$ is arbitrary, the assertion of the proposition follows. 
\end{proof}

Finally we note that if a vector field $P$ is Hamiltonian, then its averaging $ \llangle P \rrangle $ also is. Looking ahead we state the 
corresponding result here, despite the averaging of functions $\lan \cdot\ran$ is defined below in Section~\ref{ss_av_}.

\begin{proposition} 
If a locally Lipschitz vector field $P$ is Hamiltonian, i.e. 
$
P(z) = i \frac{\p}{\p \bar z} H(z)
$
for some $C^1$-function $H$, then $ \llangle P \rrangle $ also is Hamiltonian and 
$
 \llangle P \rrangle =  i \frac{\p}{\p \bar z} \lan  H\ran.
$
\end{proposition} 

For a proof see \cite[Theorem 5.2]{JKW}.

\subsection{Calculating the averagings} \label{ss_calculating}

{\bf 1)} The frequency vector $\Lambda= (\lambda_1,\dots,\lambda_n)$ 
 is called {\it completely resonant} if its components $\lambda_j$ are 
proportional to some $\lambda>0$. I.e. if $\lambda_j/\lambda \in\Z$ for all $j$. In this case all trajectories of the original 
linear system \eqref{linear-s1} are periodic, the operator $\Phi_{t\Lambda}$ is $2\pi/\lambda$-periodic in $t$ and so 
\be\label{w0}
\llangle {\tilde P}\rrangle(a)= \llangle {\tilde P}\rrangle^{2\pi/\lambda}(a)=
%\lim_{{T'}\to\infty}
\frac{\lambda}{{2\pi}}\
\int_0^{2\pi/\lambda} (\Phi_{t\Lambda })_* {\tilde P}(a)dt.
\ee

Completely resonant linear systems  \eqref{linear-s1}  and their perturbations \eqref{1} often occur in applications. In particular -- 
in non-equilibrium statistical physics. There the dimension $D=2n$ is large, all $\lambda_j$'s are equal, 
and the Wiener process $W(t)$ in  \eqref{1}  may be very degenerate  (it may have only two non-zero components).  E.g. see \cite{AD}, where 
more references may be found. 
\smallskip

\noindent 
{\bf 2)}  Let us 
consider the case, opposite to the above  and assume that  frequency vector $\Lambda$ is  {\it non-resonant}:
  \be\label{nr}
  \sum_{i=1}^nm_j\lambda_j\neq0\qquad \forall (m_1,\dots,m_n)\in\mathbb{Z}^n\setminus\{0\}
  \ee
  (that is,  the real numbers $\lambda_j$ are rationally independent).  Then
\be\label{relation}
\llangle {\tilde P}\rrangle(a)=\frac{1}{(2\pi)^n}\int_{\mathbb{T}^n} (\Phi_w)_* \tilde P(a)\,dw, \qquad \T^n = \R^n/(2\pi\Z^n).
\ee
Indeed, if ${\tilde P}$ is a  polynomial vector field, then \eqref{relation} easily follows from \eqref{P-averaging} by a direct
component-wise calculation. The general case is a consequence of  this result  since any vector field may be approximated by
  polynomial  fields. Details are left to the reader (cf. Lemma~3.5 in \cite{JKW}, where ${\tilde P}^{res}$ equals to the r.h.s. of
 \eqref{relation} if the vector $\Lambda$ is non-resonant).

 The r.h.s. of \eqref{relation} obviously is invariant with respect to all rotations  $\Phi_{w'}$, so it depends not on the vector 
 $a$, but only  on the corresponding torus
 \be\label{atorus}
\{ z \in \C^n: I_j(z) = I_j(a) \quad \forall\, j\}
 \ee
 (see \eqref{Iphi}) to which $a$ belongs,
  \be\label{relation1}
 \big(\Phi_{w}\big)_*  \llangle {\tilde P} \rrangle (a) \equiv  \llangle {\tilde P} \rrangle (a)  \qquad \forall w \in \C^n\; \;\; \text{if $\Lambda$ is non-resonant}.
 \ee  
  See below Section \ref{ss_51}  for discussing equations \eqref{1} with non-resonant vectors $\Lambda$. 
 \medskip
 
   \noindent 
 {\bf 3)}   If the field $\tilde P$ in \eqref{P-averaging} is polynomial, i.e.
\be\label{polyn}
\tilde P_j(a) = \sum_{|\alpha|, |\beta| \le N} C_j^{\alpha, \beta} a^\alpha \bar a^\beta, \qquad j=1, \dots, n, 
\ee
for some $N \in\Nn$, where $\alpha, \beta \in \Z_+^n$, $a^\alpha = \prod a_j^{\alpha_j}$ and $|\alpha| = \sum |\alpha_j|$, then 
$ \llangle {\tilde P} \rrangle = \tilde P^{res}$. Here $\tilde P^{res}$ is a polynomial vector field such that for each $j$, 
 $\tilde P^{res}_j(a)$ is given by the r.h.s. of \eqref{polyn}, where the summation is taken over all $|\alpha|, |\beta| \le N$, 
 satisfying 
 $
 \Lambda \cdot (\alpha - \beta) = \lambda_j. 
 $
 This easily follows from an explicit calculation of the integral in \eqref{part_aver} (see \cite[Lemma~3.5]{JKW}).

\subsection{Averaging of functions}  \label{ss_av_}
Similarly to  definition \eqref{P-averaging}, 
 for a locally Lipschitz function $f\in \text{Lip}_m (\C^n, \C)$, $m\ge0$, we define its averaging with respect to a frequency 
vector $\Lambda$ as
%We  introduce  another type of averaging of functions. For any Lipschitz function~$f(a)\in C(\mathbb{C}^n)$, we define
\begin{equation}\label{f-averaging}
\langle f\rangle(a)=\lim_{{T'}\to\infty}\frac{1}{{T'}}\int_0^{T'}f(\Phi_{-t\Lambda }a)dt, \quad a\in \C^n.
\end{equation}
Then by the same argument as above we obtain:

\begin{lemma}\label{average-lm-f}
If $f\in \text{Lip}_m (\C^n, \C)$, then 
\begin{enumerate}
\item The limit \eqref{f-averaging} exists for every $a$, and for 
 $a\in   \BR(\C^n)$  the rate of convergence in \eqref{f-averaging} depends not on $a$, but only on $R$.
\item $\langle f\rangle \in \text{Lip}_{m}(\mathbb{C}^n;\C)$  and $\cC^m( \langle f\rangle ) \le \cC^m(f)$.
\item If $f$ is $C^m$-smooth for some $m\in \Nn$,  then  $\langle f\rangle$ also is, and the $C^m$-norm of the latter is
bounded by the $C^m$-norm of the former. 
\item The function $\langle f\rangle$ commutes with the operators $\Phi_{\theta\Lambda}$, $ \theta\in\mathbb{R}$, in the sense  that
$\langle f\circ\Phi_{\theta\Lambda}\rangle=\langle f\rangle\circ\Phi_{\theta\Lambda}=\langle f\rangle$.

\end{enumerate}
\end{lemma}

If the vector $\Lambda$ is non-resonant, then similar to \eqref{relation}
we have
\be\label{relation2}
\langle {f}\rangle(a)=\frac{1}{(2\pi)^n}\int_{\mathbb{T}^n} {f}(\Phi_{-w}a)\,dw.
\ee
  The r.h.s. of \eqref{relation2} is  the {\it averaging of  function $f$ in angles}.  It 
   is constant on the tori \eqref{atorus}.

\section{Effective equation and the averaging theorem}\label{s_4}
In this section we show that the limiting measure $Q_0$ in \eqref{precomp}
% Lemma \ref{pre-compact1}
 is independent of the choice of  the sequence ${\ep}'_l\to0$, so
$
\mathcal{D}(a^{\ep})\rightharpoonup Q_0\; \text{as}\ {\ep}\to0,
$
and represent $Q_0$ as the law of a solution of an auxiliary {\it effective
equation}. The drift in this equation is the averaged drift in eq.~\eqref{v-equation-slowtime}. Now we construct its dispersion. 

 The diffusion matrix for  equation \eqref{a-equation2} is  the $n\times n$ complex matrix 
%\[da(\tau)=(\Phi_{\tau{\ep}^{-1}\Lambda})_*P(a)+\Phi_{\tau{\ep}^{-1}\Lambda}\Psi d\beta(\tau),\]
%we consider the corresponding
%$\cA^{\ep} (\tau)$,
\[ \cA^{\ep}(\tau)=(\Phi_{\tau{\ep}^{-1}\Lambda}\Psi)\cdot(\Phi_{\tau{\ep}^{-1}\Lambda}\Psi)^*.\]
 Denoting
 \be\label{ps_ep}
 \Phi_{\tau{\ep}^{-1}\Lambda}\Psi =
  \big ( e^{i\tau{\ep}^{-1}\lambda_l}\Psi_{lj}\big)=  : \big(\psi^{\ep}_{lj}(\tau) \big) = \psi^{\ep} (\tau) ,
 \ee
  we have
$$
\cA^{\ep}_{kj}(\tau)=\sum_{l=1}^{n_1}\psi^{\ep}_{kl}(\tau)\bar{\psi^{\ep}_{jl}}(\tau)=e^{i \tau{\ep}^{-1} (\lambda_k-\lambda_j)}\sum_{l=1}^{n_1}\Psi_{kl}\bar{\Psi}_{jl}.
$$
So for any $\tau>0$,
\[\frac{1}{\tau}\int_0^{\tau}\cA^{\ep}_{kj}(s)ds=\big(\sum_{l=1}^{n_1}\Psi_{kl}\bar{\Psi}_{jl}\big)\frac{1}{\tau}\int_0^{\tau}e^{is{\ep}^{-1}(\lambda_k-\lambda_j)}ds,\]
and we immediately see that
\be\label{AA}
\frac{1}{{\tau}}\int_0^{\tau}\cA^{\ep}_{kj}(\tau)d\tau\to A_{kj}\quad \text{as}\quad {\ep}\to0,
\ee
where
\begin{equation}\label{matrix-effective}A_{kj}=\begin{cases}\sum_{l=1}^{n_1}\Psi_{kl}\bar \Psi_{jl},&\text{if }\lambda_k=\lambda_j,\\
0,&\text{otherwise}.
\end{cases}\end{equation}
Clearly, $A_{kj}=\bar A_{jk}$, so $A$ is a Hermitian  matrix.
If $\lambda_k\neq\lambda_j$ for $k\neq j$, then
\be\label{diag_case}
A=\text{diag} \{b_1, \dots, b_n\}, \quad b_k= \sum_{l=1}^{n_1}|\Psi_{kl}|^2.
\ee

For any vector $\xi\in\mathbb{C}^n$ we get from \eqref{AA} that  $\langle A\xi,\xi\rangle\geqslant0$ since obviously
 $\langle \cA^{\ep}(\tau)\xi,\xi\rangle=|\psi^{\ep}(\tau)\xi|^2\geqslant0$ for each $ {\ep}$.
Therefore $A$ is a  nonnegative Hermitian matrix,    and there exists another nonnegative Hermitian matrix $B$
(called the principal square root of $A$) such that $BB^*=B^2=A$.   The matrix  $B$ is non-degenerate if  $\Psi$ is.

\begin{example}
If $\Psi$ is a diagonal matrix diag$\,\{\psi_1, \dots,\psi_n\}$, $\psi_j\in\R$, then
$
\cA^{\ep}(\tau) = |\Psi|^2.
$
In this case $A=|\Psi|^2$ and $B= |\Psi| = $diag$\,\{|\psi_1|, \dots,|\psi_n|\}$.
\end{example}

In fact, it is not needed that the matrix $B$ is square and Hermitian, and the argument below remains true if we take for $B$ any
complex $n\times N$-matrix (with any $N\in \Nn$), satisfying the equation
$$
B B^* =A. 
$$

Now we define the {\it  effective equation} for  eq.~\eqref{a-equation2} as follows:
\begin{equation}\label{effective-equation}
da_k- \llangle P\rrangle_k(a)d\tau= \sum_{l=1}^nB_{kl}d\beta^c_l,\quad k=1,\dots,n.
\end{equation}
Here the matrix $B$ is as above and $ \llangle P\rrangle$ is the resonant averaging of the vector field $P$.  
%$$R(a)=(R_k(a),k=1,\dots,n):= \llangle P\rrangle(a).$$
  We will usually consider this equation with the same initial condition as equations \eqref{v-equation-slowtime}
 and \eqref{a-equation2}:
\be\label{brr}
a(0) = v_0.
\ee
 Since the vector field $\llangle P\rrangle$ is locally Lipschitz and the dispersion matrix $B$ is constant, then a strong solution of
eq.~\eqref{effective-equation},~\eqref{brr}, if exists, is unique.

Note that the effective dispersion $B$ in \eqref{effective-equation} is a square root of an explicit matrix, and due to 
item 3) of Section~\ref{ss_calculating} if the vector field $P(v)$ is polynomial, then the effective drift $\llangle P\rrangle(a)$ also is 
given by an explicit formula.

\begin{proposition}\label{average-K-th} The  limiting probability  measure $Q_0$
in \eqref{precomp}  is a weak solution of effective  equation \eqref{effective-equation}, \eqref{brr}.
\end{proposition}

We recall that a measure $Q \in\mathcal{P}(C([0,T],\mathbb{C}^n))$  is a {\it weak solution}  of equation \eqref{effective-equation}, \eqref{brr}
if  $Q=\cD(\tilde a)$, where  the random process $\tilde{a}(\tau)$, $0\le \tau\le T$,  is a   weak solution
 of   \eqref{effective-equation},  \eqref{brr}. 
\footnote{Concerning weak solutions of SDEs e.g. see  e.g. \cite[Section 5.3]{brownianbook}.}
\smallskip

A proof of this result is preceded by a number of lemmas. Till the end of this section we  assume Assumption~\ref{assume-wp1}. 
 As  in  Section~\ref{s_3} we denote 
\be\label{Y}
\Ye(a;\tau{\ep}^{-1}):= (\Phi_{\tau{\ep}^{-1}\Lambda})_*P(a). 
\ee
Then  equation \eqref{a-equation2} for $a^{\ep}$ reads as
\be\label{new_form}
da^{{\ep}}(\tau)-\Ye(a^{{\ep}},\tau{\ep}^{-1})d\tau=\Phi_{\tau{\ep}^{-1}\Lambda}\Psi d\beta^c(\tau).
\ee
Denote
 \be\label{ty}
 \tilde y(a,\tau{\ep}^{-1})=\Ye(a,\tau{\ep}^{-1})-\llangle P\rrangle(a) =  (\Phi_{\tau{\ep}^{-1}\Lambda})_*P(a) - \llan P\rran (a). 
 \ee
 The following key
 lemma shows that  integrals of $\tilde y(a^{\ep},\tau{\ep}^{-1})$ with respect to $\tau$
  become small with  ${\ep}$, uniformly in the segment of integrating.

\begin{lemma}\label{key-a-lemma}
 For a  solution  $a^{\ep}(\tau)$   of  equation \eqref{a-equation2}, \eqref{same}   we have
\[
{\EE} \max_{0\leqslant\tau\leqslant T}\Big|\int_0^{\tau}\tilde y(a^{{\ep}}(s),s{\ep}^{-1})ds\Big|\to0\quad \text{as}\quad {\ep}\to0.
\]
\end{lemma}

The lemma is proved in  Subsection \ref{proof-key-a-lemma}.

Now let us  introduce a natural filtered measurable  space for the problem we consider,
\be\label{natural}
(\tilde\Omega,\mathcal{B},\{\mathcal{B}_\tau,0\leqslant\tau\leqslant T\}),
\ee
 where $  \tilde\Omega$ is the Banach space  $ C([0,T],\mathbb{C}^n)=\{a:=a(\cdot)\}$,
$ \mathcal{B}$ is its Borel $\sigma$-algebra
 and  $\mathcal{B}_\tau$  is the sigma-algebra,  generated by the r.v.'s
  $\{a(s):0\leqslant s\leqslant \tau)\}$.
   Consider the process on $\tilde\Omega$, defined by the l.h.s. of \eqref{effective-equation}
 \be\label{NR}
 N^{\llangle P\rrangle}(\tau;a)=a(\tau)-\int_0^\tau {\llangle P\rrangle}(a(s))ds, \qquad a\in\tilde\Omega, \;\tau\in[0,T].
 \ee
 Note that  for any $0\le \tau\le T$,  \ 
  $N^{\llangle P\rrangle}(\tau;\cdot)$ is a  $\mathcal{B}_\tau$-measurable 
  continuous functional on  $C([0,T]; \C^n)$.

 \begin{lemma} \label{lemma-m-limit}The random process $N^{\llangle P\rrangle}(\tau;a)$ is a martingale on the space \eqref{natural}
  with respect to the limiting    measure $Q_0$  in \eqref{precomp}.
  % where $Q_0$ is a limiting measure from Lemma \ref{pre-compact1}, as ${\ep}\to0$.
 \end{lemma}
 \begin{proof} 
  Let us fix a $\tau\in[0,T]$ and take a  $\mathcal{B}_\tau$-measurable
  function $f^\tau\in C_b(\tilde\Omega)$.
 We will show that
 \begin{equation}\label{mart-nr1}
 {\EE}^{Q_0}\Big(N^{\llangle P\rrangle}(t;a)f^\tau(a)\Big)=\EE^{Q_0}\Big(N^{\llangle P\rrangle}(\tau;a)f^\tau(a)\Big)
 \quad \text{for any} \;\; \tau\leqslant t \leqslant T,
 \end{equation}
 which would  imply  the assertion.
 To establish this let us first consider the process 
 $$
 N^{Y,{\eps}}(\tau; a^\eps) := 
 a^{{\eps}}(\tau)-\int_0^\tau Y (a^{{\eps}},s{\eps}^{-1})ds,
 $$
 who is a martingale in view of \eqref{new_form}. As
 $$
 N^{Y,{\eps}}(\tau; a^\eps)  - N^{{\llangle P\rrangle}}(\tau; a^\eps)  = \int_0^\tau \big[  {\llangle P\rrangle}(a^\eps(s)) -
 Y(a^\eps(s), s\eps^{-1})\big]ds, 
 $$
 then by  Lemma \ref{key-a-lemma}, 
 \be\label{x00}
 \max_{0\le\tau\le T} \EE\big| N^{Y,{\eps}}(\tau; a^\eps)  - N^{{\llangle P\rrangle}}(\tau; a^\eps)\big| 
 =o_\eps(1).
 % \;\;\text{as}\;\; \eps\to0. 
 \ee
 Here and below in this proof $o_\eps(1)$ stands for a quantity  which goes to zero with $\eps$. 
 Since $ N^{Y,{\eps}}$ is a martingale, then by this relation 
 \[ \begin{split}
  {\EE} \big(N^{{\llangle P\rrangle}}(t;a^{{\eps}})f^\tau(a^{{\eps}})\big) +o_\eps(1) &=
 {\EE} \big(N^{Y,{\eps}}(t;a^{{\eps}})f^\tau(a^{{\eps}})\big) \\ 
 &= {\EE} \big(N^{Y,{\eps}}({\tau};a^{{\eps}})f^\tau(a^{{\eps}})\big) =  {\EE} \big(N^{{\llangle P\rrangle}}({\tau};a^{{\eps}})f^\tau(a^{{\eps}})\big) +o_\eps(1).
 \end{split}
 \]
 So
 \be\label{xx0}
\EE^{ Q_\eps} \big(N^{{\llangle P\rrangle}}(t;a)f^\tau(a) -  N^{{\llangle P\rrangle}}({\tau};a)f^\tau(a)\big) = o_\eps(1).
 \ee
 To get \eqref{mart-nr1} we will pass in this relation to a limit as $\eps\to0$. For this end  for $M>0$ consider the function 
 $$
 G_M(t) =\begin{cases} t, &\text{if } |t|\le M, 
 \\
M \,\text{sgn}\, t,
&\text{otherwise}.
\end{cases}
$$
Since by Assumption \ref{assume-wp1} and Lemma \ref{average-lm} 
$$
 {\EE}^{Q_{ {\eps} }} \big(\sup_{\tau\in[0,T]}  |N^{\llangle P\rrangle}(\tau;a)|^2 \Big)\leqslant{\EE}^{Q_{{\eps}}}\big(C_P(1+\sup_{\tau\in[0,T]}
 |a(\tau)|^{2(m_0\vee1)} )\Big)\leqslant C_{P,m_0}(|v_0|),
 $$
 then for any $\eps$  we have 
 \be\label{xx1}
\EE^{ Q_\eps} \big| (1-G_M)\circ \big(
 N^{\llangle P\rrangle}(t;a) f^\tau(a)  - N^{\llangle P\rrangle}(\tau;a) f^\tau(a) 
 \big)\big| \le C M^{-1}.
 \ee
 As $Q_{\eps'_l}\strela Q_0$, then by Fatou lemma this estimate stays true for $\eps=0$. 
 
 Relations  \eqref{xx0} and  \eqref{xx1}  show  that 
  $$
\EE^{ Q_\eps} \big( G_M\circ (
 N^{\llangle P\rrangle}(t;a) f^\tau(a)  - N^{\llangle P\rrangle}(\tau;a) f^\tau(a) 
 )\big) = o_\eps(1) + o_{M^{-1}}(1). 
 $$
 From this and convergence \eqref{precomp}  we derive the relation 
 $$
 \EE^{ Q_0} \big( G_M\circ (
 N^{\llangle P\rrangle}(t;a) f^\tau(a)  - N^{\llangle P\rrangle}(\tau;a) f^\tau(a) 
 )\big) =  o_{M^{-1}}(1),
 $$
 which  jointly with estimate \eqref{xx1}$_{\eps=0}$ imply  \eqref{mart-nr1} when we send $M$ to $\infty$.
 The  lemma is proved.
 \end{proof}

 \begin{definition}\label{weak-m-s}  A measure $Q$ on the space \eqref{natural} is called a solution of the
  martingale problem for
   effective equation \eqref{effective-equation} with the initial condition \eqref{brr}  if $a(0) = v_0$  \,$Q$--a.s. and

 1)   the process $\{N^{\llangle P\rrangle}(\tau;a)\in\mathbb{C}^n, \;\tau\in[0,T]\}$ (see \eqref{NR}) is a vector--martingale on the filtered 
 space \eqref{natural} with respect to the measure $Q$;

 2) for any $k,j=1,\dots,n$ 
 the process 
 \be\label{kj}
  N^{\llangle P\rrangle}_k(\tau;a) \overline{N^{\llangle P\rrangle}_j}(\tau;a)-2\int_0^\tau(B B^*)_{kj}ds,\; \;\; \;\tau\in[0,T]
  \ee
  (where $BB^* =A$)  is a martingale  on the space  \eqref{natural} with respect to the measure  $Q$, as well as the process 
 $
  N^{\llangle P\rrangle}_k(\tau;a) {N^{\llangle P\rrangle}_j}(\tau;a).
 $\footnote{ 
 That is, 
 $
 \lan  N^{\llangle P\rrangle}_k(\tau;a), \overline{N^{\llangle P\rrangle}_j}(\tau;a) \ran (\tau)= 2\int_0^\tau (BB^*)_{kj} 
 $
 and 
  $
 \lan  N^{\llangle P\rrangle}_k(\tau;a), {N^{\llangle P\rrangle}_j}(\tau;a) \ran (\tau)= 0.
 $
 See Appendix~\ref{r-o-m}. } 
 \end{definition}

 This is a classical definition, written in complex coordinates. See \cite{SV} and
  \cite[Section~5.4]{brownianbook}, where we profited from  \cite[Remark~4.12 ]{brownianbook}
 and the result of  \cite[ Problem~4.13]{brownianbook}   since by Lemma~\ref{average-lm} the
 vector field ${\llangle P\rrangle}$ in \eqref{effective-equation} is locally Lipschitz.
 \medskip

 We have
 \begin{lemma}\label{weak-m-s-lm}
 The  limiting measure $Q_0$   in  \eqref{precomp} is a solution of the
  martingale problem for  effective equation \eqref{effective-equation}, \eqref{brr}.
   \end{lemma}

 \begin{proof} Since  condition 1) in Definition \ref{weak-m-s} has been verified in Lemma \ref{lemma-m-limit}, 
  it remains to check   condition 2).  For the second term in \eqref{kj} we have, as ${\ep}\to0$,
 \be\label{xx3}
 \int_0^{\tau}\big(\psi^{\ep}(s)(\psi^{\ep}(s))^*\big)_{kj}ds=\int_0^\tau e^{i{\ep}^{-1}(\lambda_k-\lambda_j)s}(\psi\psi^*)_{kj}ds\to\tau A_{kj},
 \ee
 where the matrix $(A_{kj})$ is given by \eqref{matrix-effective}. Let us pass to the first term. Since by
  \eqref{a-equation2} and  \eqref{ps_ep}
  $$
  N^{Y,{\ep}}(\tau)
 % =a^{{\ep}}(\tau)-\int_0^\tau \Ye(a^{\ep}(s),s{\ep}^{-1})ds
  =v_0+\int_0^\tau \psi^{\ep}(s) d\beta^c(s),  \qquad  \psi^{\ep}_{lj}(s) =
   e^{i s{\ep}^{-1}\lambda_l}\Psi_{lj}, 
  $$
  then by   the complex It\^o  formula (see Appendix~\ref{a_ito}) and Assumption~2.1, 
     for any $k, j\in\{1,\dots,n\}$  the process
 \be\label{20}
  N_k^{Y,{\ep}}(\tau) \overline{ N_j^{Y,{\ep}}}(\tau) -2\int_0^\tau(\psi^{\ep}(s)(\psi^{\ep}(s))^*)_{kj}ds,
  \ee
 is a martingale. As when verifying 1), we will compare  \eqref{kj} with \eqref{20}. To do this let us consider
 \[ \begin{split}
  N_k^{\llangle P\rrangle}(\tau ;a^{\ep})& \overline{ N_j^{{\llangle P\rrangle}}}(\tau;a^{\ep}) -
  N_k^{Y,{\ep}}(\tau;a^{\ep}) \overline{ N_j^{Y,{\ep}}}(\tau;a^{\ep})\\
  &=\big( a_k^\eps(\tau) - \int_0^\tau {\llangle P\rrangle}_k(a^\eps(s))\,ds \big)
  \big( \bar a_j^\eps(\tau) - \int_0^\tau  \bar {\llangle P\rrangle}_j(a^\eps(s))\,ds \big)\\
  &-\big( a_k^\eps(\tau) - \int_0^\tau Y_k(a^\eps(s), s\eps^{-1} )\,ds \big)
  \big( \bar a_j^\eps(\tau) - \int_0^\tau  \bar Y_j(a^\eps(s) , s\eps^{-1})\,ds \big) =: M_{kj} (a^\eps, \tau). 
 \end{split}
 \] 
 Closely repeating the proof of \eqref{x00}  we get that 
 $$
\sup_{0\le \tau\le T} \EE \big| M_{kj}(a^\eps; \tau)\big| =o_\eps(1) \quad \text{as} \;\; \eps\to0.
 $$ 
 Since \eqref{20} is a martingale, then this relation and \eqref{xx3} 
 imply  that \eqref{kj} is a martingale due to  the same reasoning by which relations 
 \eqref{x00} and the fact that $N^{Y,\eps}(\tau; a^\eps)$ is a martingale imply that $N^{\llangle P\rrangle}(\tau;a)$ also is one. 
 To pass to a limit as $\eps\to0$ the  proof 
 uses that the random variables like 
 $
 N_k^{Y,{\ep}}(\tau;a^{\ep}) \overline{ N_j^{Y,{\ep}}}(\tau;a^{\ep})
 $
 are integrable uniformly in $\eps>0$ due to Assumption~\ref{assume-wp1}, where $m'_0>m_0$.

  Similarly for any $k$ and $j$, 
 the process 
 $
  N_k^{\llangle P\rrangle}(\tau) {N_j^{\llangle P\rrangle}}(\tau) 
 $
 also is a martingale. The lemma's assertion  is established.
  \end{proof}
 
 Now we can prove Proposition \ref{average-K-th}:

 \begin{proof}[of Proposition \ref{average-K-th}]
It is well known that a  solution of the   martingale problem for
a stochastic differential equation is its weak solution. Instead of referring to a corresponding theorem
(see   \cite{SV} or 
\cite[Section~5.4]{brownianbook}), again following \cite{Khas68},
 we give a short direct proof, based on another strong result from the stochastic calculus.
 By Lemma \ref{weak-m-s-lm} and the martingale representation theorem for complex processes  (see Appendix \ref{r-o-m}), we know that  there exists an extension $(\hat{\Omega},\hat{\mathcal{B}},\hat{\mathbf{P}})$ 
 of the probability space $(\tilde\Omega,\mathcal{B},Q_0)$ and on it exist standard
 independent   complex Wiener processes $\beta^c_1(\tau),\dots,\beta^c_n(\tau)$ such that
 \[da_j(\tau)-{\llangle P\rrangle}_j(a)d\tau=\sum_{l=1}^nB_{jl}d\beta^c_l(\tau), \quad
  j=1,\dots,n,\]
 where the dispersion
  $B$ is a nonnegative Hermitian matrix, satisfying $BB^*=A$. Therefore the measure $Q_0$  is a weak solution of
  effective equation \eqref{effective-equation}. We thus proved the assertion of the proposition.
 \end{proof}

By Lemma \ref{average-lm}, in effective equation \eqref{effective-equation} the drift term ${\llangle P\rrangle}$ is locally Lipschitz.
So its strong solution (if exists) is unique. By Proposition~\ref{average-K-th}  the measure $Q_0$ is  a weak solution of eq.~\eqref{effective-equation}.
Hence, by the Yamada-Watanabe theorem \cite[Section~5.3.D]{brownianbook}, \cite[Chapter~8]{SV} 
a strong solution for the effective equation exists, and 
its weak solution is unique.  Therefore the limit $Q_0=\lim_{{\ep}_l\to0 }Q_{{\ep}_l}$ does not depend on the sequence ${\ep}_l\to0$. So the 
convergence  holds as $\eps\to0$ and we thus  have established

 \begin{theorem}\label{average-thm} 
 For any 
 $v_0\in\mathbb{C}^n$  the  solution $a^{\ep}(\tau;v_0)$ of  problem \eqref{a-equation2}, \eqref{same} satisfies
 \be\label{conv}
 \cD(a^{{\ep}}(\cdot;v_0))\rightharpoonup  Q_0\text{ in }\mathcal{P}(C([0,T],\mathbb{C}^n))\quad  \text{ as }{\ep}\to0,
 \ee
 where $Q_0$ is the law of a   unique weak solution $a^0(\tau;v_0)$ 
  of effective equation \eqref{effective-equation}, \eqref{brr}.
  \end{theorem}

\begin{remark}\label{r_average-thm}
1) A straightforward 
 analysis of the theorem's proof shows that it goes without changes if $a^{{\ep}}(\tau)$ solves eq.~\eqref{a-equation2} with an initial data $v_{\eps 0}$ which 
converges to $v_0$ when $\eps\to0$. So
 \be \label{conv_new}
 \cD(a^{{\ep}}(\cdot; v_{\eps 0} ))\rightharpoonup  Q_0\text{ in }\mathcal{P}(C([0,T],\mathbb{C}^n))\quad  \text{ as }{\ep}\to0,\;\; \text{if}\;\;  v_{\eps 0} \to v_0 \;\;\text{when}\;\;
 \eps\to0.
 \ee 
 2) Denoting as before $\cD (a^\eps(\cdot; v_0)) = Q_\eps \in \cP({C([0,T]; \C^n)})$
  we use the Skorokhod representation theorem (see \cite[Section~6]{Bill})   to find a sequence 
 $ \eps_j\to0$ and processes $\xi_j(\tau)$, $0\le\tau\le T$, $j=0,1,\dots$, such that 
 $
 \cD(\xi_0) =Q_0, 
 $
  $
 \cD(\xi_j) =Q_{\eps_j}, 
 $
and $\xi_j\to\xi_0$ in $C([0,T]; \C^n)$, 
 a.s. Then \eqref{moment1} and Fatou's lemma imply that 
\be\label{a0}
\EE\| a^0\|^{2m_0'}_{C([0,T]; \C^n)}  = 
\EE^{Q_0} \| a\|^{2m_0'}_{C([0,T]; \C^n)}  =  \EE \| \xi_0\|^{2m_0'}_{C([0,T]; \C^n)} 
\le C_{m'_0} (|v_0|, T). 
\ee
\end{remark}

The result of Theorem  \ref{average-thm}
 admits an immediate generalisation to the case when the initial data $v_0$ in  \eqref{same} is a random variable:

\begin{amplification}\label{amp_ran_ini}  Let $v_0$ be a r.v., independent from the Wiener process $\beta^c(\tau)$. Then  still the 
convergence \eqref{conv} holds. 
\end{amplification}
\begin{proof}   It suffices to establish \eqref{conv} when $a^\eps$ is a weak solution of the problem. Now, let 
$(\Omega', \mathcal{F}',\mathbf{P}')$ be another probability space and $\xi_0^{\om'}$ be a r.v. on $\Omega'$, distributed 
as $v_0$. 
Then $a^{\eps\om}(\tau; \xi_0^{\om'})$ is a weak solution of \eqref{a-equation2}, \eqref{same}, 
 defined on the probability space $\Omega'\times\Omega$.  
Take  $f$ to be  a bounded continuous function on $C([0,T], \mathbb{C}^n)$. Then by the theorem above,  for each $\omega'\in\Omega'$ 
%and $v_0^{\omega'}=\xi_0(\omega')$,  
 $$
 \lim_{\eps\to0}\mathbf{E}^\Omega f(a^{\eps \om} (\cdot ;\xi_0^{\omega'}))=\mathbf{E}^\Omega f(a^{0 \om}(\cdot; \xi_0^{\omega'})).
 $$
Since $f$ is bounded, then by the Lebesgue dominated  convergence theorem we have 
$$
\lim_{\eps\to0} \mathbf{E} f(a^{\eps}(\cdot; v_0))=
\lim_{\eps\to0}\mathbf{E}^{\Omega'}\mathbf{E}^\Omega  f(a^{\eps\om}(\cdot;\xi_0^{\omega'}))=\mathbf{E}^{\Omega'}\mathbf{E}^\Omega f(a^{0\om}(\cdot ;\xi_0^{\omega'})) = \mathbf{E} f(a^{0}(\cdot; v_0)).
$$
This implies the required convergence \eqref{conv}.
\end{proof}

The convergence, stated in the last amplification, holds uniformly in the class of random  initial data $v_0$, a.s. bounded by a fixed
constant. To state the result we have to introduce  a distance in the space of measures. 

\begin{definition}\label{d_dual-Lip}
Let $M$ be a Polish (i.e. a     complete and separable) metric space. For any two measures $ \mu_1,\mu_2\in\cP(M)$ we define 
the dual-Lipschitz distance between them as 
\[
\|\mu_1-\mu_2\|_{L,M}^*:=\sup_{f\in C_b (M),\,  |f|_L\leqslant1}\Big|\langle f ,\mu_1\rangle-\langle f ,\mu_2\rangle\Big|\le2,
 \]
   where  $|f|_L= |f|_{L,M}=
    \,$Lip$\,f+\|f\|_{C(M)}$.
 \end{definition} 
 
 In the definition and below we denote 
    \be\label{recall}
   \langle f ,\mu\rangle :=\int_M f(m)\mu(dm).
   \ee

\begin{example}
 Consider the Polish spaces $C([0,T], \mathbb{C}^n)$ and $\C^n$ and the mappings
 $$
 \Pi_t: C([0,T], \mathbb{C}^n) \to \C^n, \quad  a(\cdot) \mapsto a(t), \qquad 0\le t\le T.
 $$
Noting that 
$
| f\circ \Pi_t|_{L, \C^n} \le 
| f |_{L, C([0,T], \mathbb{C}^n)}$ for each $t$ 
we get that 
\be\label{HG}
\|\Pi_t\circ  \mu_1- \Pi_t\circ \mu_2\|_{L, \C^n}^*  \le \|   \mu_1- \mu_2\|_{L, C([0,T], \mathbb{C}^n)}^* 
\ee
for all $\mu_1, \mu_2 \in \cP(C([0,T], \mathbb{C}^n))$ and all $0\le t\le T$ (where 
$ \Pi_t\circ  \mu_j \in\cP(\C^n)$ stands for  the image of $\mu_j$ under the mapping $\Pi_t$).
\end{example}

 The dual-Lipschitz  distance converts  $\mathcal{P}(M)$ to a complete metric space and 
    induces on it  a  topology, equivalent to the  weak convergence of measures, e.g. see \cite[Section 11.3]{Dud} and \cite[Section 1.7]{BK}).

\begin{proposition}\label{p_ran_ini}  Under the assumption of Amplification \ref{amp_ran_ini}
 let  the r.v. $v_0$ be  such that $|v_0| \le R$ a.s.,
for some $R>0$. Then the rate of convergence in \eqref{conv} with respect to the dual-Lipschitz distance depends only on $R$. 
\end{proposition}

\begin{proof}  
The proof of  Amplification \ref{amp_ran_ini}   shows that it suffices to verify that for a non-random initial data $v_0 \in \bar B_R(\C^n)$ the rate of convergence in \eqref{conv} depends only on $R$.  
 Assume the opposite. Then there exist a $\delta>0$, a sequence 
$\eps_j\to0$ and vectors $v_j\in \bar B_R(\C^n)$ such that 
\begin{equation}\|\cD(a^{\eps_j}(\cdot; v_j))-\cD(a^0(\cdot; v_j))\|_{L,C([0,T],\mathbb{C}^n)}^*\geqslant \delta.\end{equation}
By the same argument as in the proof of Lemma \ref{pre-compact1}, we know that the two sets of probability measures 
$\{\cD(a^{\eps_j}(\cdot; v_j))\}$ and $\{\cD(a^0(\cdot; v_j))\}$ are pre-compact 
 in  $C([0,T],\mathbb{C}^n)$. Therefore, there exists a sequence $k_j\to\infty$ such that  $\eps_{k_j}\to0$, $v_{k_j}\to v_0$  and 
$$
\cD(a^{\eps_{k_j}}(\cdot; v_{k_j}))\rightharpoonup \tilde Q_0,\;\;\;
 \cD(a^{0}(\cdot; v_{k_j}))\rightharpoonup Q_0\; \text{in}\; \cP(C([0,T],\mathbb{C}^n)).
$$
Then 
\begin{equation}\label{contr1}
 \;\|\tilde Q_0-Q_0\|_{L,C([0,T],\mathbb{C}^n)}^*\geqslant\delta.
\ee
Since in the well-posed eq. \eqref{effective-equation} 
the drift and  dispersion   are  locally Lipschitz, then  the law 
$\cD(a^0(\cdot;v'))$ is continuous with respect to the initial condition $v'$.\footnote{This is well known, and can be easily 
proved using the estimate in Remark~\ref{r_average-thm}.2).
}
 Therefore $Q_0$ is the unique weak solution of the effective equation \eqref{effective-equation} with initial condition $a^0(0)=v_0$. By \eqref{conv_new} the measure 
$\tilde Q_0$ also is a weak solution of  the problem  \eqref{effective-equation}, \eqref{brr}. 
 Hence, $Q_0=\tilde Q_0$. This   contradicts  \eqref{contr1} and proves  the assertion.
\end{proof}

  We continue  with an obvious application of Theorem \ref{average-thm} to solutions $v^\eps(\tau;v_0)$ of the original 
  eq.~\eqref{v-equation-slowtime}. 
   Consider the action-mapping
  $
  (z_1, \dots, z_n) \mapsto (I_1, \dots, I_n)=:I
   $
   (see \eqref{Iphi}).   Since the interaction representation \eqref{interac} does not change the  actions, then we have from the theorem that 
   
  \begin{corollary}\label{c_4.7}
  For any $v_0$, as  ${\ep}\to0$
  \be\label{c1}
    \cD(I\big( v^{{\ep}}(\cdot;v_0)) \big)\rightharpoonup  I\circ \cD(a(\cdot;v_0))\quad\text{in}\quad
   \mathcal{P}(C([0,T],\R_+^n)) ,
  \ee
 % and
  %\be\label{c2} \varphi \circ \cD(v^{{\ep}}(\tau;v_0)) +\eps^{-1} \Lambda\tau
   %\rightharpoonup   \varphi \circ \cD(a(\tau;v_0))\quad\text{in}\quad
   %\mathcal{P}(\T^n)\; \forall\, 0\le \tau\le T,  \ee
  where $a(\cdot;v_0)$ is a unique weak solution of effective equation \eqref{effective-equation}, \eqref{brr}.
  \end{corollary}
 % Here for a measure $\mu$ and a mapping $F$ we denote by $F\circ \mu$ the image of $\mu$ under $F$.
 \begin{example} If the drift $P$ in \eqref{v-equation-slowtime} is globally Lipschitz, that is, $\text{Lip}(P)\leqslant M$ for some $M>0$,
 then  it is not hard to see that Assumption~\ref{assume-wp1} holds, so % for equation \eqref{v-equation-slowtime}, so
  Theorem~\ref{average-thm} and Corollary~\ref{c_4.7} apply. A more interesting example is discussed below 
   in Section~\ref{s_9}. 
 \end{example}

  \subsection{Proof of Lemma \ref{key-a-lemma}}\label{proof-key-a-lemma}
    In this subsection we denote by  $\mathcal{H}_k(r;c_1,\dots)$, $k=1,2, \dots$,
      non-negative functions of $r>0$ which go to zero with $r$ and
     depend on  parameters $c_1, \dots$ (dependence of  $\mathcal{H}_k$'s on $T$  and $P$      is not indicated).
      Also, for  an event $Q$ we denote ${\EE}_{Q}f(\xi)={\EE}\mathbf{1}_{Q}f(\xi)$.

  For any ${M_1}\geqslant1$ we set
 $$
 \cE^1=  \cE^{1\eps}_{M_1}
% \Omega_{{M_1}}^{\ep}
 =\{\omega\in\Omega:\sup_{0\leqslant\tau\leqslant T}|a^{{\ep}}(\tau)|\leqslant {M_1}\}.
 $$
 By  Assumption \ref{assume-wp1} and Chebyshev's inequality,
 $$\mathbf{P}(\Omega\setminus \cE^1 )\leqslant\cH_1({M_1}^{-1}; |v_0|).
 $$
 Recalling that $\tilde y$ was defined in \eqref{ty}, by  Lemma \ref{average-lm} we have that
\[
 | \tilde y(a^{{\ep}}(s),s{\ep}^{-1}) |\leqslant |\Ye(a^{{\ep}}(s),s{\ep}^{-1})|+|{\llangle P\rrangle}(a^{{\ep}}(s))|
 \leqslant 2 \CC | a^{\ep}(s)|^{m_0}.
\]
 So, abbreviating $\tilde y(a^{{\ep}}(s),s{\ep}^{-1}) $ to $\tilde y(s)$,   in view of \eqref{moment1} we have:
 \[\begin{split}&{\EE}_{\Omega\setminus {\cE^1} }\max_{\tau\in[0,T]}|\int_0^{ \tau}\tilde y(s)ds
 |\leqslant \int_0^{T}{\EE}\big(\mathbf{1}_{\Omega\setminus{\cE^1}}|\tilde y(s)|\big)ds\\
 &\leqslant  2\CC \big( P(\Omega\setminus{\cE^1})\big)^{1/2}(\int_0^{T}{\EE}|a^{{\ep}}(s)|^{2m_0}ds)^{1/2} \\
& \leqslant 2\CC(\cH_1({M_1}^{-1}; |v_0|))^{1/2} =:\cH_2({M_1}^{-1}; |v_0|).
 \end{split}\]
 Now we should  estimate
 ${\EE}_{{\cE^1}}\max_{\tau\in[0,T]}|\int_0^{ \tau}\tilde y(s)ds|$.  For an $ M_2\geqslant1$
 consider the event 
 $$
 \cE^2= \cE^{2 \eps}_{M_2}
% \cF_{M_2}^{{\ep}}
 =\{\omega\in\Omega:\|a^{\ep}\|_{1/3}\leqslant M_2\}
 $$ 
 (see \eqref{Hold}).
 Then by \eqref{compact-up1}
 \[\mathbf{P}(\Omega\setminus {\cE^2} )\leqslant \cH_3(M_2^{-1}; |v_0|).\]
 Therefore,
 \[{\EE}_{\Omega\setminus\cE^2}\max_{\tau\in[0,T]}|\int_0^{\tau}\tilde y(s)ds|\leqslant \big(\mathbf{P}(\Omega\setminus {\cE^2})\big)^{1/2}(C_P\int_0^{T}{\EE}|a^{{\ep}}(s)|^{2m_0}ds)^{1/2}\leqslant \cH_4(M_2^{-1}; |v_0|)\,.
\]
It remains to bound
 $ \ 
 {\EE}_{{\cE^1}\cap{\cE^2}}\max_{\tau\in[0,T]}\big|\int_0^{\tau} \tilde y(s)ds\big|.$
%Without loss of generality  assuming  that ${\ep}\leqslant\frac{1}{4}$, we

 We set
$$
%L=\max\{\frac{1}{l}\leqslant\sqrt{{\ep}}:l\in\mathbb{N}\}.
N= [T/\sqrt{\eps}] +1,  \qquad L= T/N.
$$
Then
$C^{-1} \sqrt{\eps} \leqslant L \leqslant C  \sqrt{{\ep}}$
and
$ c^{-1} /  \sqrt{\eps} \leqslant N \leqslant c  /  \sqrt{\eps}$
for some constants $C$ and $ c$. We consider  a partition of  interval  $[0,T]$ by  points $\grr_l=lL$, $l=0,\dots, N$,  and denote 
%$\frac{1}{\sqrt{{\ep}}}\leqslant N\leqslant \frac{2}{\sqrt{{\ep}}}.$
$$
\eta_l=\int_{\grr_l}^{\grr_{l+1}}\tilde y(s)ds, \quad  l=0,\dots,N-1.
$$
For any $\tau\in[0,T]$ let us find  $l=l(\tau) $ such that $\tau\in[\grr_{l},\grr_{l+1}]$. Then
\[|\int_0^{\tau}\tilde y(s)ds|\leqslant |\eta_1|+\dots+|\eta_l|+|\int_{\grr_l}^{\tau}\tilde y(s) ds|.\]
If $\omega\in\cE^1$, then
$
|\int_{\grr_l}^{\tau}\tilde y(s)ds|\leqslant 2\CC {M_1}^{m_0}L.
$
Therefore
\begin{equation}\label{y-m-upper1}
{\EE}_{{\cE^1}\cap{\cE^2}} \max_{0\leqslant\tau\leqslant T}\big|\int_0^{\tau}\tilde y(s)ds\big|  \leqslant 2 \CC {M_1}^{m_0}L+\sum_{l=0}^{N-1}{\EE}_{\cE^1\cap\cE^2}|\eta_l|,
\end{equation}
and it  remains to estimate the integrals 
${\EE}_{\cE^1\cap\cE^2}|\eta_l|$ for  $  l=0,\dots,N-1$.
Observe that
\[|\eta_l|\leqslant\Big|\int_{\grr_l}^{\grr_{l+1}}[\tilde y(a^{\ep}(s),s{\ep}^{-1})-\tilde y(a^{\ep}(\grr_l),s{\ep}^{-1})]ds\Big|
 +\Big|\int_{\grr_l}^{\grr_{l+1}}\tilde y(a^{\ep}(\grr_l),s{\ep}^{-1})ds\Big|
 =: |U_l^1| + |U_l^2|.
\]
Since $\tilde y(a^{\ep},\tau{\ep}^{-1})=(\Phi_{\tau{\ep}^{-1}\Lambda})_* P(a^{\ep})-{\llangle P\rrangle}(a^{\ep})$ and $P,{\llangle P\rrangle}\in
 \text{Lip}_{m_0}(\mathbb{C}^n,\mathbb{C}^n)$, then for $\omega\in{\cE^1}\cap\cE^2$ 
 the integrand in  $U_l^1$ is bounded by
 $$2\CC {M_1}^{m_0}\sup_{\grr_l\leqslant s\leqslant\grr_{l+1}}|a^{\ep}(s)-a^{\ep}(\grr_l)|\leqslant 2\CC {M_1}^{m_0}M_2L^{1/3}.$$
So
 $$|U_l^1| \leqslant  2\CC {M_1}^{m_0}M_2L^{4/3}.$$
Now consider integral  $U_l^2$. By the definition of $\tilde y(a^{\ep},\tau{\ep}^{-1})$, we have
\[
U_l^2=\int_{\grr_l}^{\grr_{l+1}} \Ye(a^{{\ep}}(\grr_l),s{\ep}^{-1})ds\,- \, L{\llangle P\rrangle}(a^{{\ep}}(\grr_l)\big)=:Z^1+Z^2.
\]
For  integral $Z^1$, making the change of variable $s=\grr_l+{\ep} x$ with  $x\in[0,\frac{L}{{\ep}}]$, we have 
$$
Z^1={\ep}\int_0^{L/{\ep}}\Ye(a^{{\ep}}(\grr_l),\grr_l{\ep}^{-1}+x)dx.
$$
Since 
$$
\Ye(a^{{\ep}}(\grr_l), \grr_l{\ep}^{-1}+x)=\Phi_{\grr_l{\ep}^{-1}\Lambda}\circ\Phi_{x\Lambda}
P\big(\Phi_{-x\Lambda}(\Phi_{-\grr_l{\ep}^{-1}\Lambda}a^{\ep}(\grr_l))\big), 
$$
then 
$$
Z^1=
L\Phi_{\grr_l{\ep}^{-1}\Lambda}\Big(\frac{{\ep}}{L}\!
\int_0^{L/{\ep}}\Phi_{x\Lambda} %\Big(
P\big(\Phi_{-x\Lambda} \Big(\Phi_{-\grr_l{\ep}^{-1}\Lambda}a^{\ep}(\grr_l) \big)\big)dx\Big)
= L\Phi_{\grr_l{\ep}^{-1}\Lambda} \llangle P \rrangle^{L/{\ep}}\big(\Phi_{-\grr_l{\ep}^{-1}\Lambda}a^{\ep}(\grr_l)\big)
$$
(see \eqref{part_aver}). 
As $L/{\ep}\sim{\ep}^{-1/2}\gg1$ and $\big| \Phi_{-\grr_l{\ep}^{-1}\Lambda}a^{\ep}(\grr_l) \big|=  |a^{\ep}(\grr_l)|\leqslant {M_1}$ for
 $\omega\in{\cE^1}\cap\cE^2$, then  by Lemma \ref{average-lm} the partial averaging 
 $
  \llangle P \rrangle^{L/{\ep}} % ( \Phi_{-\grr_l{\ep}^{-1}\Lambda}a^{\ep}(\grr_l) )
 $
 is close to the complete averaging 
 $
  \llangle P \rrangle. %(\Phi_{-\grr_l{\ep}^{-1}\Lambda}a^{\ep}(\grr_l) ).
 $
 Thus 
\[ \begin{split}
\big| Z^1 -L\Phi_{\grr_l{\ep}^{-1}\Lambda}  &{\llangle P\rrangle}\big(\Phi_{-\grr_l {\ep}^{-1}\Lambda}(a^{\ep}(\grr_l))\big)  \big|
= \big| Z^1 -L (\Phi_{\grr_l{\ep}^{-1}\Lambda} )_* {\llangle P\rrangle} (a^{\ep}(\grr_l))  \big|\\
=  &\big| Z^1 -L  {\llangle P\rrangle} (a^{\ep}(\grr_l)) \big|
 \le L\cH_5(\sqrt{{\ep}}; {M_1}, |v_0|),
 \end{split}
\]
 where  we used  Lemma \ref{average-invariant}  to get the second equality. 
Since $Z^2 = -L{\llangle P\rrangle}(a^\eps(\grr_l))$, then
 $$
 |U_l^2| = | Z^1+Z^2| \le
 %=\big | L\big( R(a^{\ep}(\grr_l)) + \cH_5(\sqrt{{\ep}}; {M_1}, |v_0|) \big) -L R(a^{\ep}(\grr_l)) \big| =
  L\cH_5(\sqrt{{\ep}};{M_1}, |v_0|).
 $$

We then have obtained that
$${\EE}_{\cE^1\cap{\cE^2}}|\eta_l|\leqslant 2\CC {M_1}^{m_0}M_2L^{4/3}+
L\cH_5(\sqrt{{\ep}};{M_1}, |v_0|).$$
Together with \eqref{y-m-upper1} it gives us that
\[\begin{split}
{\EE}_{{\cE^1}\cap{\cE^2}}\Big(  \max_{0\leqslant\tau\le T}\big|\int_0^{\tau}\tilde y( s )ds\big| \Big)&\leqslant
 2 \CC {M_1}^{m_0}L
 +2\CC {M_1}^{m_0}M_2L^{1/3}+\cH_5(\sqrt{{\ep}}; {M_1}, |v_0|).
 \end{split}
 \]
Therefore
\[\begin{split}
{\EE}\max_{0\leqslant\tau\leqslant T}\big| \int_0^{\tau}\tilde y( s)ds\big|&
\leqslant
\cH_2({M_1}^{-1}; |v_0|)
+\cH_4(M_2^{-1};  |v_0|)\\
&+2\CC {M_1}^{m_0} {
(M_2+1)} {\ep}^{1/6}+\cH_5(\sqrt{{\ep}}; {M_1}, |v_0|).
\end{split}
\]
Now for any $\delta>0$, we perform the following procedure:
\begin{enumerate}
\item choose ${M_1}$ so large  that $\cH_2({M_1}^{-1}; |v_0|)\leqslant\delta$;
\item choose $M_2$ so large   that $\cH_4(M_2^{-1}; |v_0|)\leqslant \delta$;
\item finally, choose ${\ep}_\delta>0$ so small that 
$$
2\CC {M_1}^{m_0}(M_2+1){\ep}^{1/3}  + 
\cH_5(\sqrt{{\ep}}; {M_1}, |v_0|)\leqslant\delta \quad \text{if} \;\; \;0<{\ep}\leqslant{\ep}_\delta.
$$
\end{enumerate}

\noindent
We have seen that for any $\delta>0$,  \ 
$
{\EE}\max_{0\leqslant\tau\leqslant T}\big|\int_0^{\tau}\tilde y(a^{{\ep}}(s),s{\ep}^{-1})ds\big| \le 3 \delta
$
if ${\ep}\le {\ep}_\delta$. So
\[
{\EE}\max_{0\leqslant\tau\leqslant T} \big|\int_0^{\tau}[\Ye(a^{{\ep}}(s),s{\ep}^{-1})-
{\llangle P\rrangle}(a^{\ep}(s))]ds \big|
\to0\quad \text{as}\quad {\ep}\to0,
\]
which  completes the proof of Lemma \ref{key-a-lemma}.

  \section{Stationary solutions and mixing}\label{s_5}
  In this section we  study relation between  stationary solutions of  equation~\eqref{v-equation-slowtime} and of  the  effective equation.
   We recall that  a solution $a(\tau)$, $\tau\geqslant0$, of  equation \eqref{v-equation-slowtime} (or of effective equation \eqref{effective-equation}) is stationary if $\mathcal{D}(a(\tau))\equiv\mu$ for all $\tau\geqslant0$ and some measure 
   $\mu\in\mathcal{P}(\mathbb{C}^n)$, called a {\it  stationary measure} for the corresponding  equation.

  Through  this section  we assume that  equation \eqref{v-equation-slowtime} satisfies the following strengthening   of
   Assumption~\ref{assume-wp1}:
 \begin{assumption}\label{assume-mixing}
 \begin{enumerate}
 \item The  drift $P(v)$ is a locally Lipschitz  vector filed, belonging  to
  $\text{Lip}_{m_0}(\mathbb{C}^n,\mathbb{C}^n)$ for some  $m_0\in \Nn$.
 \item For any $v_0\in\mathbb{C}^n$  equation \eqref{v-equation-slowtime} has a unique strong solution $v^{\ep}(\tau;v_0)$,
 $\tau\ge 0$,   equal $v_0$ at $\tau=0$. There exists $m_0'>(m_0\vee1)$ such that 
  \begin{equation}\label{v-upper-mix}
  {\EE}\sup_{ {T'}\leqslant\tau\leqslant  {T'}+1}|v^{\ep}(\tau;v_0)|^{2m'_0}\leqslant C_{m'_0}(|v_0|)
  \end{equation}
for every $ {T'}\geqslant0$ and ${\ep}\in (0,  1]$,   where $C_{m_0'}(\cdot)$ is a continuous non-decreasing function. 
\item   Equation \eqref{v-equation-slowtime} is {\it mixing}.
 That is, it has a stationary  solution 
$v^\eps_{st}(\tau)$,  $\cD( v^\eps_{st}(\tau)) \equiv   \mu^\eps\in \cP(\C^n)$,  and
\be\label{mixing}
\mathcal{D}(v^{\ep}(\tau;v_0))\rightharpoonup\mu^{\ep}\text{ in }\mathcal{P}(\mathbb{C}^n)\; \;\text{ as }\tau\to+\infty,
\ee
for every $v_0$.
  \end{enumerate}
  \end{assumption}

  Under   Assumption \ref{assume-mixing} equation \eqref{v-equation-slowtime} defines in $\C^n$ a mixing 
   Markov process   with the transition probability
  $
  \Sigma_\tau(v) \in \cP(\C^n)$, $\tau \ge0$, $v\in \C^n,
  $
  where 
  $
  \Sigma_\tau(v) = \cD v^\eps(\tau; v);
  $
  e.g. see \cite[Section~5.4.C]{brownianbook}. Let us denote by $X$ the complete separable metric space
  $
  X=C([0,\infty), \C^n)
  $
  with the distance
  \be\label{X}
  \text{dist} (a_1, a_2) = \sum_{N=1}^\infty 2^{-N} \frac{ \| a_1 -a_2\|_{C([0,N], \C^n)}} {1+ \| a_1 -a_2\|_{C([0,N], \C^n)}},\qquad a_1, a_2\in X\,,
  \ee
  and %for $m\in[1,2m_0']$
    consider  on $X$ continuous function
  $
  {g}(a) = \sup_{0\le t\le 1} |a(t)|^{2m'_0}.
  $
  Denoting
  $
  \mu^\tau_\eps = \cD(v^\eps(\tau; 0))
  $
  we have by the Markov property that
  \be\label{g1}
   {\EE}\sup_{T'\leqslant\tau\leqslant T'+1}|v^{\ep}(\tau;0)|^{m} = \int_{\C^n} \EE {g}(v^\eps(\cdot; v_0)) \mu^{T'}_\eps(dv_0),
  \ee
  and 
   \be\label{g2}
   {\EE}\sup_{0\leqslant\tau\leqslant 1}|v^{\ep}_{st}(\tau;0)|^{m} = \int_{\C^n}\EE  {g}(v^\eps(\cdot; v_0))   \mu^\eps(dv_0).
  \ee
  The l.h.s. of \eqref{g1} is estimated in \eqref{v-upper-mix}. To estimate \eqref{g2} we will  pass in the r.h.s. of \eqref{g1} to the limit
  as ${T'}\to\infty$, using \eqref{mixing}. To do that we  start with a lemma
  
  \begin{lemma}\label{l_Lip}
  Let $n_1, n_2 \in \Nn$, $\cB\subset \R^{n_1}$ is a closed convex set which contains more than one point, and $F:\cB  \to \R^{n_2}$ is
  a Lipschitz mapping. Then $F$  may be extended to a map $\tilde F: \R^{n_1} \to \R^{n_2}$ in such a way that 
  
  a) Lip$\, \tilde F$ = Lip$\,F$,
  
  b) $\tilde F(\R^{n_1}) = F(\cB )$.
  \end{lemma}
  \begin{proof}
  Let $\Pi: \R^{n_1} \to \cB $ be the projection, sending any point of $\R^{n_1}$  to a nearest point    in $\cB $. 
  Then Lip$\, \Pi$ =1, see Appendix~\ref{app_Lip}, 
    and  obviously  $\tilde F = F \circ \Pi$ is a required extension of $F$. 
  \end{proof}
  
   Since $\cC^{m_0}(P) =: C_*<\infty$, then for any $M\in\Nn$ the restriction of $P$ to $\bar B _M(\C^n)$ and its Lipschitz constant both 
   are bounded by $(1+M)^{m_0} C_*$. In view of the lemma above we may extend 
  $
  P\!\mid_{\BM(\C^n)}
  $
  to a Lipschitz mapping $P^M:  \C^n \to \C^n$ such that  
   $$
  \text{Lip} \,(P^M) \le (1+M)^{m_0} C_*, \quad \sup | P^M(v)| \le (1+M)^{m_0} C_*.
  $$
  Considering  for a solution $v(\tau)$ of equation \eqref{v-equation-slowtime}  the stopping time
  $\tau_M = \inf\{t\ge0: | v(t)| \ge M\}$ and denoting by  $v^{\eps M} $ the stopped solution
  $
  v^{\eps M} (\tau; v_0) = v^{\eps} ( \tau\wedge \tau_M; v_0),
  $
 we note that the process $ v^{\eps M} $ will not change if  in \eqref{v-equation-slowtime} we replace $P(v)$ by $P^M(v)$. So
  $ v^{\eps M} (\tau;v_0)$ is a stopped solution of a stochastic equation with Lipschitz coefficients,  and thus  a.s. the
 curve $v^{\eps M \omega}(\cdot; v_0)\in X$ continuously depends on $v_0$, for each $M\in \Nn$. As
 $$
 {g}( v^{\eps M} ) \le  {g}( v^{\eps} )\quad \text{ a.s.,}
 $$
 then in view of \eqref{mixing} and \eqref{v-upper-mix},
  for every $M$ and any $N>0$,
 $$
 \int \EE  (N\wedge {g})( v^{\eps M} (\cdot; v)) \,   \mu^\eps(dv) = \lim_{{T'}\to\infty}  \int \EE (N\wedge {g})( v^{\eps M} (\cdot; v)) \,
 \mu^{T'}_\eps(dv)  \le C_m(0)
 $$ 
 (to get the last inequality from \eqref{v-upper-mix} we used the Markov property). 
 Passing in the l.h.s.   to the limit as $N\to\infty$ using the monotone convergence theorem we see that
 \be\label{gru}
  \int\EE  {g}( v^{\eps M} (\cdot; v)) \,   \mu^\eps(dv) \le C_m(0).
 \ee
  Since for every $v$   a.s. 
 $
 {g}( v^{\eps M} (\cdot; v)) \nearrow  {g}( v^{\eps} (\cdot; v)) \le\infty$
 as 
  $ M\to\infty,
 $
  then using again the latter  theorem we derive from \eqref{gru} that
 $$
 \int\EE  {g}( v^{\eps } (\cdot; v)) \,   \mu^\eps(dv) \le C_m(0).
 $$
 Evoking \eqref{g2} we get

  \begin{lemma}\label{stat_est}
  The stationary solution  $v^{\ep}_{st}(\tau)$   satisfies  estimate  \eqref{v-upper-mix} with
  $C_{m_0'}(|v_0|)$ replaced by $C_{m_0'}(0)$.
    \end{lemma}
  Let us consider the interaction representation for $v^{\ep}_{st}$,
  $v^{\ep}_{st}(\tau)=\Phi_{-\tau{\ep}^{-1}\Lambda}a^{\ep}(\tau)$ (note that $a^{\ep}$ is not a stationary process!).
   Then $a^{\ep}(\tau)$ satisfies  equation \eqref{a-equation2}, so for any $N\in\Nn$
     the collection of measures
  $\{\cD(a^{\ep}|_{[0,N]}),0<{\ep}\leqslant1\}$ is tight in view of \eqref{v-upper-mix} (for the same reason as in Section~\ref{s_2.3}). We
   choose a sequence ${\ep}_l\to0$ (depending on $N$)   such that
  \[\cD(a^{{\ep}_l}|_{[0,N]})\rightharpoonup Q_0\text{ in } \mathcal{P}\big(C([0,N],\mathbb{C}^n)\big).\]
  Applying the diagonal process and replacing $\{{\ep}_l\}$ by a subsequence, which we still denote $\{{\ep}_l\}$, we achieve that
  $
  \cD a^{{\ep}_l} \strela Q_0 \quad\text{in}\quad \cP(X)
  $
  (see \eqref{X}).

   Since $a^{\ep}(0)=v_{st}^{\ep}(0)$, then 
 $$
 \mu^{{\ep}_l}\rightharpoonup \mu^0:=Q_0|_{\tau=0}.
 $$
  Let $a^0(\tau)$ be a process in $\C^n$   such that $\cD (a^0)=Q_0.$ Then
   \be\label{hrum}
   \cD (a^{{\ep}_l} (\tau) )\strela  \cD (a^{0} (\tau)) \quad \forall\, 0\le \tau <\infty.
   \ee
   In particular, $\cD( a^0(0) )= \mu^0$.

  \begin{proposition}\label{p_5.4}
  \begin{enumerate}
  \item The limiting process $a^0$is a stationary weak solution of effective equation \eqref{effective-equation} and
  $\cD(a^0(\tau))\equiv\mu^0$, $\tau\in[0,\infty)$. In particular,
  the limiting points of the collection of stationary   measures
  $\{\mu^{\ep},{\ep}\in(0,1]\}$ as ${\ep}\to0$ are stationary measures of the
  effective equation.
  \item   Any limiting measure $\mu^0$ is invariant under  operators  $\Phi_{\theta\Lambda}$, $\theta\in\mathbb{R}$. So
    $\cD(\Phi_{\theta\Lambda}a^0(\tau))=\mu^0$ for all  $\theta\in\mathbb{R}$ and $\tau\in[0,\infty)$.
  \end{enumerate}
  \end{proposition}
  \begin{proof} (1)
  Using  Lemma \ref{stat_est} and repeating the argument  in the proof of Proposition~\ref{average-K-th} we
   obtain that $a^0$ is a weak solution of the effective equation. It remains to  prove its stationarity.

   Take any bounded Lipschitz function $f$ on $\mathbb{C}^n$ and consider
  \[{\EE}\int_0^{1}f(v_{st}^{{\ep}_l}(\tau))d\tau={\EE}\int_0^1 f(\Phi_{-\tau{{\ep}_l}^{-1}\Lambda}a^{{\ep}_l}(\tau))d\tau.\]
Using the same argument as in the proof of Lemma \ref{key-a-lemma} (but applied to averaging of functions rather than that of vector-fields),
we obtain that
 \be\label{grr}
 \begin{split}
 {\EE}\int_0^{1} &f\big(v_{st}^{{\ep}_l}(\tau)\big)d\tau-
 {\EE} \int_0^{1}\langle f\rangle \big(a^{{\ep}_l}(\tau)\big)d\tau \\
& =
  {\EE} 
  \int_0^{1} \big( f\big(\Phi_{-\tau {\ep}_l ^{-1}\Lambda}  a^{{\ep}_l} (\tau) )-
   \langle f\rangle (a^{{\ep}_l}(\tau)\big)  \big) d\tau \to0 \quad \text{as} \;\; {\ep}_l\to0.
  %\int_0^1\langle f\rangle \big(a^{{\ep}_l}(\tau)\big)d\tau
% \to0
\end{split}
 \ee
%  Since the process $v_{st}^{{\ep}_l}(\tau)$ is stationary, then ${\EE}f(v_{st}^{{\ep}_l}(\tau))=const$, $\forall\tau\geqslant0$.
   By Lemma~\ref{average-lm-f},
   $$
   \langle f\rangle\big(a^{{\ep}_l}(\tau)\big)=\langle f\rangle\big(\Phi_{\tau{\ep}_l^{-1}\Lambda}v_{st}^{{\ep}_l}(\tau)\big)=\langle f\rangle \big(v_{st}^{{\ep}_l}(\tau)\big)
   $$
   for every $\tau$.
     Since the process $v_{st}^{{\ep}_l}(\tau)$ is stationary, then
     $$
     \text{
  ${\EE}f (v_{st}^{{\ep}_l}(\tau))=\text{Const}$ \ and \
   ${\EE} \langle f\rangle \big(a^{{\ep}_l}(\tau)\big)={\EE} \langle f\rangle  (v_{st}^{{\ep}_l}(\tau))=
    \text{Const}'$.
  }
  $$
     So we get from \eqref{grr} that
     \begin{equation}\label{st-1}{\EE}f\big(v_{st}^{{\ep}_l}(\tau)\big)-{\EE}\langle f\rangle \big(a^{{\ep}_l}(\tau)\big)\to0\;\text{ as }{\ep}_l\to0\quad \forall\, \tau.
     \end{equation}

  For any $\tau$ let us consider $\tilde f_\tau =f\circ\Phi_{\tau {\ep}_l^{-1}\Lambda}$.
  % where $w= \tau {\ep}_l^{-1}\Lambda$.
  Then $f(a^{{\ep}_l}(\tau)) =   \tilde f_\tau (v^{{\ep}_l}_{st}(\tau))$.
   Since  $\lan f \ran = \lan \tilde f_\tau \ran$ by Lemma~\ref{average-lm-f}.(4), then  applying \eqref{st-1}
   to $\tilde f_\tau $ we get:
   $$
   \lim_{{\ep}_l\to0} \EE f(a^{{\ep}_l} (\tau)) =   \lim_{{\ep}_l\to0} \EE \tilde f_\tau (v_{st}^{{\ep}_l} (\tau))=
    \lim_{{\ep}_l\to0} \EE \lan\tilde f_\tau  \ran(a^{{\ep}_l} (\tau))=   \lim_{{\ep}_l\to0} \EE \lan  f \ran(a^{{\ep}_l} (\tau)).
   $$
   From this relation, \eqref{st-1} and \eqref{hrum} we find that
   $
\EE f(a^0(\tau)) = \int\! \!f(v) \, \mu^0(dv)
   $
   for each $\tau$ and every
    $f$ as above. This implies  the first     assertion  of the lemma.
   \medskip

   (2)   Passing to the limit in   \eqref{st-1}  using  \eqref{hrum} we have
   $$
    \int f (v) \mu^0(dv) =  \EE \lan f  \ran (a^0(\tau)) \quad \forall\, \tau.
   $$
   Using this relation with $f:= f\circ \Phi_{\theta\Lambda}$ and next  with  $f:=f$  we get  that
   $$
   \int f\circ  \Phi_{\theta\Lambda} (v) \mu^0(dv) = \EE \lan f \circ\Phi_{\theta\Lambda} \ran (a^0(\tau)) = \EE \lan f  \ran (a^0(\tau))=
   \int f (v) \mu^0(dv) ,
   $$
   for  any $\theta\in \R$ and any $\tau$,  for every bounded Lipschitz function $f$. This implies the second assertion.
     \end{proof}

If effective  equation \eqref{effective-equation} is mixing, then it has a unique  stationary measure. In this case the
 measure  $\mu^0$ in Proposition~\ref{p_5.4}   does not depend on a choice of the sequence ${\ep}_l\to0$,
  and so  $\mu^{\ep} \rightharpoonup\mu^0$  as ${\ep}\to0$. Therefore, we have

  \begin{theorem}\label{thm-limit-mixing}
 If  in addition to Assumption \ref{assume-mixing} we assume
  that the effective equation is mixing and  $\mu^0$ is its unique stationary measure, then
  $$
  \mu^{\ep}\rightharpoonup\mu^0\text{ in }\mathcal{P}(\mathbb{C}^n)\quad \text{as}\quad {\ep}\to0.
  $$
 Moreover, the  measure $\mu^0$ is invariant for all  operators $\Phi_{\theta\Lambda}$,
 and the law of the stationary solution of  equation \eqref{v-equation-slowtime}, written in the interaction presentation, converges to the law of the stationary solution of effective equation \eqref{effective-equation}.
  \end{theorem}

  We recall that Theorem~\ref{average-thm}   and Corollary~\ref{c_4.7} only assure that on finite time-intervals $\tau\in [0,T]$ the 
  actions of solutions for   eq. \eqref{v-equation-slowtime}, as $\eps\to0$, 
   converge in law to the actions of  solutions for the effective equation with the same initial data. In difference, when
  ${\ep}\to0$  {\it all } of the stationary measure for eq. \eqref{v-equation-slowtime} converges
   to that for  the effective equation. This important fact was first observed
  in \cite{AD} for a special class of equations \eqref{v-equation-slowtime}.

  \begin{corollary} Under the assumption of Theorem \ref{thm-limit-mixing} , for any $v_0\in\mathbb{C}^n$ we have
  \[\lim_{{\ep}\to0}\lim_{\tau\to\infty}\cD(v^{\ep}(\tau;v_0))=\mu^0.
  \]
 % where $v^{{\ep}}$ is the solution of  equation \eqref{v-equation-slowtime} with $v^{\ep}(0)=v_0$.
  \end{corollary}
  \begin{proof} Since $\lim_{\tau\to\infty}\cD(v^{\ep}(\tau))=\mu^{\ep}$, then the result follows from Theorem \ref{thm-limit-mixing}. \end{proof}

  \begin{remark}\label{rm-density} Let us decomplexify  $\mathbb{C}^n$ to  $\mathbb{R}^{2n}$ and   write the
  effective equation  in the real coordinates $\{x=(x_1,\dots,x_{2n})\} $: 
    \[
  d x_j(\tau)-{\llangle P\rrangle}_j(x)d\tau=\sum_{l=1}^{2n}\cB_{jl}dW_l(\tau),\; j=1,\dots,2n,
  \]
  where   $W_l$  are independent  standard  real Wiener processes.
  %The matrix $\cB=(\cB_{jl})$ is non-degenerate if the matrix $B=(B_{kj})$ in \eqref{effective-equation} is non-degenerate.
  Then the stationary 
    measure $\mu^0\in\mathcal{P}(\mathbb{R}^{2n})$ satisfies the stationary Fokker-Plank equation
  \[\frac{1}{2}\sum_{l=1}^{2n}\sum_{j=1}^{2n}\frac{\partial^2}{\partial x_l\partial x_j}(\cB_{lj}\mu^0)=\sum_{l=1}^{2n}\frac{\partial}{\partial x_l}({\llangle P\rrangle}_l(x)\mu^0)
  \]
  in the sense of distributions.   If the  dispersion  matrix $\Psi$ is non-degenerate, then the diffusion   $\cB$  also is, 
   and since the drift  ${\llangle P\rrangle}(x)$ is locally
     Lipschitz,  then by the standard theory of the Fokker-Plank equation   $\mu^0=\varphi(x)dx$, where $\varphi\in C^1(\mathbb{R}^{2n})$.\footnote{E.g. firstly Theorem 1.6.8 from \cite{BKRS} implies
   that  $\mu^0=\varphi(x)dx$, where $\varphi$ is a H\"older function, and then by the usual elliptic regularity $\varphi\in C^1$.}
   \end{remark}

  \section{The non-resonant case}\label{ss_51}

 Assume that the frequency vector $\Lambda= (\lambda_1,\dots,\lambda_n)$ is non-resonant (see \eqref{nr}).
  In  Section~\ref{ss_calculating}.2)   we saw that
 in this case vector field $\llangle P\rrangle $  may be calculated via the averaging \eqref{relation} and commutes with all rotations
 $\Phi_w$, $w\in \R^n$.
   For any $j\in\{1,\dots,n\}$ let us denote
  $w^{j,t}:=(0,\dots0,t,0,\dots,0) $ (only the $j$-th entry is non-zero).  Consider ${\llangle P\rrangle}_1(z)$ and write it as
  ${\llangle P\rrangle}_1(z)=z_1R_1(z_1,\dots,z_n)$, for some complex  function $R_1$. 
    Since for $w=w^{1,t}$ we have $\Phi_w(z)=(e^{it}z_1,z_2,\dots,z_n)$, then now the first component in  relation \eqref{relation1} reads
  $
  e^{-it}z_1R_1(e^{-it}z,z_1,\dots,z_n) = e^{-it}z_1R_1(z_1,\dots,z_n), 
  $ for every $t$.
  So
   $$
   R_1(e^{-it}z_1,z_2,\dots,z_n)\equiv R_1(z_1,\dots,z_n)
   $$
  and  $R_1(z_1,\dots,z_n)$ depends not on $z_1$, but only on $|z_1|$.
  Similarly we verify  that $R_1(z_1,\dots,z_n)$ depends only on $|z_2|,\dots,|z_n|$. Therefore
  ${\llangle P\rrangle}_1(z)=z_1R_1(|z_1|,\dots,|z_n|)$. Same is true for any ${\llangle P\rrangle}_j(z)$. We then obtain the following statement:
  
  \begin{proposition}\label{p_nonres}
  If  Assumption \ref{assume-wp1} holds and   the frequency vector  $\Lambda$ is  non-resonant, then
   $\llangle P\rrangle $  satisfies  \eqref{relation}, and
  \begin{enumerate}
   \item ${\llangle P\rrangle}_j(a)=a_j R_j(|a_1|,\dots,|a_n|)$, $j=1,\dots,n$.
  \item The effective equation reads
  \be\label{ef_eq}
  da_j(\tau)-a_jR_j(|a_1|,\dots,|a_n|)d\tau=b_jd\beta^c_j(\tau),\quad \; j=1, \dots, n,
\ee
where $ b_j=(\sum_{l=1}^n|\Psi_{jl}|^2)^{1/2}$ (and $a \mapsto (a_1R_1, %(|a_1|,\dots,|a_n|), 
\dots, a_nR_n)$
is a locally Lipschitz vector-field).
\item If $a(\tau)$ is a solution of \eqref{ef_eq}, then the vector of its 
actions $I(\tau) =(I_1, \dots, I_n)(\tau)\in \R_+^n$ is a weak solution of equation
\be\label{ef_eqq}
  dI_j(\tau)- 2I_j (\Re 
  R_j)\big(\sqrt{2 I_1},\dots, \sqrt{2 I_n}\big)d\tau - b_j^2 d\tau= b_j \sqrt{2 I_j} \,d W_j(\tau), \quad
 % b_jd\beta_j(\tau),\quad b_j=(\sum_{l=1}^n|\Psi_{jl}|^2)^{1/2},
  I_j(0) = \tfrac12 |v_{0j}|^2,
\ee
$ j=1, \dots, n$, where $\{W_j\}$ are independent  standard real Wiener processes.
\item If in addition the assumptions of Theorem \ref{thm-limit-mixing} are met
 and the matrix $\Psi$ is non-degenerate, then the stationary measure $\mu^0$ reads
$
d \mu^0 = p(I) dI\,d\varphi,
$
 where $p$ is a continuous function on $\R^n_+$,   $C^1$-smooth outside the   boundary $\p \R^n_+$.
  \end{enumerate}
  \end{proposition}
  \begin{proof}
  (1) is  proved above, and (2) follows from it and \eqref{diag_case}.

 (3)  Writing the diffusion in effective equation \eqref{effective-equation} as in eq.~\eqref{ef_eq} and applying  It\^o's formula (see Appendix~\ref{a_ito})
  to $I_j = \frac12 |a_j|^2$ we get that 
  \be\label{newI}
  dI_j(\tau)- 
  \tfrac12  \big(\bar a_ j\llan P\rran_j + a_j \overline{\llan P\rran}_j \big) d\tau    - b_j^2 d\tau=  b_j \lan a_j(\tau), d\beta_j(\tau)\ran =: b_j |a_j|  d\xi_j(\tau),
%  b_j \sqrt{2 I_j} \,d W_j(\tau),
  \ee
  where $d\xi_j(\tau)=    \lan a_j / | a_j|   , d\beta_j(\tau)\ran$   (see \eqref{scalar}) and 
  for $a_j=0$ we define $a_j / | a_j|$ as 1.  Using (1) we see that the l.h.s. in \eqref{newI} is the same as that in \eqref{ef_eqq}. 
  Since $\big|a_j / | a_j| \big|(\tau)\equiv1$ for each $j$, then by
   L\'evy's theorem (e.g. see \cite[p.~157]{brownianbook}),
  $
  \xi(\tau) = (\xi_1, \dots, \xi_n)(\tau)
  $
  is a standard $n$-dimensional Wiener process and  (3) follows.

(4)   By Theorem \ref{thm-limit-mixing} the stationary measure $\mu^0$
   is invariant for all  operators $\Phi_{\theta\Lambda}$,   $\theta\in\mathbb{R}$.
   Since  the curve $\theta\mapsto \theta\Lambda\in \T^n$ is dense in  $\T^n$, then
     $\mu^0$    is invariant for all  operators $\Phi_w$,   $w\in\mathbb{T}^n$.
     As the matrix $\Psi$ is non-degenerate, then by  Remark~\ref{rm-density} we have  $d\mu^0=\tilde p(z)dz$, where
      $\tilde p$ is a $C^1$-function ($dz$ stands for the volume element in $\C^n\simeq \R^{2n}$).  
       Let us write  $z_j = \sqrt{ 2I_j} e^{i\varphi_j}$. Then $d\mu^0=  p(I,\varphi)dz$.
   In the   coordinates $(I,\varphi)$ the operators $\Phi_w$ reads as $(I,\varphi)\mapsto(I,\varphi+w)$. Since $\mu^0$ is invariant for
    all  of them,  then $ p$ does not depend on $\varphi$. So
  $
 d \mu^0= p(I)dz =  p(I)\ dI\,d\varphi
  $
  and (4) holds.
     \end{proof}

     By assertion (3) of the proposition, in the non-resonant case eq. \eqref{ef_eqq} describes asymptotic as $\eps\to0$ behaviour of actions $I_j$       of 
     solutions for eq.~\eqref{v-equation-slowtime}.  But how regular is that equation? Let us denote by 
     $
     r_j = |a_j| = \sqrt{2I_j} , \ 1\le j\le n,
     $
     the radii of components of a vector $a\in \C^n$, consider the smooth  polar-coordinates mapping 
     $
     \R_+^n \times \T^n \to \C^n$, \ $  (r,\phi) \mapsto (r_1e^{i\phi_1}, \dots, r_ne^{i\phi_n}), 
     $
     and extend it to a mapping 
     $$
     \Phi:  \R^n \times \T^n \to \C^n,
     $$
     defined by the same formula.  A $j$-th component of the drift in eq. \eqref{ef_eqq}, written in the form \eqref{newI} without the It\^o term $b_j^2d\tau$, 
     is $\Re\big( a_j \overline{\llangle P\rrangle_j(a)}\big) $.
      Due to \eqref{relation}, in the polar coordinates we write this as 
     \[ \begin{split}
     \frac1{(2\pi)^n} \int_{\T^n} \Re\Big( r_j e^{i\phi_j} \overline{ e^{iw_j} P_j(r, \phi-w)}  \Big)dw=
      \frac{1}{(2\pi)^n} \int_{\T^n} \Re\big( r_j
      e^{i\theta_j}  \bar P_j(r, \theta)       \Big)d\theta =: F_j(r), \quad r\in\R^n,
     \end{split}
     \]
     where  $F_j(r)$ is a continuous function, vanishing with $r_j$.  Since the 
      integrand in the second integral  will not change if for some $l=1,\dots, n$ we replace $r_l$ by $-r_l$ and 
    replace $\theta_l$ by  $\theta_l+\pi$, then $F_j(r)$  is even in each its 
    argument $r_l$. So  it may be written as 
     $$
     F_j (r_1, \dots, r_n) = f_j (r_1^2, \dots, r_n^2), \qquad f_j \in C(\R^n_+),
     $$
     where $f_j(x_1, \dots, x_n)$ vanishes with $x_j$. Now assume that the vector field $P$ is $C^2$-smooth. In this case   the 
     integrand in the integral for $F_j$ is $C^2$-smooth in 
     $(r,\theta) \in \R^n\times \T^n$ and  $F_j$ is a 
     $C^2$-smooth function of $r$. Then by a result of H.~Whitney (see in \cite{Whit} Theorem 1 with $s=1$ and the remark, concluding the paper),  $f_j$ extends to $\R^n$ in such a way that 
     $f_j(x)$ is $C^1$-smooth in each its argument $x_l$ and $(\p/\p x_l) f_j(x)$ is continuous on $\R^n$. So $f_j$ is $C^1$-smooth. Since $r_j^2=2I_j$, then
     we have established the following result.

      \begin{proposition}\label{p_C2}
  If  the frequency vector  $\Lambda$ is  non-resonant and $P$ is $C^2$smooth, 
  then equation \eqref{ef_eqq} may be written as 
  \be\label{newII}
   dI_j(\tau)-  G_j(I_1, \dots, I_n) d\tau - b_j^2 d\tau=  b_j \sqrt{2 I_j} \,d W_j(\tau), \quad 1\le j\le n, 
     \ee
     where $G$ is a $C^1$-smooth vector-field such that   $G_j(I)$ vanishes with $I_j$,  for each $j$. 
  \end{proposition}

  We stress that although due to  the square-root singularity in the dispersion the averaged $I$-equation \eqref{newII} is a stochastic equation without uniqueness of a solution,
  still the limiting law $\cD(I(\cdot))$ for actions of solutions for eq.~\eqref{v-equation-slowtime} is uniquely defined by Corollary~\ref{c_4.7}.

\section{Uniform in time convergence.} \label{s_8}
In this section we study the uniform in time convergence  in distribution 
of   solutions for eq. \eqref{a-equation2} to those for   effective equation~\eqref{effective-equation},  with respect to the 
dual-Lipschitz metric (see Definition~\ref{d_dual-Lip}). These results are finite-dimensional versions of those, obtained in \cite{HGK22} for stochastic PDEs. 
  Throughout the section we assume 
  
   \begin{assumption}\label{a_8.1}
   The first two items of Assumption~\ref{assume-mixing} hold, and    
   \begin{enumerate}
\item[$(3')$]   Effective equation \eqref{effective-equation} is  mixing with a stationary  measure $\mu^0$. 
% such that   $ \lan  |z|^{2m_0}, \mu^0(dz) \ran  =C_0<\infty $ 
%$\EE^{\mu^0}|z|^{2m_0} <\infty$,
For any its solution
 $a(\tau)$, $\tau\ge0$, such that $\cD(a(0))=:\mu$ and  $ \lan  |z|^{2m_0'}, \mu(dz) \ran = \EE |a(0)|^{2m_0'}
  \le M'$ for some $M'>0$ (we recall  notation \eqref{recall})  we have 
\begin{equation}
\label{mixing-quantify}
\|\cD(a(\tau))-\mu^0\|_{L,\C^n}^*\leqslant g_{M'}(\tau, d)
%)\|\mu-\mu^0\|_{L}^*,
\qquad \forall \tau\geqslant0, \;
 \text{ if } \| \mu- \mu^0\|_{L,\C^n}^*\leqslant d\leqslant2.
\end{equation}
Here the function 
$
g: \R_+^3 \to \R_+, \; (\tau, d, M) \mapsto g_M(\tau, d),
$
is continuous,  vanishes with $d$, converges to zero when  $\tau\to\infty$ 
and is such that  for each fixed
$M\ge 0$ the function $ (\tau, d)\mapsto  g_M (\tau, d)$ 
 is uniformly continuous in $d$ for $(\tau, d) \in [0, \infty)\times[0,2]$. \footnote{So $g_M$ 
extends to a continuous function on $[0, \infty]\times [0,2]$ which vanishes when $\tau=\infty$ or $d=0$.}
 \end{enumerate}
  \end{assumption}
  
  We emphasize that now  we assume   mixing  for  the effective equation,  but not  for the
   original equation~\eqref{v-equation-slowtime}.  Since
  Assumptions~\ref{a_8.1} imply   Assumptions~\ref{assume-wp1}, then the assertions of Section~\ref{s_4} with any $T>0$ hold 
   to solutions of   equations  \eqref{a-equation2} which we analyse in this section.

 %\begin{assumption}\label{assume-uniform-mixing}
 %\begin{enumerate}
 %\item The  drift $P(v)$ is a locally Lipschitz  vector filed, belonging  to  $\text{Lip}_{m_0}(\mathbb{C}^n,\mathbb{C}^n)$ for some integer $m_0\ge0$.
% \item For any $v_0\in\mathbb{C}^n$  equation \eqref{a-equation2} has a unique strong solution $v^{\ep}(\tau;v_0)$, $\tau\ge 0$,   equal $v_0$ at \tau=0$. It   satisfies 
%  \begin{equation}\label{v-upper-uniform}  {\EE}\sup_{ {T'}\leqslant\tau\leqslant  {T'}+T}|v^{\ep}(\tau;v_0)|^{2m_0}\leqslant C_{m_0}(|v_0|,T),
%  \end{equation}
%for every $ {T'}\geqslant0$, $T>0$ and ${\ep}\in (0,  1]$.
%\item   The effective equation \eqref{effective-equation} is {\it mixing} in the sense that 
%it has a unique stationary measure $\mu_0$ satisfying $\EE^{\mu_0}|z|^{2m_0}<+\infty$and there exists a function $g: \mathbb{R}^2_+\to\mathbb{R}_+$ such that  for any $\mu\in\mathcal{P}(\mathbb{C}^n)$ with $\EE^{\mu}|z|^{2m_0}<M$, 
%\begin{equation}\label{mixing-quantify}
%\|\cD(a(\tau))-\mu_0\|_{L}^*\leqslant g(\tau,M)\|\mu-\mu_0\|_{L}^*,\; \forall \tau\geqslant0,\end{equation}
%where $a(\tau)$ solves the effective equation \eqref{effective-equation} with $\cD(v(0))=\mu$ and $g(\tau,M)\to0$  as $\tau\to\infty$ uniformly with respect to $M$ on any bounded set of $\mathbb{R}_+$.  \end{enumerate}
%  \end{assumption}

  \begin{proposition}\label{p_suff_cond} 
  Assume that the first two items of 
   Assumption~\ref{assume-mixing} hold, eq.~ \eqref{effective-equation} is mixing with a stationary measure $\mu^0$, 
    and  for each $M>0$ and any $v^1, v^2 \in \bar B_M(\C^n)$ we have 
  \be\label{x1}
  \| \cD a(\tau ; v^1) - \cD a(\tau ; v^2) \|_{L,\C^n}^*  \le {\mathfrak g}_M(\tau ),
  \ee
  where 
   ${\mathfrak g}$ is a continuous function of $(M,\tau)$  which 
     goes to zero when $\tau \to\infty$ and is a non-decreasing function of $M$.     Then condition   $(3')$ holds with some function $g$. 
  \end{proposition}
    The proposition is proved below in Subsection \ref{ss_proof_prop}.  

  Note that \eqref{x1} holds (with ${\mathfrak g}$ replaced by $2{\mathfrak g}$) if 
   \be\label{x11}
  \| \cD a(\tau ; v^1) - \mu^0 \|_{L,\C^n}^*  \le {\mathfrak g}_M(\tau )\quad \forall\,  v^1 \in \bar B_M(\C^n) .
  \ee
  Usually a proof of mixing for eq.  \eqref{effective-equation} in fact establishes \eqref{x11}. So  condition  $(3')$ is a rather mild restriction. 
  
\begin{example} \label{ex_8.3}
If the assumption of Proposition \ref{p_6.4} below  is fulfilled,  then  \eqref{x1} is satisfied since 
 in this case  \eqref{x11} holds with ${\mathfrak g}_M(\tau)= \bar{V}(M)e^{-c \tau}$. Here $c>0$ is a constant and
  $\bar V(M)=\max\{V(x): x\in \bar B_M(\mathbb{C}^n)\}$, where  $V(x)$ is the  Lyapunov function as in Proposition~\ref{mixing-sufficient}.  
  See  e.g. \cite[Theorem 2.5]{mattingly} and \cite[Section~3.3]{Kulik}. 
  \end{example}

   \begin{theorem}\label{thm-uniform} 
  Under Assumption \ref{a_8.1}, 
   for  any $v_0\in\mathbb{C}^n$ %and $\delta>0$, there exists $\epsilon_\delta>0$ such that if $0<\epsilon<\epsilon_\delta$, then 
  \[
  \lim_{\eps\to0}\,
  \sup_{\tau\geqslant0}\|\cD(a^{\eps}(\tau;v_0))-\cD(a^{0}(\tau;v_0))\|_{L,\C^n}^* =0, 
  \]
  where $a^{\eps}(\tau;v_0)$ and $a^{0}(\tau;v_0)$ solve respectively \eqref{a-equation2} and \eqref{effective-equation} with the same 
   initial condition $a^{\eps}(0;v_0)=a^{0}(0;v_0)=v_0$. 
  \end{theorem}
%We first prove  the following lemma. 

\begin{proof}
Since $v_0$ is fixed, we  abbreviate $a^\eps(\tau; v_0)$ to  $a^\eps(\tau)$. 
Due to \eqref{v-upper-mix}
\be\label{M}
 {\EE}| a^{\ep}(\tau)|^{2m'_0}\leqslant C_{m'_0}(|v_0|) =:M^*    \quad \forall\, \tau\ge0. 
\ee
By \eqref{M} and   \eqref{conv}\,\footnote{Indeed, for any $N>0$ the estimate with
$ |a|^{2m'_0}$ replaced by $ |a|^{2m'_0}\wedge N$ follows from the convergence $\cD a^\eps({\tau};v_0) \strela \cD a^0({\tau};v_0) $. Then the required estimate follows from Fatou's lemma as $N\to\infty$.}
\be\label{8.05}
 {\EE}| a^{0}(\tau;v_0)|^{2m'_0} =
\lan |a|^{2m'_0}, \cD a^0({\tau};v_0)\ran \le M^* \quad \forall\, \tau\ge0.
\ee
Since $\cD a^0({\tau};0) \strela \mu^0$ as ${\tau}\to\infty$, then  from the estimate above with $v_0=0$  we get that 
\be\label{x0}
\lan |a|^{2m'_0}, \mu^0 \ran \le C_{m'_0} (0)=: C_{m'_0} .
\ee   
For later usage we note that 
since to derive  estimates  \eqref{8.05} and 
\eqref{x0} we only used Assumptions~\ref{assume-mixing}.(1),~\ref{assume-mixing}.(2) and the fact that eq.~\eqref{effective-equation} is mixing, 
then the two estimates hold under the assumptions of Proposition~\ref{p_suff_cond}.

The  constants in estimates below depend on $M^*$, but usually this   dependence  is not indicated. 
For any $T\ge0$ we denote by $a_T^0(\tau)$ a weak solution of  effective equation \eqref{effective-equation} such that 
$$
\cD a^0_T(0) = \cD a^\eps(T).
$$  
Note that $a^0_T(\tau)$ depends on $\eps$ and  that  $a^0_0(\tau) = a^0(\tau; v_0)$.

\begin{lemma}\label{random-initial-average}
Take any $\delta>0$. Then 
\begin{enumerate}
\item for any $T>0$  there exists $\eps_1 =\eps_1(\delta,T)>0$ such that if $\eps\le \eps_1$, then
\be\label{gr}
\sup_{\tau\in[0,T]}\|\cD(a^{\eps}(T'+\tau)) - \cD(a_{T'}^{0}(\tau))\|_{L,\C^n}^*\leqslant\delta/2 \quad \forall\, T'\geqslant0.
\ee
%where $a^{0}(\tau)$ solves the effective equation \eqref{effective-equation} with initial condition $\cD(a^{0}(0))=\cD(a^{\eps}(T';v_0))$. 
\item
Let us choose a $T^* = T^*(\delta)>0$, satisfying   $g_{M^*}( T, 2) \le \delta/4$ for any $T\ge T^*$. 
Then there exists 
$\eps_2 =\eps_2(\delta)>0$ such that if $\eps \le \eps_2$ and  $\|\cD(a^{\eps}(T'))-\mu^0\|_{L,\C^n}^*\leqslant\delta$ for some $T'\ge0$,   then also 
\begin{equation}\label{83a}
\|\cD(a^{\eps}(T'+T^*))-\mu^0\|_{L,\C^n}^*\leqslant  {\delta}, % \; \text{and}\; \sup_{\tau\in[T',T'+T^*]}\|\cD(a^{\eps}(\tau))-\mu^0\|_L^*\leqslant \frac{\delta}{4}+C\|\cD(a^{\eps}(T'))-\mu^0\|_L^*,
\end{equation}
and 
\begin{equation}\label{83b}
\sup_{\tau\in[T',T'+T^*]}\| \cD(a^{\eps}(\tau))-\mu^0\|_{L,\C^n}^*\leqslant \tfrac{\delta}2 + \sup_{\tau\ge0}
g_{M^*}(\tau,\delta). 
\end{equation}
\end{enumerate}
\end{lemma}
  Below we abbreviate $\| \cdot\|^*_{L, \C^n} $ to $\| \cdot\|^*_{L} $, for a measure $\nu\in \cP(\C^n)$ denote by
$a^\eps(\tau;\nu)$ a weak solution of eq.~\eqref{a-equation2} such that $\cD (a^\eps(0)) =\nu$, and define $a^0(\tau;\nu)$ 
similarly. Since eq.~\eqref{a-equation2} defines a Markov process in $\C^n$  (e.g. see \cite[Section~5.4.C]{brownianbook} and 
\cite[Section~3.3]{ khasminskii}), then 
$$
\cD a^\eps(\tau;\nu) = \int_{\C^n} \cD a^\eps(\tau;v ) \, \nu(dv), 
$$
and a similar relation holds for $\cD a^0(\tau;\nu)$. 
\begin{proof}
Denote $\nu^\eps = \cD (a^\eps(T'))$. Then 
\be\label{y1}
\cD (a^\eps(T'+\tau)) = \cD (a^\eps(\tau; \nu^\eps)), \quad  \cD (a^0_{T'}(\tau) )= \cD (a^0(\tau; \nu^\eps)). 
\ee
By  \eqref{M}, for any $\delta>0$  there exists $K_\delta>0$ such that for each $\eps$, 
 $\nu^\eps(\C^n \setminus \bar B_{K_\delta} )\le \delta/8$, where
$\bar B_{K_\delta}:=\bar B_{K_\delta}(\C^n)$.  
So
$$
\nu^\eps = A^\eps \nu^\eps_\delta +  \bar A^\eps \bar\nu^\eps_\delta , \quad  A^\eps = \nu^\eps( \bar B_{K_\delta}),\;
\bar A^\eps= \nu^\eps(\C^n\setminus \bar B_{K_\delta}), 
%A^\eps +\bar A^\eps =1,
$$
where  $ \nu^\eps_\delta $ and $ \bar\nu^\eps_\delta $ are the conditional probabilities $ \nu^\eps(\cdot \mid  \bar B_{K_\delta})$ and
$ \nu^\eps(\cdot \mid\C^n\setminus   \bar B_{K_\delta})$.
Accordingly, 
\be\label{decompp}
\cD (a^\kappa (\tau; \nu^\eps) )= A^\eps \cD (a^\kappa (\tau; \nu^\eps_\delta) )+ \bar A^\eps \cD (a^\kappa (\tau; \bar\nu^\eps_\delta)),
\ee
where $\kappa=\eps$ or $\kappa=0$. Therefore,
$$
\| \cD (a^\eps (\tau; \nu^\eps) )- \cD( a^0 (\tau; \nu^\eps)) \|_L^* \le 
A^\eps \| \cD (a^\eps (\tau; \nu^\eps_\delta) )- \cD (a^0 (\tau; \nu^\eps_\delta) )\|_L^* +
\bar A^\eps \| \cD (a^\eps (\tau; \bar \nu^\eps_\delta)) - \cD (a^0 (\tau;\bar \nu^\eps_\delta) )\|_L^*.
$$
The second term on the r.h.s  obviously is bounded by 
$2\bar A^\eps\leqslant\frac{\delta}{4}$. While by  Proposition~\ref{p_ran_ini}
and \eqref{HG} 
 there exists $\eps_1>0$, depending only on   $K_\delta$ and $T$,  such that for $0\le \tau\le T$  and  $\eps\in(0,\eps_1]$
 the first term in the r.h.s.   is $\leqslant\frac{\delta}{4}$.
 Due to \eqref{y1} this proves the first  assertion.
  \smallskip

To prove the second assertion let us  choose  $\eps_2= \eps_1(\delta/2, T^*(\delta))$. Then  we get from \eqref{gr}, 
\eqref{M}, \eqref{mixing-quantify} and the definition of $T^*$   that  for  $\eps\le\eps_2$, 
$$
\|\cD(a^{\eps}(T'+T^*))-\mu^0\|_L^*\leqslant\|\cD(a^{\eps}(T'+T^*))
-\cD(a^{0}_{T'}(T^*))\|_L^*+\|\cD(a^{0}_{T'}(T^*))-\mu^0\|_L^*\leqslant {\delta}.
$$
This proves \eqref{83a}.  Next, in view of  \eqref{gr} and \eqref{mixing-quantify}, \eqref{8.05}, 
\[\begin{split}\sup_{\theta\in[0,T^*]}\|\cD(a^{\eps}(T'+\theta))-\mu^0\|_L^*&\leqslant\sup_{\theta\in[0,T^*]}\|\cD(a^{\eps}(T'+\theta))-\cD(a_{T'}^{0}(\theta))\|_L^*\\
+\sup_{\theta\in[0,T^*]}\|\cD(a_{T'}^{0}(\theta))-\mu^0\|_L^*
&\leqslant\frac{\delta}{2}+\max_{\theta \in[0,T^*] }g_{M^*} (\theta, \delta). 
 \end{split}\]
This implies \eqref{83b}. 
\end{proof}

Now we continue to prove the theorem. 
Let us fix arbitrary $\delta>0$ and take some $0<\delta_1\le \delta/4$. Below in the proof the functions $\eps_1$, $\eps_2$ and $T^*$ are as in 
Lemma~\ref{random-initial-average}.

i) By the definition of $T^*$, \eqref{mixing-quantify}  and \eqref{M}, 
%\eqref{M} and \eqref{mixing-quantify} with $d=2$, there exists $T^*=T^*(\delta)$ such that 
\be\label{84}
\| \cD\big( a^0_{T'} (\tau)) - \mu^0\|_L^* \le  \delta_1 \quad \forall\, \tau\ge T^*(\delta_1),
\ee
for any $T'\ge0$.  We will abbreviate $T^*(\delta_1)$ to $T^*$. 

ii)  By \eqref{gr}, if $\eps\le \eps_1=\eps_1(\delta_1, T^*)>0$, then 
\be\label{85}
\sup_{0 \le \tau \le T^*} 
\| \cD\big( a^\eps (\tau)) -   \cD\big( a^0 (\tau;v_0) \big) \|_L^* \le     \tfrac{\delta_1}2. 
\ee
In particular, in view of \eqref{84} with $T'=0$, 
\be\label{855}
\| \cD\big( a^\eps (T^*)) -  \mu^0 \|_L^* <  2\delta_1.
\ee

iii) %Let  $\eps_2 =\eps_2(\delta_1,T^*)>0$  be as in item (2) of the lemma above. Then b
By \eqref{855} and 
\eqref{83a} with $\delta:=2\delta_1$ and with  $T'= nT^*$, $n=1, 2, \dots $ we get inductively  that 
\be\label{856}
\| \cD\big( a^\eps (nT^*)) -  \mu^0 \|_L^* \le  2\delta_1 \quad \forall\, n\in\Nn,
\ee
if $\eps\le\eps_2= \eps_2(2\delta_1)$. 

iv) Now by \eqref{856} and \eqref{83b} with $\delta:= 2\delta_1$, 
 for any $n\in \Nn$ and $0\le \theta\le T^*$, 
\be\label{857}
\| \cD\big( a^\eps (nT^* +\theta)) -  \mu^0 \|_L^* \le  \delta_1 + \sup_{\theta\ge0} g_{M^*} (\theta, 2\delta_1),
\ee
if $\eps\le\eps_2(2\delta_1)$.

v) Finally, if $\eps \le \eps_\# (\delta_1) =  \min\big(\eps_1(\delta_1, T^*), 
 \eps_2 (2\delta_1)\big)$, then by \eqref{85} if $\tau\le T^*$ and by \eqref{84}+\eqref{857} if $\tau\ge T^*$
 we have that  
$$
\| \cD\big( a^\eps (\tau)) -   \cD\big( a^0 (\tau;v_0) \big) \|_L^* \le 2\delta_1 + \sup_{\theta\ge0} g_{M^*}(\theta, 2\delta_1)\qquad \forall\, \tau\ge0.
$$
By the assumption, imposed in $(3')$ on function $g_{M}$,   $g_{M} (t, d)$ is uniformly continuous in $d$ and 
vanishes at $d=0$. So there exists $\delta^* =\delta^*(\delta)$, which we may assume to be $\le \delta/4$, such that if $\delta_1 =\delta^*$, then 
$g_{M^*}(\theta, 2\delta_1) \le \delta/2$  for every $\theta\geqslant0$. Then by the estimate above,
$$
\| \cD\big( a^\eps (\tau)) -   \cD\big( a^0 (\tau;v_0) \big) \|_L^* \le
 \delta \quad\text{if} \quad \eps \le \eps_*(\delta) :=   \eps_\# (\delta^*\big(\delta)\big)>0, 
 %\min(\eps_1, \eps_2)(\delta^*, T^*),
$$
for every positive $\delta$. This proves the theorem's assertion. 
\end{proof}

Since the interaction representation does not change actions, then for  the action variables of  solutions for the original  equations
 \eqref{v-equation-slowtime}  we have 
\begin{corollary}
Under the assumptions of Theorem \ref{thm-uniform} the actions of a  solution $v^\eps(\tau; v_0)$ 
for eq.~\eqref{v-equation-slowtime} which equals $v_0$ at $\tau=0$ satisfy
 \[
  \lim_{\eps\to0}\,
  \sup_{\tau\geqslant0}\|\cD(I\big( v^{\eps}(\tau;v_0)) \big) -\cD((I\big( a^{0}(\tau;v_0))\big) \|_{L,\C^n}^* =0. 
  \]
\end{corollary}
In \cite[Theorem 2.9]{AD} the assertion of the corollary is proved for a class of systems \eqref{v-equation-slowtime}. The proof in \cite{AD}
is based on the observation that the mixing rate in the corresponding equation \eqref{v-equation-slowtime} is uniform in $0<\eps\le1$. This is
 a delicate property which is harder to establish than $(3')$ in Assumption~\ref{a_8.1}.  We also note that Theorem~\ref{thm-uniform} immediately 
 implies that if equations  \eqref{a-equation2} are mixing with stationary measures $\mu^\eps$, then 
 $\mu^\eps \strela \mu^0$. Cf. Theorem~\ref{thm-limit-mixing}.

\subsection{Proof of Proposition \ref{p_suff_cond} }\label{ss_proof_prop} 

In this subsection we write solutions $a^0(\tau; v)$ 
of effective equation \eqref{effective-equation} as $a(\tau;v)$. We will   prove the assertion of 
Proposition~\ref{p_suff_cond} in four steps. 
 
 i) 
At this step,  for any non-random $v^1, v^2 \in  \bar B_M(\C^n)$ we
   denote by $a_j(\tau) := a(\tau; v^j)$, $j=1,2$, and examine the distance $\|\cD(a_1(\tau))-\cD(a_2(\tau))\|_L^*$ as a function of $\tau$ and $|v^1-v^2|$.  Let us set 
 $w(\tau) = a_1(\tau) -a_2(\tau)$ and  assume that 
  $| v^1-v^2|\le \bar d $ for some $\bar d\ge0$. Then we have 
 $$
 \dot w  %+i \eps^{-1} \text{diag} \, (\lambda_j) w
   = {\llangle P\rrangle}(a_1) - {\llangle P\rrangle}(a_2).
 $$
 Since by Lemma \ref{average-lm} and Assumption \ref{assume-wp1}.(1) 
 $$
 |{\llangle P\rrangle}(a_1(\tau)) - {\llangle P\rrangle}(a_2(\tau))| \le C |w(\tau)|\, X(\tau), \quad X(\tau) =
 1 + |a_1(\tau)|^{ m_0}  \vee |a_2(\tau)|^{ m_0} , 
 $$
   then\ 
 $
 (d/d\tau) |w|^2 \le 2C  X(\tau)   |w|^2  $ with $|w(0)| \le \bar d.
 $
 So
 \be\label{x2}
 |w(\tau)| \le \bar d \exp\big( C \int_0^\tau X(l)\,dl\big). 
 \ee 
 Denote  $ Y(T) = \sup_{0\le \tau \le T } |X(\tau)|$.  By \eqref{v-upper-mix}  estimate \eqref{moment1} holds with
 $
 C_{m_0'}(|v_0|, T) = C_{m_0'}(M) (T+1). 
 $
 So by Remark~\ref{r_average-thm}.2) \ 
 $
 \EE Y(T) \le (C_{m'_0}(M)+1) (T+1)$ (since $m_0'>(m_0\vee1)$). 
 
   For $K>0$ denote  by 
 $ \Omega_K(T)$ the event  $ \{ Y(T) \ge K\}$. Then 
 $
 \PP (\Omega_K(T)) \le (C_{m'_0}(M)+1) (T+1) K^{-1},
 $
 and 
 $
  \int_0^\tau X(l)\,dl \le \tau K
 $
 for $\omega\notin \Omega_K(T)$. 
 From here and \eqref{x2} we see that if $f$ is such that $|f|\le1$ and Lip$\,f\le1$, then 
 \be\label{x3}
 \begin{split} 
 \EE \big(  f(a_1(\tau)) -  f(a_2(\tau) ) \big)& \le 2 \PP(\Omega_K(\tau )) + \bar d \exp \big( C\tau K) \\
 &= 2(C_{m'_0}(M)+1) (\tau +1) K^{-1} + \bar d \exp \big( C\tau  K) \quad \forall\, K>0.
 \end{split} 
 \ee 
 Let us denote by $g^1_M(\tau ,\bar d)$ the function in the r.h.s. with $K= \ln \ln\big(\bar d^{-1} \vee 3 \big)$. This is a continuous function of $(\tau , \bar d,M) \in \R_+^3$, 
 vanishing when $\bar d=0$. Due to \eqref{x1} and  \eqref{x3}, 
 \be\label{x33}
 \begin{split}
 \| \cD(a(\tau;v^1)) - \cD(a(\tau;v^2)) \|_L^*& =
 \| \cD(a_1(\tau)) - \cD(a_2(\tau)) \|_L^* \\
 &\le {\mathfrak g}_M(\tau) \wedge g_M^1(\tau,\bar d) \wedge 2 =: g_M^2(\tau,\bar d) \quad \text{if} \; \; |v^1-v^2| \le \bar d. 
 \end{split}
 \ee
 The function $g_M^2$ is continuous in the variables $(\tau,\bar d, M)$, vanishes with $\bar d$ and goes to zero when $\tau \to\infty$ since $ \mathfrak g_M(\tau)$ does. 
    \smallskip
 
 ii) At this step we consider a solution $a^0(\tau;\mu)=:a(\tau;\mu)$ of \eqref{effective-equation} with  $\cD(a(0))=\mu$ as  in Assumption~\ref{a_8.1}.$(3')$ and examine the l.h.s of  \eqref{mixing-quantify} as a function of $\tau$. For any $M>0$ consider 
  the conditional probabilities 
$
\mu_M = \PP( \cdot\mid  \bar B_M(\mathbb{C}^n))$ and $ \bar\mu_M =\PP(\cdot\mid \C^n\setminus \bar B_M(\mathbb{C}^n)).
$
  Then 
 \begin{equation}\label{split1}
 \cD(a(\tau;\mu))=A_M\cD( a(\tau;\mu_M) )+\bar A_M\cD( a(\tau;\bar\mu_M)), %\quad A_M = \mu( \bar B_M(\mathbb{C}^n)), \;  \bar A_M=
 %\mu(  \C^n\setminus \bar B_M(\mathbb{C}^n))
 \end{equation} 
 where $ A_M = \mu( \bar B_M(\mathbb{C}^n))$ and  $\bar A_M= \mu(  \C^n\setminus \bar B_M(\mathbb{C}^n))$
 (cf. \eqref{decompp}). As $\EE |a(0)|^{2m_0'} \le M'$, then $ \bar A_M = \PP\{ a(0) >M\}
 \leqslant M'/M^{2m_0'}$. 
 Since eq.  \eqref{effective-equation} is assumed to be mixing, then
 $
 \|\cD (a(\tau; 0)) -\mu^0\|_L^* \le \bar g(\tau),
 $
 where $\bar g\ge0$ is a continuous function, going to 0 as $\tau\to\infty$. So in view of \eqref{x1}, 
 $$
  \|\cD (a(\tau; v)) -\mu^0\|_L^* \le    {\mathfrak g}_M(\tau) +  \bar g(\tau)=:  \tilde g_M(\tau), \qquad \forall\,v\in \bar B_M(\C^n).
 $$
  From here, 
 \[
 \begin{split}
  \|\cD (a(\tau;\mu_M)) -\mu^0\|_L^*  =   \| \int\big[ \cD (a(\tau; v)) \big] \mu_M(dv)   -\mu^0\|_L^*  \le \int \| \cD (a(\tau; v)) -\mu^0\|_L^* \mu_M(dv) \le
   \tilde g_M(\tau).
 \end{split}
 \]
  Therefore due to \eqref{split1},
  \[ \begin{split}
 \|\cD(a(\tau;\mu))-&\mu^0\|_L^* \leqslant 
  A_M   \|\cD(a(\tau;\mu_M))-\mu^0\|_L^* + \bar A_M   \|\cD(a(\tau;\bar\mu_M))-\mu^0\|_L^* \\
 & \le     \|\cD(a(\tau;\mu_M))-\mu^0\|_L^* + 2\bar A_M   \le
  \tilde g_M(\tau) + 2
  \frac{M'}{M^{2m'_0}} \quad\text{for any $ M>0$ and  $\tau\ge 0 $. }
  \end{split}
  \]  
 Let $M_1(\tau)>0$ be a continuous non-decreasing function, growing to infinity with $\tau$, and 
   such that $\tilde g_{M_1(\tau)}(\tau)\to0$ as $\tau\to\infty$
 (it exists since $\tilde g_M(\tau)$ is a continuous function of $(M,\tau)$, going to 0 as $\tau\to\infty$  for each fixed $M$). 
   Then 
 \begin{equation}\label{y2}
 \|\mathcal{D}(a(\tau;\mu))-\mu^0\|_L^*\leqslant 2  \frac{M'}{M_1(\tau)^{2m_0'}}+ \tilde g_{M_1(\tau)}(\tau)=:
 {\hat g}_{M'}(\tau).
 \end{equation}
 Obviously $ {\hat g}_{M'}(\tau)\ge 0$ is a continuous function on $\R_+^2$ ,  converging to $0$ as $\tau\to\infty$.

     \smallskip
 iii)   Now we examine the l.h.s of \eqref{mixing-quantify} as a function of $\tau$ and $d$.  Recall that the Kantorovich distance between measures 
 $\nu_1, \nu_2$ on $\C^n$ is 
 $$
 \| \nu_1 -\nu_2\|_K = \sup_{\text{Lip}\, f\le1} \lan f, \nu_1\ran - \lan f,\nu_2\ran \le \infty. 
 $$
 Obviously $  \| \nu_1 -\nu_2\|_L^* \le  \| \nu_1 -\nu_2\|_K$. Due to  \eqref{x0} 
 and  the assumption on
 $\mu$,  the $2m_0'$-moments of $\mu$ and $\mu^0$ are bounded by $M'\vee C_{m_0'}$. Thus 
 \be\label{x7}
 \| \mu- \mu^0\|_K \le \tilde{C} (M'\vee C_{m_0'}) ^{\gamma_1} d^{\gamma_2}:=D, \quad \gamma_1 = \tfrac{1}{2m'_0} , \; 
 \gamma_2 = \tfrac{2m'_0-1}{2m_0'} ,
 \ee
   see  \cite[Section 11.4]{BK} and    \cite[Chapter 7]{Vil}.  So by the Kantorovich--Rubinstein theorem (see \cite{Vil, BK})
  there exist r.v.'s $\xi$ and $\xi_0$, defined 
 on a new probability space $(\Omega', \cF', \PP')$,  such that $\cD( \xi )= \mu$,  $\cD (\xi_0) = \mu^0$ and 
 \be\label{x77}
 \EE\,  |\xi_1 -\xi_0| = \| \mu- \mu^0\|_K.
 \ee
 Then using \eqref{x33} and denoting by $ a_{st}(\tau)$ a stationary solution of eq. \eqref{effective-equation},  $\cD( a_{st}(\tau))\equiv \mu^0$, 
   we have:
 \[
 \begin{split} 
 \| \cD (a(\tau) )- \mu^0\|_L^* =  \| \cD (a(\tau; a(0)) )- \cD (a_{st}(\tau) ) \|_L^* \le
 \EE^{\om'} \| \cD( a(\tau; \xi^{\om'}) )- \cD (a(\tau; \xi_0^{\om'} ))  \|_L^* \\
 \le  \EE^{\om'} g_{\bar M}^2(\tau, |\xi^{\om'} -\xi_0^{\om'}|), \qquad {\bar M}={\bar M}^{\om'}=|\xi^{\om'}| \vee  |\xi_0^{\om'}| .
  \end{split} 
 \]
As $\EE ^{\om'} {\bar M}^{2m'_0} \le 2 (M'\vee C_{m_0'})$ by \eqref{8.05} and the assumption on $\mu$, then denoting 
$Q'_K =\{ \bar M\ge K\} \subset \Omega'$, for any $K>0$
 we have 
$$
\PP^{\om'} (Q'_K) \le  2K^{-2m_0'} (M'\vee C_{m_0'}). % \qquad \forall\, K>0. 
$$
Since $g_M^2 \le 2$ and for $\om' \notin Q'_K$ we have $ |\xi^{\om'}| , |\xi_0^{\om'}| \le K$,  then 
 $$
 \| \cD (a(\tau) )- \mu^0\|_L^*  \le 4K^{-2m_0'} (M'\vee C_{m_0'}) + \EE ^{\om'} g_K^2(\tau,  |\xi^{\om'} -\xi_0^{\om'}|).
 $$ 
 Now let $\Omega'_r =\{  |\xi^{\om'} -\xi_0^{\om'}| \ge r\}$. Then by \eqref{x77} and 
 \eqref{x7},  $\PP^{\om'} \Omega'_r\le D r^{-1}$. So
 \be\label{x8}
 \| \cD (a(\tau)) - \mu^0\|_L^*  \le 4K^{-2m'_0} (M'\vee C_{m'_0}) +  2D r^{-1} + g_K^2(\tau,r) \qquad \forall\, \tau\ge0, \; \forall\, K, r>0.
  \ee

     \smallskip
 iv) End of the proof. 
  Let $g_0(s)$ be a positive continuous function on $\mathbb{R}_+$ such that $g_0(s)\to\infty$ as
   $s\to+\infty$ and $|C_{m_0'}\big(g_0(s)\big)(\ln\ln s)^{-1/2}|\leqslant 2C_{m_0'}(0)$ for $s\geqslant3$.
   Taking   $r=D^{1/2}$ and choosing in the r.h.s. of \eqref{x8}  $K=g_0(r^{-1})$, we denote the  r.h.s 
  as $g_{M'}^3(\tau,r)$ (so we substitute in \eqref{x8}     $D=r^2$ and $ K=g_0(r^{-1})$). 
 By \eqref{x8} and  the definition of $g_M^2$ (see \eqref{x3} and \eqref{x33}), we have 
 \[
  \begin{split}
 % \| \cD (a(\tau)) &- \mu^0\|_L^* \le
 g_{M'}^3(\tau,r) &\leqslant 4(g_0(r^{-1}))^{-2m'_0}(M'\vee C_{m'_0})+2r \\
 &+2(C_{m'_0}(g(r^{-1}))+1)
 \big(\ln\ln(r^{-1}\vee 3)\big)^{-1}
 +r\exp\big(C\tau\ln\ln (r^{-1}\vee3)\big).
 \end{split}
 \]
 By the choice of $g_0$, as $r\to0$ the first, the second  and the fourth terms converge $0$. The third term 
  is $\leqslant 4(C_{m'_0}(0)+1)(\ln\ln(r^{-1}))^{-1/2}$ for $r\leqslant\frac{1}{3}$, so it also  converges  to zero with $r$.   
Hence  $g_{M'}^3(\tau,r)$ defines a continuous function on $\mathbb{R}_+^3$, vanishing with $r$.  Using the expression for $D$ in
\eqref{x7} let us write $r=D^{1/2}$ as $r=R_{M'}(d)$, where $R$ is a continuous function $\R_+^2 \to \R_+$,  non-decreasing in $d$ 
and vanishing with $d$. Setting 
$
g^4_{M'} (\tau, d) = g^3_{M'} (\tau, R_{M'}(d \wedge 2))
$
we get from the above that 
$$
 \| \cD( a(\tau)) - \mu^0\|_L^* \le g^4_{M'} (\tau, d ) \quad\text{if} \;\; \| \mu -\mu^0\|_L^*  \le d\le2.
$$
Finally, evoking \eqref{y2} we arrive at \eqref{mixing-quantify} with $g=g^5$, where 
$$
g^5_{M'} (\tau, d) = g^4_{M'} (\tau, d) \wedge {\hat g}_{M'}(\tau) \wedge 2. 
$$
The function $g^5$ is continuous, vanishes with $d$ and converges to zero as $\tau\to\infty$.
For any fixed ${M'}>0$ this convergence is uniform in $d$ due to the term $\hat{ g}_{M'}(\tau)$. So for a fixed ${M'}>0$ the function $(\tau,d)\mapsto g_{M'}^5(\tau,d)$ extends to a continuous function on the compact set $[0,\infty]\times[0,2]$, where it vanishes when $\tau=\infty$. Thus $g_{M'}^5$ is uniformly continuous in $d$, and the   assertion of the proposition is  proved.

%\newpage

\section{Averaging  for systems with general noises}\label{s-ef-eq}
In this section we sketch a proof of Theorem \ref{average-thm} for equations \eqref{1} with general  stochastic terms
$
\sqrt\eps\, \cB(v)\,dW.
$
 The proof follows  the argument in Section~\ref{s_4} with an extra difficulty which appears in the
case of equations  with  non-additive degenerate noises. 

Let us consider $v$-equation \eqref{v-equation-slowtime} with a general noise (possibly non-additive) and decomplexify it by writing the components 
$v_k(\tau)$ as  $(\tilde v_{2k-1}(\tau), \tilde v_{2k}(\tau))\in\mathbb{R}^2$, $k=1, \dots, n$. 
Now a solution $v(\tau)$ is a vector in $\R^{2n}$ and the equation reads
\be\label{E0}
dv(\tau) +\eps^{-1} Av(\tau) \,d\tau =  P(v(\tau))\,d\tau +  \cB(v(\tau)) \,d\beta(\tau), \qquad v(0) =v_0\in\mathbb{R}^{2n}. 
\ee
Here $A$ is the block-diagonal matrix as in Section \ref{s_1.1}, $\cB(v)$ is a real $2n\times n_2$-matrix and 
 $\beta(\tau) = (\beta_1(\tau) , \dots, \beta_{n_2}(\tau))$, where $\{\beta_j(\tau)\}$ are independent  standard real Wiener processes. 
 Note that in the real coordinates  in $  \R^{2n} \simeq\C^n$ the  operator $\Phi_w$ in \eqref{Phi} with some vector 
  $w\in\R^n$ is given by a 
 block-diagonal matrix, where for $j=1,\dots, n$ its  $j$-th diagonal block is the $2\times2$-matrix of rotation by the angle $w_j$. 
 \smallskip

     In this section we assume that
  
   \begin{assumption}\label{a_7}
   The dirft 
   $P\in$ Lip$_{m_0} (\R^{2n}, \R^{2n})$,  the matrix-function $\cB(v)$  belongs to
Lip$_{m_0} \big(\R^{2n}, M(2n\times n_2)\big )$,  equation \eqref{E0}
is well-posed and its solutions satisfy \eqref{moment1}.
  \end{assumption}

Passing to the interaction representation
$
v(\tau) = \Phi_{\tau \eps^{-1} \Lambda} a(\tau) 
%P(v(\tau))\,d\tau +  \Phi_{\tau \eps^{-1} \Lambda} \cB(v(\tau)) \,d\beta(\tau),
$
  we rewrite the equation as 
\be\label{E}
da(\tau) = \Phi_{\tau \eps^{-1} \Lambda} P(v(\tau))\,d\tau +  \Phi_{\tau \eps^{-1} \Lambda} \cB(v(\tau)) \,d\beta(\tau),
\qquad a(0)=v_0.
\ee
As in Section~\ref{s_4} we will see that as $\eps\to0$, the asymptotic  behaviour of distributions of the equation's solutions is described by an 
effective equation. As before,  the effective drift  is $\llangle P\rrangle(a)$.  
To calculate the effective dispersion, as in the proof of Lemma \ref{weak-m-s-lm} we consider the martingale 
\[
N^{Y,\varepsilon}:=a^{\varepsilon}(\tau)-\int_0^\tau Y(a^\varepsilon(s),s\varepsilon^{-1})ds=v_0+\int_0^\tau\cB^{\Lambda}(a^{\varepsilon}(s);s\varepsilon^{-1})d\beta(s),
\]
where $Y$ is defined in \eqref{Y} and
$\cB^{\Lambda}(a;t)=\Phi_{t\Lambda}\cB(\Phi_{-t\Lambda}a)$. % and it is   is a martingale.
  By  It\^o's formula, for $i,j=1,\dots, n$ the process 
 \[
 N_i^{Y,\varepsilon}(\tau) N_j^{Y,\varepsilon}(\tau)-\int_0^\tau \cA_{ij}^\Lambda(a^{\varepsilon}(s);s\varepsilon^{-1})ds,
 \qquad 
 (\cA_{ij}^\Lambda(a; t ))=\cB^{\Lambda}(a; t )\cB^{\Lambda*}(a; t ),
 \]
 where $\cB^{\Lambda*}$ is the transpose of $\cB^\Lambda$, also is a martingale. 
 By a straightforward analogy of Lemma~\ref{average-lm}, 
   the limit 
   $$\cA^0(a):=\lim_{T\to\infty}\frac{1}{T}\int_0^T\cA^{\Lambda}(a;t)dt$$
    exists 
   and belongs to $\text{Lip}_{2m_0}(\mathbb{R}^{2n},M(2n\times2n))$. 
   Then, by analogy with Lemma \ref{key-a-lemma}, 
   $$
   \mathbf{E}\Big|\int_0^{\tau}\cA^{\Lambda}(a^{\varepsilon}(s);s\varepsilon^{-1})ds-\int_0^\tau\cA^0(a^{\varepsilon}(s))ds\Big|\to0,\; \text{ as }\varepsilon\to0,
   $$
 for any $\tau\geqslant0$. From here, as in Section \ref{s_4}, we conclude that now for 
  the effective diffusion we should take 
 $\cA^0(a)$, which is a non-negative symmetric matrix. Denoting by $\cB^0(a)=\cA^0(a)^{1/2}$ 
 its principal square root, as in Section \ref{s_4} we verify  that any limiting measure $Q_0$ as in \eqref{precomp} 
 is a solution for the martingale problem for the effective equation 
 \begin{equation}\label{E2}
 da(\tau)-{\llangle P\rrangle}(a(\tau))d\tau=\cB^0(a(\tau))d\beta(\tau),\qquad a(0)=v_0,
 \end{equation}
 and so is a  weak solution  of the equation. 
 If the noise in equation \eqref{E0} is additive, then $\cB^0$ is a constant matrix, eq. \eqref{E2} has a unique solution and Theorem \ref{average-thm} with (modified) effective equation \eqref{E2} remains true for solutions of eq. \eqref{E}. In particular, the theorem applies to equation \eqref{v-equation1} with general additive random forces \eqref{general} (but then the effective dispersion matrix is given by a more complicated formula than in Section~\ref{s_4}). 
 
 Similarly, if the diffusion in \eqref{E0} is  non-degenerate in the sense that 
 \begin{equation}\label{E3}
 | \cB(v) \cB^*(v)  \xi|\geqslant \alpha|\xi|,\qquad  \forall\, v,\;  \forall\,
 \xi\in\mathbb{R}^{2n},
 \end{equation}
 for some $\alpha>0$, then the matrix $\cB^{\Lambda}(a, \tau)$ also satisfies \eqref{E3} for all $a$ and $\tau$, i.e.
  $\langle \cA^{\Lambda}(a;s\varepsilon^{-1})\xi,\xi\rangle\geqslant \alpha |\xi|^2$. Thus 
 $\cA^0(a)\geqslant   \alpha \mathbb{I}$,
 and so $\cB^0(a)=\cA^0(a)^{1/2}$ is a locally Lipschitz matrix function of $a$ 
 (e.g. see \cite[Theorem 5.2.2]{SV}). So again eq. \eqref{E2} has a unique solution and Theorem~\ref{average-thm} remains true for eq. \eqref{E0} (with the effective equation of the form \eqref{E2}). 
 
 To treat equations \eqref{E0} with degenerate non-additive noises we write the matrix $\cA^0(a)$ as
 \[\cA^0(a)=\lim_{T\to\infty}\frac{1}{T}\int_0^T\big(\Phi_{t\Lambda}\cB(\Phi_{-t\Lambda}a)\big)\cdot \big(\Phi_{t\Lambda}\cB(\Phi_{-t\Lambda}a)\big)^*dt.\]
 By the same reason as in Proposition  \ref{C-m-smooth-average}  we find that 
 \[|\cA^0|_{C^2(B_R)}\leqslant C|\cB|_{C^2(B_R)}^2,\quad \forall R>0.\]
 Now using \cite[Theorem 5.2.3]{SV}, we get that 
 \begin{equation}\label{E4}
 \text{Lip}(\cB^0(a)|_{\bar B_R})\leqslant C|\cA^0|_{C^2(B_{R+1})}^{1/2}\leqslant C_1|\cB|_{C^2(B_{R+1})},\quad \forall R>0.
 \end{equation}
 So the matrix-function $\cB^0(a)$ is locally Lipschitz continuous,  eq. \eqref{E2} has a unique solution and 
  the assertion of Theorem \ref{average-thm} remains true for  eq.~\eqref{E0}. We have obtained
 \begin{theorem}\label{t_7}
 Suppose that  Assumption \ref{a_7} holds and for the matrix function $\cB(v)$ in \eqref{E0} one of the following three options  is true:
 \begin{enumerate}
 \item it is $v$-independent;
 \item it satisfies the non-degeneracy  condition  \eqref{E3};
 \item it  is a $C^2$-smooth matrix-function of $v$. 
 \end{enumerate}
 Then
 for any $v_0\in \mathbb{R}^{2n}$ a solution $a^\varepsilon(\tau;v_0)$ of eq. \eqref{E} satisfies 
 $$\cD(a^{{\ep}}(\cdot;v_0))\rightharpoonup  Q_0\text{ in }\mathcal{P}(C([0,T],\mathbb{C}^n))\quad  \text{ as }{\ep}\to0,
 $$ 
 where $Q_0$ is the  law of a unique weak solution of the effective equation \eqref{E2}.  
 \end{theorem}
 An obvious analogy of Corollary \ref{c_4.7} holds for solutions of eq. \eqref{E0}. 
  %{\color{red} 
%More references, needed below: \cite{SK_GAFA, Yor} }

%\newpage

     \section{A sufficient condition for Assumptions \ref{assume-wp1}, \ref{assume-mixing} and \ref{a_8.1}}\label{s_9}
   In this section we  derive a condition which implies Assumptions \ref{assume-wp1}, \ref{assume-mixing} and \ref{a_8.1}. So when it is met,  all  theorems in Sections~\ref{s_4},~\ref{s_5} and~\ref{s_8} apply to eq.~\eqref{v-equation-slowtime}.

    Consider a stochastic differential equation on $\mathbb{R}^l$,
    \begin{equation}\label{x-equation-r}
    dx =b( x)d\tau+\sigma( x)d \beta (\tau),\;\;\; x\in\mathbb{R}^l,\; \tau\geqslant0,
    \end{equation}
    where  $\sigma( x)$ is an $l\times k$-matrix   and $\beta (\tau) $ is a     standard 
     Wiener processes in $\R^k$.      We assume

    \begin{assumption}\label{assume-lip-x} The drift $b( x)$ and dispersion    $\sigma( x)$ are locally Lipschitz in $x$, and 
     $
     \cC^m(b)$,   $\cC^m(\sigma) \le C<\infty
     $
     for some $m\ge0$.
    \end{assumption}

  The diffusion
   $a( x)= \sigma( x) \cdot \sigma^t( x)$ is a  nonnegative symmetric   $l\times l$-matrix. Consider the 
    differential  operator
   $$
   \mathscr{L} \big(v(x)\big)=\sum_{j=1}^lb_j( x)\frac{\partial v}{\partial x_j}+\frac{1}{2}\sum_{i=1}^l
   \sum_{j=1}^la_{ij}( x)\frac{\partial^2 v}{\partial x_i\partial x_j}.
   $$
    We have the following result concerning   the well-posedness of eq.~\eqref{x-equation-r} from   \cite[Theorem~3.5]{ khasminskii}:

    \begin{theorem}\label{lya-thm}
    Assume Assumption \ref{assume-lip-x} and suppose that  there exists a nonnegative function $V(x)\in C^2(\mathbb{R}^l)$ such that for some constant $c>0$ we have
    \begin{equation*}
    \mathscr{L}  (V(x))\leqslant cV(x)\quad \forall \tau\geqslant0,\; x\in\mathbb{R}^l,
    \end{equation*}
    and
    \begin{equation*}
     \inf_{|x|>R}V(x)\to\infty\; \text{ as }R\to\infty.
    \end{equation*}
    Then for any $x_0 \in\mathbb{R}^l$       equation \eqref{x-equation-r} has  a unique strong solution $X(\tau)$ with the
    initial condition $X(0)=x_0$.  Furthermore, the process $X(\tau)$ satisfies
    $${\EE}V(X(\tau))\leqslant e^{c\tau} V(x_0) \quad  \forall \tau\geqslant0.$$
    \end{theorem}

     The function $V$ is called {\it a Lyapunov function} for  equation \eqref{x-equation-r}.  In its terms  a
     sufficient condition for the mixing in equation \eqref{x-equation-r} is given by the following statement:

    \begin{proposition}\label{mixing-sufficient}
    Assume that  in addition to Assumption \ref{assume-lip-x} we have \begin{enumerate}
    \item the drift $b$ satisfies 
        \begin{equation}
    \label{mix-condition1}\langle b(x), x\rangle\leqslant-{\alpha_1}|x|+{\alpha_2} \qquad  \forall x\in\mathbb{R}^l,
    \end{equation}
    for some   constants ${\alpha_1}>0$ and ${\alpha_2}\geqslant0$, where $\langle\cdot,\cdot\rangle$ is the standard inner product in $\mathbb{R}^l$.
    \item The diffusion matrix $a(x)=\sigma(x)\sigma(x)^t$ is uniformly non-degenerate, that is
    %there exists $\gamma_1\geqslant\gamma_2>0$ such that
    \begin{equation}\label{mix-condition2}
    \gamma_2\mathbb{I}\leqslant a(x)\leqslant \gamma_1 \mathbb{I}\quad \forall x\in\mathbb{R}^l,
    \end{equation}
    for some  $\gamma_1\geqslant\gamma_2>0$.
    \end{enumerate}
    Then  for any $c'>0$
      equation \eqref{x-equation-r} has a smooth      Lyapunov function  $V(x)$, equal $ \exp( c' |x|)$ for $|x|\ge1$,
       estimate \eqref{v-upper-mix} holds true for its solutions for every $m\in\Nn$,
      and the   equation is mixing.
    \end{proposition}

    In Appendix \ref{proof-of-mixing}
    it is shown how to derive the proposition from abstract results    in  \cite{khasminskii}.
    Moreover, it can be proved that under the 
    assumptions of the proposition the equation is exponentially mixing and  \eqref{x11}  holds,
    see  Example~\ref{ex_8.3}. 
    \medskip

    Let us decomplexify $\C^n$ to $\R^{2n}$ and 
    identify  equation \eqref{v-equation-slowtime} with a real equation \eqref{x-equation-r}, where $l=2n$ (and $x=v$).  Then
    $b( v)\cong(b_j( v )=(-i{\ep}^{-1} \lambda_jv_j+P_j(v),\;j=1,\dots,n)$, where $b_j \in \C \cong \R^2 \subset \R^{2n}$. Since in  complex
    terms  the real inner product reads as $\langle v,w\rangle=Re\sum v_j\bar w_j$, then
    \[\langle b( v),v\rangle=\langle P(v),v\rangle.\]
    So for  equation \eqref{v-equation-slowtime}  condition \eqref{mix-condition1} is equivalent to
    \begin{equation}\label{mix-v-condition1}
    \langle P(v),v\rangle\leqslant -{\alpha_1}|v|+{\alpha_2},\quad \forall v\in\mathbb{C}^n,
    \end{equation}
    for some constants ${\alpha_1}>0$ and ${\alpha_2}\geqslant0$.

    Now consider  effective equation \eqref{effective-equation}.      Since in  eq. \eqref{a-equation2} the drift is
    $$
    \Ye(a,\tau{\ep}^{-1})= (\Phi_{\tau{\ep}^{-1}\Lambda})_*P( a),
    $$
    then under the assumption  \eqref{mix-v-condition1}  we have
    \[\langle \Ye(a,\tau{\ep}^{-1}),a\rangle=
    \langle P(\Phi_{\tau{\ep}^{-1}\Lambda}a),\Phi_{\tau{\ep}^{-1}\Lambda}a\rangle\leqslant -
    \alpha_1 |\Phi_{-\tau{\ep}^{-1}\Lambda}a|+{\alpha_2}=-{\alpha_1}|a|+{\alpha_2}, \]
    for all ${\ep}$.     Therefore $ \llangle P\rrangle$ satisfies
    $$
    \langle {\llangle P\rrangle} (a),a\rangle=\lim_{T\to\infty}\frac{1}{T}\int_0^T\langle \Ye(a,\tau{\ep}^{-1}),a\rangle d\tau\leqslant-{\alpha_1}|a|+{\alpha_2}.
    $$
    We saw that  assumption  \eqref{mix-v-condition1}   implies the validity of  condition \eqref{mix-condition1} also  for the effective equation.

     As it was pointed out, if the dispersion
     matrix $\Psi$ is non-degenerate, then the  dispersion  $B$ in  the effective equation is non-degenerate as well.
     The corresponding diffusion matrix also is  non-degenerate and condition \eqref{mix-condition2} holds for it. 
     Thus  we obtained the following statement:

\begin{proposition}\label{p_6.4}
If in  equation \eqref{v-equation-slowtime} the dispersion matrix $\Psi$ is non-degenerate, the drift  $P\in \text{Lip}_{m_0}$ for some $m_0\in\mathbb{N}$ and \eqref{mix-v-condition1} holds
for some constants ${\alpha_1}>0$ and ${\alpha_2}\geqslant0$,
then the assumption of Theorem~\ref{thm-limit-mixing} holds true, and so does the assumption of Theorems~\ref{average-thm} and \ref{thm-uniform}.
 \end{proposition}

    \appendix

\section{Proof of Proposition \ref{mixing-sufficient}}\label{proof-of-mixing}
By  condition \eqref{mix-condition2} the diffusion   $a$ is uniformly bounded. So there exist constants  $k_1, k_2>0$ such that
    \begin{equation}\label{a-upper-ap}
    \text{Tr}(a(x))\leqslant k_1,\; \|a(x)\|\leqslant k_2 \quad \forall x\in\mathbb{R}^l.
    \end{equation}
    Take $V(x)=e^{{c'}f(x)}$, where ${c'}>0$ is a  positive constant and $f(x)$ is a non-negative smooth function, equal $|x|$ for $|x|\ge1$, 
    and such that its first and second derivatives are bounded by 3.  Then
    \[\frac{\partial V(x)}{\partial x_j}={c'} V(x)\partial_{x_j}f(x),\quad \frac{\partial^2V(x)}{\partial x_j\partial x_j}={c'}V(x)\partial_{x_ix_j}f(x)+{c'}^2V(x)\partial_{x_i}f(x)\partial_{x_j}f(x).\]
  Therefore  we have that
        \[\mathscr{L}\big(V(x)\big)={c'} V(x) \mathcal{K}(c',x), \]
        where $$\mathcal{K}(c',x)=\sum_{j=1}^lb_j(x)\partial_{x_j}f(x)+\frac{1}{2}\sum_{ij}a_{ij}(x)\partial_{x_ix_j}f(x)+\frac{1}{2}c'\sum_{ij}a_{ij}(x)\partial_{x_i}f(x)\partial_{x_j}f(x).$$
        From the condition \eqref{mix-condition1} and \eqref{a-upper-ap} it is evident that
        \begin{equation}\label{upper-K}
        \begin{cases}|\mathcal{K}(c',x)|\leqslant (c'+1)C, &\text{if } |x|<1,\\
        \mathcal{K}(c',x)\leqslant (-\alpha_1+\frac{\alpha_2}{|x|}+\frac{C}{|x|}+c'C), &\text{if }|x|\geqslant1,
        \end{cases}\end{equation}
        where $C>0$ is a constant depending on $k_1$, $k_2$ and $\sup_{|x|\leqslant1}|b(x)|$.
    Then we achieve that
    \[\mathscr{L}\big(V(x)\big)\leqslant cV(x)\quad \forall x\in\mathbb{R}^l,\]
    where $c=c'(\alpha_2+(c'+1)C)$. Clearly $\inf_{|x|>R}V(x)\to\infty$ as $R\to\infty$. So
    $V(x)$ is a Lyapunov function for  equation \eqref{x-equation-r}. Then  by  Theorem~\ref{lya-thm} for any $x_0\in \R^l$   the
     equation has a unique solution $X(\tau)=X(\tau;x_0)$, equal $x_0$ at $\tau=0$, which satisfies
    \[
    {\EE}e^{{c'}f(X(\tau))}\leqslant e^{c\tau} e^{{c'}f(x_0)}\quad\forall \tau\geqslant0.
    \]

    Let us     apply  It\^o's formula to the process
    $F\big(X(\tau)\big)=e^{\eta'f(X(\tau))}$, where   $ 0<\eta'\leqslant\frac{1}{2}{c'}$ is a constant to be determined later.
    Then {\color{red}
    \[\begin{split} 
    d  F\big( X \big)&= \mathscr{L}\big(F( X )\big)d\tau + \eta' F(X) \lan\nabla f(x), \sigma^t(  X) dW\ran\\
 &=\eta' F(X)\mathcal{K}(\eta',X) d\tau +  \eta' F(X) \lan\nabla f(x), \sigma^t(  X) dW\ran.  
 \end{split}\] }
   By \eqref{upper-K}, choosing $\eta'=\min\{\frac{\alpha_1}{4C},\frac{c'}{2}\}$, we have
    $$
     F(X)\mathcal{K}(\eta',X)\le -\frac{\alpha_1}{2}F(X) +C_0(\alpha_1, \eta',k_1,k_2)
    $$
    uniformly in  $X$. Then {\color{red}
   \begin{equation}\label{exp-diff-eq}
   d F\big( X \big) \leqslant \big( -\frac{\alpha_1}{2}\eta' F\big( X \big)+C_0\big)d\tau + \eta' F(X) \lan \nabla f(x),\sigma^t(  X) dW\ran,
   \end{equation} }
  where the constant $C_0>0$ depends on $k_1$, $k_2$, $\alpha_1, \eta'$ and  $\alpha_2$. Taking  expectation and applying   Gronwall's lemma
  we obtain that {\color{red}
    \begin{equation}\label{exp-bound}
     {\EE}e^{\eta'f(X(\tau))}\leqslant e^{-\frac{\alpha_1}{2}\eta'\tau} e^{\eta'f(x_0)}+C_1,\quad
    \tau\geqslant0,
    \end{equation} }
    where   $C_1>0$ depends on the same parameters as $C_0$.

    Now let us take any $T\geqslant0$ and for 
     $\tau\in[T,T+1]$ consider relation \eqref{exp-diff-eq}, where $F(X)$ is replaced by
    $\tilde F(X) = e^{ \tilde\eta f(X)}$ with $0<\tilde\eta \le \tfrac12 \eta'$, and integrate it from $T$ to $\tau$:
    \be\label{ccc}
    \begin{split}\tilde F(X(\tau))&\le  \tilde F(X(T))+C_0 + \tilde{\eta} \int_T^\tau
     \tilde F(X) \lan \nabla f(x),\sigma^t(s, X)dW\ran\\
     &=: \tilde F(X(T)) + C_0+\M(\tau). \end{split}
    \ee
    Due to estimate \eqref{exp-bound}, $\M(\tau)$ is a continuous square-integrable martingale. Therefore by Doob's inequality,
    \[\begin{split}
    \EE \sup_{T\le \tau\le T+1} |\M(\tau)|^2 \le 4 \EE |\M(T+1)|^2 \le C \int_T^{T+1} \EE \tilde F^2 (X(s))ds
    \le C \int_T^{T+1} \EE F (X(s))ds \le C',
    \end{split}
    \]
    where $C'$ depends on $k_1$, $k_2$, $\alpha_1, \eta',\alpha_2$ and  $|x_0|$.  From here, \eqref{ccc} and \eqref{exp-bound} it follows that {\color{red}
    $$
     \EE \sup_{T\le \tau\le T+1} e^{ \tilde\eta f(X(\tau))} \le C^{''},
    $$ }
    where $C^{''}$ depends on the same parameters as $C'$. This bound implies   that solutions $X(\tau)$ satisfy estimate
     \eqref{v-upper-mix} in Assumption~\ref{assume-mixing} for every $m\ge0$.

     To prove the proposition it remains to show that under the imposed assumptions   equation \eqref{x-equation-r} is mixing.  Due to
\cite[Theorem 4.3]{ khasminskii} we just need to verify  that there exists an absorbing 
 ball $B_R=\{|x|\leqslant R\}$ such that for any compact set $K\subset\mathbb{R}^l\setminus B_{R}$
 \begin{equation}\label{tau-stop-time}
 \sup_{x_0\in K}\EE \tau(x_0)<\infty,
 \end{equation}
 where $\tau(x_0)$ is the hitting time for $B_R$ of  trajectory $X(\tau; x_0)$.
   Indeed, let $x_0\in K\subset \mathbb{R}^l\setminus B_{R}$ for some $R>0$ to be determined later. 
   We set $\tau_M:=\min\{\tau(x_0), M\}$, $M>0$.
  Applying  It\^o's formula to the process $F(\tau,X(\tau))=e^{\frac{1}{4}\eta'\alpha_1\tau}|X(\tau)|^2$ and using 
   \eqref{exp-bound} we find that {\color{red}
 \[
 d F(\tau, X(\tau))= \Big( \frac{\eta'\alpha_1}{4}F(\tau,X(\tau))+\mathscr{L}\big(F(\tau,X(\tau))\big)\Big)d\tau+d\mathcal{M}(\tau),
 \] }
 where $\mathcal{M}(\tau)$ is a corresponding stochastic integral.
By \eqref{a-upper-ap},  \eqref{exp-bound}  and \eqref{mix-condition1}, we have
\[{\EE}e^{\frac{\eta'\alpha_1}{4}\tau_M}|X(\tau_M)|^2+{\EE}\int_0^{\tau_M}e^{\frac{\eta'\alpha_1 s}{4}}(2\alpha_1|X(\tau)|-C_3)ds
\leqslant |x_0|^2+2e^{\eta'f(x_0)}=:\gamma(x_0),\]
where $C_3>0$ depends  on $\alpha_1$, $\alpha_2$, $k_1$ and $k_2$.  Since $|X(s)|\geqslant R$
for $0\leqslant s\leqslant \tau_M$,   then
$$
{\EE}\Big(C_3 \int_0^{\tau_M}e^{\frac{\eta'\alpha_1 s}{4}}ds\Big)\leqslant \gamma(x_0),
$$
if $R\ge {C_3/\alpha_1}$.
Therefore ${\EE} \tau_M  \leqslant {\gamma(x_0)}/{C_3}$.
 Letting $M\to\infty$ we  verify \eqref{tau-stop-time} for  $R\ge {C_3/\alpha_1}$.
This completes  the proof of Proposition \ref{mixing-sufficient}.

  \section{Representation of martingales}\label{r-o-m}

    Let $\{M_k(t),\;t\in[0,T]\}$, $k=1,\dots,d$, be continuous square-integrable 
   martingales on a filtered probability space $(\Omega,\cF,\mathbf{P},\{\cF_t\})$. We recall that their brackets (or their cross-variational process)
   is an  $ \{\cF_t\}$-adapted continuous matrix-valued process of bounded variation  $\langle M_k,M_j\rangle(t), 1\le k,j\le d$, a.s. vanishing at 
   $t=0$, and such that for all $k,j$ the process 
   $
    M_k(t) M_j(t) - \langle M_k,M_j\rangle(t)
   $
   is an  $ \{\cF_t\}$-martingale; see  \cite{brownianbook}, Definition~1.5.5 and Theorem~1.5.13.

 \begin{theorem}
 [\cite{brownianbook}, Theorem 3.4.2]
 Let $(M_k(t), 1\le k\le d)$ be a vector of martingale as above. Then  there exists an extension
  $(\tilde\Omega,\tilde{\mathcal{F}},\tilde{\mathbf{P}},\tilde{\mathcal{F}}_t)$  of the probability space
  on which are defined independent standard
    Wiener processes $W_1(t),\dots,W_d(t)$, and a measurable adapted
    matrix $X= (X_{k j}(t))_{k,j=1,\dots,d}$, $t\in[0,T]$    such that   $\EE \int_0^{T} \| X(s)\|^2 ds <\infty$, and 
     $\tilde{\mathbf{P}}$-a.s. we have the representations
 \[M_k(t)-M_k(0)=\sum_{j=1}^d\int_0^tX_{kj}(s)dW_j(s),\quad 1\leqslant k\leqslant d, \; t\in[0,T],\]
 and
 \[\langle M_k,M_j\rangle(t)=\sum_{l=1}^d\int_0^tX_{kl}(s)X_{jl}(s)ds,\quad 1\leqslant k,j\leqslant d, \; t\in[0,T].\]

 \end{theorem}
 Now let $(N_1(t),\dots, N_d(t))\in\mathbb{C}^d$ be a vector of  complex continuous  square-integrable 
 martingales. Then
 $N_j(t)=N^+_j(t)+i N^-_j(t)$, where   $\big(N^+_1(t),N^-_1(t)$, $\dots, N^+_d(t),N^-_d(t) \big)\in\mathbb{R}^{2d}$ is a vector of real continuous martingales.
 The brackets $\lan N_i, N_j\ran$ and $\lan  N_i,  \bar N_j\ran$ are defined by linearity. For example,
 $
 \lan N_i, N_j\ran = \lan N_i^+, N_j^+\ran - \lan N_i^-, N_j^-\ran +i \lan N_i,^+ N_j^-\ran+ i \lan N_i,^- N_j^+\ran.
 $\footnote{ There is no need to define  the brackets  $\lan \bar N_i,  \bar N_j\ran$ and  $\lan \bar N_i,  N_j\ran$ since these are just the processes, 
 complex-conjugated to  $\lan  N_i,   N_j\ran$ and $\lan  N_i,  \bar N_j\ran$, respectively.}
 Equivalently $ \lan N_i, N_j\ran$  may be be defined as  the only adapted continuous complex
 process of bounded variation, vanishing at zero, and such that
 $
 N_i N_j - \lan N_i, N_j\ran
 $
 is a martingale. The brackets  $\lan  N_i,  \bar N_j\ran$ may be defined similarly. The result above implies 
  a representation theorem for complex continuous martingales. Below we give  its special case, relevant for our work.

 \begin{corollary} Suppose that the brackets $\lan N_i, N_j\ran(t)$ and $\lan \bar N_i, \bar N_j\ran(t)$ all vanish, while the brackets
  $ \lan N_i, \bar N_j\ran(t)$, $1\le i,j\le d$, a.s. are absolutely continuous complex  processes.  Then there exist
  an adapted process $\Psi(t)$, valued in $d\times d$ complex matrices, satisfying $\EE \int_0^{T} \| \Psi(t)\|^2 dt<\infty$,
   and   independent standard complex Wiener processes   $\beta^c_1(t),\dots,\beta^c_d(t) $, all defined on an extension of the original
 probability space,   such that a.s.
 \[
 N_j(t)-N_j(0)=\sum_{k=1}^d\int_0^t\Psi_{jk}(s)d\beta^c_k(s) \quad \forall \, 0\le t \le T,
 %+\sum_{k=1}^d\int_0^t\Psi_2^{jk}(s)d\bar\beta_k^c(s),
 \; j=1,\dots,d.
 \]
 Moreover, $ \lan N_i, N_j\ran(t)  \equiv 0$ and 
 $$
 \lan N_i, \bar N_j\ran(t) =  2 \int_0^t (\Psi \Psi^*)_{ij} (s) ds, \quad 1\le i,j\le d.
 $$
 \end{corollary}

   \section{It\^o's formula for complex processes} \label{a_ito}
 Consider a complex It\^o process  $v(t)\in\C^n$, defined on a  filtered probability space:
 \be\label{ito1}
 dv(t) = g(t)\,dt   + M^1(t)\,dB(t) + M^2(t) \,d\bar B(t)\,.
 \ee
 Here $v(t)$ and $g(t)$ are   adapted processes in $\C^n$, $M^1$ and $M^2$
 are   adapted  processes in the space of complex $n\times N$--matrices,
 and $B(t)=(\beta^c_1(t),\dots,\beta^c_N(t))$,  $\bar B(t)=(\bar\beta^c_1(t),\dots,\bar\beta^c_N(t))$,
 where $\{\beta^c_j\}$ are  independent standard complex Wiener processes. 
 We recall that for a $C^1$--function $f$ on $\C^n$,% where $z_j=x_{j}+iy_{j}, 1\le j\le n$,  we write
 $\
 {\p f}/{\p z_j}=\frac12 \big( {\p f}/{\p x_{j}}-i  {\p f}/{\p y_{j}}\big) $
 and
 $
 {\p f}/{\p \bar z_j}=\frac12 \big( {\p f}/{\p x_{j}} + i  {\p f}/{\p y_{j}}\big).
 $
 If $f$ is a polynomial of $z_j$ and $\bar z_j$, then $\p f/\p z_j$ and  $\p f/\p \bar  z_j$ may be calculated as if $z_j$ and
 $\bar z_j$ are independent variables.

  The processes  $g, M^1, M^2$  and the function $f(t,v)$ in the theorem below are
 assumed to satisfy the usual conditions,  needed for the applicability of It\^o's formula (e.g. see in \cite{brownianbook}), 
  which we do not repeat here.

 \begin{theorem} \label{t_Ito}
 Let $f(t,v)$ be a $C^2$--smooth complex function. Then
 \be\label{c_ito}
 \begin{split}
& df(t,v(t)) =\Big[ \p f/\p t +d_vf(t,v) g + d_{\bar v} f(t,v) \bar g \\
 &+\Tr \big[\big( M^1 ( M^2)^t + M^2 ( M^1)^t\big) \frac{\p^2 f}{\p v \p v}
 + \big( \bar M^1 ( \bar M^2)^t + \bar M^2 ( \bar M^1)^t\big) \frac{\p^2 f}{\p \bar v \p \bar v}\\
&  + {2}\big( M^1( \bar M^1)^t + M^2 ( \bar M^2)^t\big) \frac{\p^2 f}{\p v \p \bar v}\big] \Big] dt\\
&+ d_vf(M^1 dB + M^2 d\bar B )+ d_{\bar v}f(\bar M^1 d\bar B + \bar M^2 d  B )\,.
 \end{split}
 \ee
 Here $d_vf(t,v) g =\sum \frac{\p f}{\p v_j}g_j $,  $d_{\bar v}f(t,v) \bar g =\sum \frac{\p f}{\p \bar v_j}\bar g_j $,
 $\frac{\p^2 f}{\p v \p v} $ stands for the matrix with  entries $ \frac{\p^2 f}{\p v_j \p v_k} $, etc.
 If the function $f$ is real valued, then $d_{\bar v}f( v) =\overline{d_{ v}f( v)} $,
  and  the It\^o term, given by the second and third lines of   \eqref{c_ito}, reeds 
 $$
 2\text {Re} \Tr \Big[ ( M^1 ( M^2)^t + M^2 ( M^1)^t ) \frac{\p^2 f}{\p v \p v}+
 ( M^1 ( \bar M^1)^t +  M^2 (\bar M^2)^t ) \frac{\p^2 f}{\p v \p \bar v}\Big].
 $$
 \end{theorem}

 To prove the result one may re-write $v(t)$ as an It\^o process in $R^{2d}$ in terms of the real Wiener processes Re\,$W_j(t)$, Im\,$W_j(t)$,
 apply the usual It\^o's formula to $f(t,v(t))$ and then re-write the result back in terms of the
 complex Wiener processes. Corresponding straightforward
 calculation is rather heavy and it is not easy to make it without mistake. Below we suggest a better way to get the formula.

 \begin{proof}
 The linear part of the formula \eqref{c_ito}, given by its first and forth
  lines, follows from the real case by linearity. It remains to prove that the
 It\^o term has the form, given by the second and third lines. From the real  formula we see that the  It\^o term is an expression,
 linear in $\p^2 f/\p v \p v$, $\p^2 f/\p \bar v \p \bar v$ and $\p^2 f/\p v \p \bar v$, with the coefficients, quadratic in the matrices $M^1$ and $M^2$.
 So it may be written as
 \be\label{hren}
 \Big[ \Tr \big( Q^1  \frac{\p^2 f}{\p v \p v}\big)  + \Tr \big( Q^2  \frac{\p^2 f}{\p \bar v \p \bar v}\big)  + \Tr \big( Q^3  \frac{\p^2 f}{\p v \p \bar v}\big)
 \Big]dt,
 \ee
 where $Q^j$ are complex $n\times n$-matrices, quadratic in $M^1, M^2$. We should show that they
 have the form, specified in
 \eqref{c_ito}. To do that we note that since  the processes $\beta^c_j$ are independent and have the form \eqref{complex}, then  for all $j,l$ their
 brackets have the following form:
 \be\label{brackets}
 \lan \beta^c_j, \beta^c_l \ran = \lan \bar\beta^c_j, \bar\beta^c_l \ran = 0, \quad
 \lan \beta^c_j, \bar\beta^c_l \ran = \lan \bar\beta^c_j, \beta^c_l \ran = 2 \delta_{j,l} t.
 \ee
 Now let in \eqref{ito1} $g=0, v(0)=0$ and $M^1, M^2$ are constant matrices. Then
 $$
 v(t) = M^1 B(t) + M^2 \bar B(t).
 $$
 Taking $f(v) = v_{i_1} v_{i_2}$ and using \eqref{brackets} we see that
 \[\begin{split}
 f(v(t)) &= \Big ( \sum_j M^1_{i_1 j} B_j(t)  + \sum_j M^2_{i_1 j} \bar B_j(t) \Big) \cdot
  \Big ( \sum_j M^1_{i_2 j} B_j(t)  + \sum_j M^2_{i_2 j} \bar B_j(t) \Big)\\
 & = \Big[ \big( M^1 (M^2)^t\big)_{i_1 i_2} +  \big( M^2 (M^1)^t\big)_{i_1 i_2} \Big]  2\,t +\text{ a martingale}.
  \end{split}
 \]
 Since due to \eqref{hren} the linear in $t$ part should equal $(Q_{i_1 i_2}+Q_{i_2 i_1}) t$, then
 $
 Q^1 = M^1 (M^2)^t +  M^2 (M^1)^t .
 $
 Similar, considering $f(v) = \bar v_{i_1} \bar v_{i_2}$ we find that
 $$
 Q^2 =\bar  M^1 (\bar M^2)^t +  \bar M^2 (\bar M^1)^t ,
 $$
while   letting $f(v) =  v_{i_1} \bar v_{i_2}$ leads to 
 $
2 \big( M^1 (\bar M^1)^t +   M^2 (\bar M^2)^t\big)_{i_1 i_2}  = Q^3_{i_1 i_2};
 $
 so
 $$
 Q^3 = 2 ( M^1 (\bar M^1)^t +   M^2 (\bar M^2)^t) .
 $$
 This completes the proof of \eqref{c_ito}.  The second assertion of the theorem follows by a straightforward calculation.
 \end{proof}

  \section{Projections to convex sets } \label{app_Lip}
  \begin{lemma}
  Let $\cB$ be a closed convex subset of a Hilbert space $X$ of finite or infinite dimension. Let $\cB$
   contains at least two points and let $\Pi:X\to \cB$ be the projection, 
  sending any
  point of $X$ to a nearest point of $\cB$. Then Lip$\,\Pi =1$. 
  \end{lemma}
  \begin{proof}
  Let $A,B \in X$  and let  $a=\Pi A\in\cB, b=\Pi B\in\cB$. If $A,B\in \cB$ then $a=A$ and $B=b$. So  Lip$\,\Pi \ge 1$ and 
   it remains to show that 
  $$
  \| a-b\| \le \| A-B\| \qquad \forall\, A\ne B. 
  $$
  If $a=b$ the assertion is trivial. Otherwise consider the vectors 
  $
  \xi=b-a, $\, %\overrightarrow{ab}$, \, 
  $ l^a=A-a$, \, %\overrightarrow{aA}$, \,  
  $l^b= B-b$, %\overrightarrow{bB} $
  and introduce in $X$ an orthonormal basis $(e_1, e_2, \dots)$ such that 
  $e_1 =\xi/|\xi|$. Then $\xi = (\xi_1, \xi_2, \dots)$, where $\xi_1 =| \xi|$ and $\xi_j=0$ for $j\ge2$. Since $a$ is a point in 
  $[a,b]\subset X$, the 
  closest to $A$,
  then $l^a_1 = l^a\cdot e_1 \le0$. Similar $l^b_1 \ge0$. Thus 
  $$
  \| B-A\| = \| \xi+l^b -  l^a\| \ge | \xi_1 + l_1^b - l_1^a| \ge \xi_1 = \| b-a\|,
  $$
  and the assertion is proved.
  \end{proof} 
  
  Note that an analogy of the lemma's statement for a Banach space $X$ in general is wrong.   

\section*{ Acknowledgement}

GH was supported by National Natural Science Foundation of China (No. 20221300605). Both authors were supported
 by the Ministry of Science and Higher Education of the Russian Federation (megagrant No. 075-15-2022-1115).

\bibliography{reference}{}

\begin{thebibliography}{10}

\bibitem{AKN}
V.~I. Arnold, V.~V. Kozlov, and A.~I. Neishtadt.
\newblock {\em Mathematical Aspects of Classical and Celestial Mechanics}.
\newblock Springer, third edition, 2006.

\bibitem{Bill}
P.~Billingsley.
\newblock {\em Convergence of Probability Measures}.
\newblock John Wiley \& Sons, New York, 1999.

\bibitem{BKRS}
V.I. Bogachev, N.~V. Krylov, M.~Rockner, and S.V. Shaposhnikov.
\newblock {\em Fokker-Plank-Kolmogorov equations}.
\newblock AMS Publications, 2015.

\bibitem{BM}
N.~N. Bogoliubov and Y.~A. Mitropolsky.
\newblock {\em Asymptotic Methods in the Theory of Non-linear Oscillations}.
\newblock Gordon and Breach, New York, 1961.

\bibitem{BK}
A.~Boritchev and S.~B. Kuksin.
\newblock {\em One-dimensional turbulence and the stochastic {B}urgers
  equation}.
\newblock Mathematical Surveys and Monographs. AMS Publications, Providence,
  2021.

\bibitem{Burd}
V.~Burd.
\newblock {\em The Averaging Method on Infinite Time-Interval and Some Problems
  of the Theory of Oscillations}.
\newblock Yaroslavl University Press, Yaroslavl, 2013.
\newblock (in Russian).

\bibitem{DW}
J.~Duan and W.~Wang.
\newblock {\em Effctive Dynamics of Stochastic Partial Differential Equations}.
\newblock Elsevier, 2014.

\bibitem{Dud}
R.~M. Dudley.
\newblock {\em {Real Analysis and Probability}}.
\newblock Cambridge studies in advanced mathematics. Cambridge University
  Press, 2002.

\bibitem{AD}
A.~Dymov.
\newblock Nonequilibrium statistical mechanics of weakly stochastically
  perturbed system of oscillators.
\newblock {\em Ann. Henri Poincare}, 17:1825--1882, 2016.

\bibitem{FW}
M.~Freidlin and A.~Wentzell.
\newblock {\em Random {P}erturbations of {D}ynamical {S}ystems}.
\newblock Springer-Verlag, New York, 2nd edition, 1998.

\bibitem{HGK22}
G.~Huang and S.~Kuksin.
\newblock On averaging and mixing for stochastic {PDE}s.
\newblock {\em J. Dyn. Diff. Eq.}, 2022.
\newblock doi 10.1007/s10884-022-10202-w.

\bibitem{HKM}
G.~Huang, S.~B. Kuksin, and A.~Maiocchi.
\newblock {Time-averaging for weakly nonlinear CGL equations with arbitrary
  potential}.
\newblock {\em Fields Inst. Comm.}, 75:323--349, 2015.

\bibitem{JKW}
W.~Jian, S.B. Kuksin, and Y.~Wu.
\newblock {The Bogolyubov-Krylov averaging}.
\newblock {\em Russ. Math. Surveys}, 75:427--444, 2018.

\bibitem{brownianbook}
I.~Karatzas and S.~E. Shreve.
\newblock {\em Brownian Motion and Stochastic Calculus}.
\newblock Springer, 2002.

\bibitem{Khas66}
R.~Khasminski.
\newblock On stochastic processes defined by differential equations with a
  small parameter.
\newblock {\em Th. Probab. Appl.}, 11:211--228, 1966.

\bibitem{Khas68}
R.~Khasminski.
\newblock On the averaging principle for {I}to stochastic differential
  equations.
\newblock {\em Kybernetika}, 4:260--279, 1968.
\newblock (in Russian).

\bibitem{khasminskii}
R.~Khasminskii.
\newblock {\em Stochastic Stability of Differential Equations}.
\newblock Springer, second edition, 2012.

\bibitem{Kif}
Yu. Kifer.
\newblock {\em Large Deviations and Adiabatic Transitions for Dynamical Systems
  and Markov Processes in Fully Coupled Averaging}, volume 944 of {\em Memoirs
  of the American Mathematical Society}.
\newblock AMS Publications, Providence, 2009.

\bibitem{KM}
S.~B. Kuksin and A.~Maiocchi.
\newblock { Resonant averaging for small-amplitude solutions of stochastic
  nonlinear Schr\"odinger equations}.
\newblock {\em Proc.~A of the Royal Soc. Edinburgh}, 147:357--394, 2017.

\bibitem{Kulik}
A.~Kulik.
\newblock {\em Ergodic Behaviour of Markov Processes}.
\newblock De Gruyter, Berlin, 2018.

\bibitem{Krs}
S.-J. Liu and M.~Krstic.
\newblock {\em Stochastic Averaging and Stochastic Extremum Seeking}.
\newblock Springer, 2012.

\bibitem{mattingly}
J.~Mattingly, A.~Stuart, and D.~Higham.
\newblock E{rgodicity for SDEs and approximations: locally Lipschitz vector
  fields and degenerate noise}.
\newblock {\em Stochastic processes and their applications}, 101(2):185--232,
  2002.

\bibitem{Skor}
A.~V. Skorokhod.
\newblock {\em Asymptotic Methods in the Theory of Stochastic Differential
  Equations}, volume~78 of {\em Translations of Mathematical Monographs}.
\newblock AMS Publications, Providence, 1989.

\bibitem{SV}
D.~Stroock and S.~R.~S. Varadhan.
\newblock {\em Multidimensional Diffusion Processes}.
\newblock Springer, 1979.

\bibitem{Ver}
A.~Yu. Veretennikov.
\newblock On the averaging principle for systems of stochastic differential
  equations.
\newblock {\em Math. USSR Sbornik}, 69:271--284, 1991.

\bibitem{Vil}
C.~Villani.
\newblock {\em {Optimal Transport}}.
\newblock Springer, 2009.

\bibitem{Whit}
H.~Whitney.
\newblock Differentiable even functions.
\newblock {\em Duke Math. J.}, 10:159--160, 1943.
\newblock Also see in H. Whitney, Collected Papers vol.~1, Birkh\"auser, 1992,
  pp. 309-310.

\end{thebibliography}
\bibliographystyle{plain}

\end{document}